\documentclass[12pt]{amsart}
\usepackage{amsmath,amssymb,amsthm}
\usepackage{xypic}
\begin{document}
\bibliographystyle{alpha}
\theoremstyle{plain}
\newtheorem{tech*}{Key Idea}
\newtheorem{thm}{Theorem}[section]
\newtheorem*{thm*}{Theorem}
\newtheorem{prop}[thm]{Proposition}
\newtheorem*{prop*}{Proposition}
\newtheorem{lemma}[thm]{Lemma}
\newtheorem{cor}[thm]{Corollary}
\newtheorem*{conj*}{Conjecture}
\newtheorem*{cor*}{Corollary}
\newtheorem{defn}[thm]{Definition}
\newtheorem{eg}[thm]{Example}
\newtheorem{egs}[thm]{Examples}
\theoremstyle{definition}
\newtheorem*{defn*}{Definition}
\newtheorem{rems}[thm]{Remarks}
\newtheorem{rem}[thm]{Remark}
\newtheorem*{rem*}{Remark}
\newtheorem*{rems*}{Remarks}
\newtheorem*{proof*}{Proof}
\newtheorem*{not*}{Notation}
\newtheorem*{eg*}{Example}
\newtheorem*{egs*}{Examples}
\newtheorem*{dis*}{Discussion}
\setlength{\parskip}{.3cm}
\newcommand{\nc}{\newcommand}
\nc{\nt}{\newtheorem}
\nc{\gf}[2]{\genfrac{}{}{0pt}{}{#1}{#2}}
\nc{\mb}[1]{{\mbox{$ #1 $}}}
\nc{\real}{{\bf R}}
\nc{\comp}{{\bf C}}
\nc{\ints}{{\bf Z}}
\nc{\Ltoo}{\mb{L^2({\mathbf H})}}
\nc{\rtoo}{\mb{{\mathbf R}^2}}
\nc{\bra}{\langle}
\nc{\ket}{\rangle}
\nc{\cal}{\mathcal}
\nc{\frk}{\mathfrak}

\title{}

\titlepage

\begin{center}{\bf CENTRE-VALUED INDEX FOR TOEPLITZ OPERATORS WITH NONCOMMUTING
SYMBOLS}\\

{\bf by}\\
\vspace{.1in}
{\bf John Phillips}\\Department of Mathematics and Statistics\\
University of Victoria\\Victoria, B.C. V8W 3P4, CANADA\\

{\bf and}\\

{\bf Iain Raeburn}\\Department of Mathematics and Statistics\\University of Otago, PO Box 56,\\
Dunedin 9054, NEW ZEALAND\\

This research was supported by the Natural Sciences and Engineering Research Council of Canada, The Australian Research Council, and the University of Otago.
\end{center}
{\bf Abstract.} We formulate and prove a ``winding number'' index theorem for certain ``Toeplitz'' operators in the same spirit as the Gohberg-Krein Theorem and generalizing previous work of Lesch and others. The ``number'' in ``winding number'' is replaced by a self-adjoint operator in a subalgebra $Z\subseteq Z(A)$
of a unital $C^*$-algebra, $A$. We assume that there is a faithful $Z$-valued trace $\tau$ on $A$ which is left invariant under an action $\alpha:{\bf R}\to Aut(A)$ which leaves $Z$ pointwise fixed. If $\delta$ is the infinitesimal generator of $\alpha$ and $u$ is an invertible element in ${\rm dom}(\delta)$ then the
``winding operator'' of $u$ is $\frac{1}{2\pi i}\tau(\delta(u)u^{-1})\in Z_{sa}.$ By a careful choice of representations we can extend the data $(A,Z,\tau,\alpha)$ to a von Neumann setting 
$(\frk{A},\frk{Z},\bar\tau,\bar\alpha)$ where $\frk{A}=A^{\prime\prime}$ and $\frk{Z}=Z^{\prime\prime}.$
Then, $A\subset\frk{A}\subset \frk{A}\rtimes{\bf R}$, the von Neumann crossed product, and there is a faithful, dual $\frk{Z}$-trace on $\frk{A}\rtimes{\bf R}$. If $P$ is the projection in $\frk{A}\rtimes{\bf R}$
corresponding to the non-negative spectrum of the generator of the representation of $\bf R$ in $\frk{A}\rtimes{\bf R}$ and $\tilde\pi:A\to\frk{A}\rtimes{\bf R}$ is the embedding then we define for $u\in A^{-1}$, $T_u=P\tilde\pi(u) P$
and show that it is Fredholm in an appropriate sense and the $\frk{Z}$-valued index of $T_u$ is the negative of the winding operator, i.e.,
$\frac{-1}{2\pi i}\tau(\delta(u)u^{-1})\in Z_{sa}.$
In outline the proof follows the proof of the scalar case done previously by the authors. The difficulties arise in making sense of the various constructions when the scalars are replaced by $\frk{Z}$ in the von Neumann setting. In particular, the construction of the dual $\frk{Z}$-trace on $\frk{A}\rtimes{\bf R}$ required the nontrivial development of a $\frk{Z}$-Hilbert Algebra theory. We show that certain of these Fredholm operators fiber as a ``section'' of Fredholm operators with scalar-valued index and the centre-valued index fibers as a section of the scalar-valued indices. 

\section{WINDING OPERATOR}
\noindent{\bf Objects of Study:} We consider a unital $C^*$-algebra, 
$A$ with a unital
$C^*$-subalgebra $Z$ of the centre of $A$; $Z(A)$. We also assume that
there exists a faithful, unital, tracial, conditional expectation $\tau: A
\to Z$ (a ``faithful $Z$-trace'') and a continuous action 
$\alpha:\real \to Aut(A)$
which leaves $\tau$ invariant. That is , $\tau\circ\alpha_t = \tau$ for all
$t \in \real.$ That is, our Objects of Study are $4$-tuples $(A,Z,\tau,\alpha)$ satisfying these conditions.

Under these hypotheses we show that the {\it ``winding number theorem''}
of \cite {PhR} holds. We will often refer to this as a ``winding operator''.

\begin{thm}\label{wind}
Let $(A,Z,\tau,\alpha)$ be a $4$-tuple; so that $A$ is a unital $C^*$-algebra and $Z\subseteq Z(A)$
is a unital $C^*$-subalgebra of the centre of $A;$ 
$\tau: A \to Z$ is a faithful, unital, tracial, conditional expectation; and
$\alpha:\real \to Aut(A)$ is a continuous action leaving $\tau$
invariant. Let $\delta$ be the infinitesimal generator of $\alpha.$
Then, $$a\mapsto \frac{1}{2 \pi i}\tau(\delta(a)a^{-1}):
dom(\delta)^{-1} \to Z_{sa}$$
is a group homomorphism which is constant on connected components
and so extends uniquely to a group homomorphism $A^{-1} \to Z_{sa}$
which is constant on connected components and is $0$ on $Z^{-1}.$ We denote this map by $wind_\alpha(a).$
\end{thm}

\begin{proof}[\bf Proof]
It is an easy calculation to see that $a\mapsto \tau(\delta(a)a^{-1}): dom(\delta)^{-1}\to (Z,+)$
is a homomorphism.
We next calculate that $\alpha_t(z) = z$ for all $z \in Z$ and $t \in \real:$

$$\tau((\alpha_t(z)-z)^*(\alpha_t(z)-z))=\cdots=\tau(z^*z)-\tau(z^*)z-z^*\tau(z)+\tau(z^*z)=0.$$

Therefore, $\alpha_t(z)-z=0$ since $\tau$ is faithful. So, $Z\subseteq dom(
\delta)$ and $\delta(Z) = \{0\}.$ But then for each $z\in Z^{-1}$ we
have $\tau(\delta(z)z^{-1}) = 0.$

Now, for any $a\in dom(\delta)$, we have 
$$\tau(\delta(a))=\tau\left(\lim_{h\to 0}\frac{\alpha_h(a)-a}{h}\right)=
\lim_{h\to 0}\frac{1}{h}\tau(\alpha_h(a)-a)=0.$$ 
Hence, by the Leibnitz rule, for each $n\geq 1$
\begin{eqnarray}
0 &=& \tau(\delta(a^n)) = \tau\left(\sum_{k=0}^{n-1}a^k\delta(a)a^{(n-1)-k}
\right)\nonumber\\
&=& \sum_{k=0}^{n-1}\tau(a^k\delta(a)a^{(n-1)-k}) = \sum_{k=0}^{n-1}
\tau(a^{n-1}\delta(a))\nonumber\\
&=& n\tau(a^{n-1}\delta(a)).\nonumber
\end{eqnarray}
Thus, for each $a\in dom(\delta)$ and each $k\geq 0$ we have 
$\tau(\delta(a)a^k)=\tau(a^k\delta(a))=0.$

Now, if $a\in dom(\delta)$ and $\|1-a\| < 1$ then $a$ is invertible and
$a^{-1}=\sum_{k=0}^{\infty}(1-a)^k$ which converges in norm. Since 
$\delta(1)=0$ we have:
$$\tau(\delta(a)a^{-1})=-\tau(\delta(1-a)a^{-1})=-\tau\left(\delta(1-a)
\sum_{k=0}^{\infty}(1-a)^k\right)=-\sum_{k=0}^{\infty}
\tau(\delta(1-a)(1-a)^k)=0.$$

To see that the map is constant on connected components, we use the previous 
paragraph to show that it is locally constant. So we fix 
$a\in dom(\delta)^{-1}$ and suppose $b\in dom(\delta)^{-1}$ where 
$\|b-a\| < 1/\|a^{-1}\|.$ Then, $\|ba^{-1}-1\|\leq \|b-a\|\;\|a^{-1}\|<1$
so that 
$$0=\tau(\delta(ba^{-1})(ba^{-1})^{-1})=\tau(\delta(b)b^{-1}) +
\tau(\delta(a^{-1})a) = \tau(\delta(b)b^{-1}) - \tau(\delta(a)a^{-1})$$
as required.

Finally, to see that $\tau(\delta(a)a^{-1}) \in iZ_{sa}$, we observe that since
$dom(\delta)$ is a $*$-subalgebra of $A$ that $a \in dom(\delta)^{-1}$
implies that $a^*a \in dom(\delta)^{-1}$ and so $t\mapsto t1+(1-t)a^*a$
is a path of invertible elements in $dom(\delta)^{-1}$ connecting $1$ to
$a^*a.$ Hence, $\tau(\delta(a^*a)(a^*a)^{-1})=\tau(\delta(1)1)=0.$
Since the map is a homomorphism, this implies that 
$\tau(\delta(a^*)(a^*)^{-1})=-\tau(\delta(a)a^{-1}).$ But, then:
$$[\tau(\delta(a)a^{-1})]^* = \tau((a^*)^{-1}\delta(a^*)) = 
\tau(\delta(a^*)(a^*)^{-1}) = -\tau(\delta(a)a^{-1})$$ as required.

Since $dom(\delta)$ is a dense $*$-subalgebra of $A$ and 
$A^{-1}$ is open, $dom(\delta)^{-1}$ is dense in $A^{-1}$
and so the map extends uniquely to $A^{-1}.$

\end{proof}
\begin{defn}\label{morph} {\bf (Morphism)}
For $i=1,2$ let $(A_i,Z_i,\tau_i,\alpha^i)$ be two such $4$-tuples where $A_i$ is a unital $C^*$-algebra and $Z_i$ is a unital $C^*$-subalgebras of the centre of $A_i$, etc. A {\bf morphism} from
$(A_1,Z_1,\tau_1,\alpha^1)$ to $(A_2,Z_2,\tau_2,\alpha^2)$ is given by a unital
$*$-homomorphism $\varphi:A_1\to A_2$ which maps $Z_1\to Z_2$ and makes all the appropriate diagrams commute:

$$\xymatrix{A_1 \ar[d]_{\tau_1} \ar[r]^{\varphi} & A_2 \ar[d]^{\tau_2}\\
 Z_1  \ar[r]_{\varphi}& Z_2 }\hspace*{2in}\xymatrix{A_1 \ar[d]_{\alpha_t^1} \ar[r]^{\varphi} & A_2 \ar[d]^{\alpha_t^2}\\
 A_1  \ar[r]_{\varphi}& A_2 }$$
\end{defn}

\begin{prop}\label{wind2}
If $\varphi: A_1\to A_2$ defines a morphism from $(A_1,Z_1,\tau_1,\alpha^1)$ to $(A_2,Z_2,\tau_2,\alpha^2)$
and if $a\in A_1^{-1}\cap(dom(\delta_1))$ then $\varphi(a)\in A_2^{-1}\cap(dom(\delta_2))$  and
$$wind_{\alpha^1}(a)\in (Z_1)_{sa}\;\; {\rm while}\;\; wind_{\alpha^2}(\varphi(a))\in (Z_2)_{sa}\;\;{\rm and\;\; also}$$
 $$\varphi(wind_{\alpha^1}(a))=wind_{\alpha^2}(\varphi(a)).$$
\end{prop}
\begin{proof}
We first show that $a\in dom(\delta_1)$ implies that $\varphi(a)\in dom(\delta_2)$
and that $\varphi(\delta_1(a))=\delta_2(\varphi(a)).$
So if $a\in dom(\delta_1)$ then
$$\varphi(\delta_1(a))=\varphi\left(\lim_{t\to 0}\frac{\alpha_t^1(a)-a}{t}\right)
=\lim_{t\to 0}\varphi\left(\frac{\alpha_t^1(a)-a}{t}\right)=\lim_{t\to 0}\frac{\alpha_t^2(\varphi(a))-\varphi(a)}{t}.$$
So the right hand limit exists and defines $\delta_2(\varphi(a)).$ That is  $\varphi(\delta_1(a))=\delta_2(\varphi(a)).$
Now for $a\in A_1^{-1}\cap(dom(\delta_1))$:
\begin{eqnarray*}\varphi(wind_{\alpha^1}(a))&=&\frac{1}{2\pi i}\varphi(\tau_1(\delta_1(a)a^{-1}))=
\frac{1}{2\pi i}\tau_2(\varphi(\delta_1(a)a^{-1}))\\
&=&\frac{1}{2\pi i}\tau_2(\varphi(\delta_1(a))\varphi(a)^{-1}))=\frac{1}{2\pi i}\tau_2(\delta_2(\varphi(a))\varphi(a)^{-1}))=
wind_{\alpha^2}(\varphi(a)).
\end{eqnarray*}

\end{proof}
\section{EXTENSION to an ENVELOPING von NEUMANN ALGEBRA}

\begin{tech*} Since the range of our 
$C^*$-algebra trace, $Z$ (an abelian $C^*$-algebra), is no 
longer restricted to being the scalars, the index of our generalized
Toeplitz operators will not be scalar-valued either, but will necessarily
take values in an abelian {\bf von Neumann} algebra, say $\frk Z$, containing $Z$. Unless,
$Z$ is finite-dimensional (a relatively trivial extension of the scalar-valued
trace) we will generally have $Z\neq\frk Z$ (if $Z$ is separable but 
not finite-dimensional we must have $Z\neq\frk Z$).
\end{tech*}

\noindent We want our unital $C^*$-algebra, $A$, to be concretely represented on a Hilbert space, $\mathcal H$ in such a way that the following {\bf nontrivial} conditions hold. Let 
$\frk A = A^{\prime\prime}$ and $\frk Z = Z^{\prime\prime}$.
 
(1) There exists a necessarily unique faithful, tracial, uw-continuous 
conditional expectation, $\bar\tau:\frk A \to \frk Z$ extending $\tau.$ We will
refer to this as a $\frk Z$-trace.

(2) The continuous action $\alpha:\real \to Aut(A)$ which leaves 
$\tau$ invariant extends to an ultraweakly continuous action 
$\bar\alpha :\real \to
Aut(\mathfrak A)$ which leaves $\bar\tau$ invariant.\\

To achieve this we will {\bf assume} that $Z$ has a faithful state, $\omega$ (this is automatically true if $Z$ is separable). We will use the following Proposition to define a concrete representation where these conditions obtain. We emphasize that the extension depends on the choice of the faithful state on $Z$. However, our notation will not show the dependence on this state. Of course if $Z=\bf{C}$ the state is unique.
If $\varphi$ is a morphism from $(A_1,Z_1,\tau_1,\alpha^1)$ to $(A_2,Z_2,\tau_2,\alpha^2)$, we will assume that
$\varphi$ carries the faithful state $\omega_1$ on $Z_1$ to $\omega_2$ on $Z_2:$ that is $\omega_1 = \omega_2\circ\varphi$ restricted to $Z_1.$

\begin{prop}\label{extension}
Let $(A,Z,\tau,\alpha)$ be a $4$-tuple
and let $\omega$ be a faithful state on $Z$. 
Then $\bar\omega:= \omega\circ\tau$ is a faithful
tracial state on $A$ which is left invariant by the action $\alpha$.
If we let $(\pi,\mathcal H, \xi_0)$ be the GNS representation of  $A$ afforded by  $\bar\omega,$ with
cyclic separating trace vector $\xi_0$, then there is a continuous
unitary representation $\{U_t\}$ of $\real$ on $\mathcal H$ so that 
$(\pi,U)$ is
covariant for $\alpha$ on $A$. Then $\{U_t\}$ implements an
uw-continuous extension of $\alpha$ to $\bar\alpha$ acting on 
$\frk A=\pi(A)^{\prime\prime}.$ Morover, letting $\frk Z=\pi(Z)^{\prime\prime},$
there exists a unique faithful unital, uw-continuous 
$\frk Z$-trace $\bar\tau:\frk A \to \frk Z$ extending $\tau,$ 
and $\bar\alpha$ leaves $\bar\tau$ invariant.
\end{prop}

\begin{proof}[\bf Proof] 
Denoting the image of $a\in A$ in $\mathcal H_{\bar\omega}:=\mathcal H$
by $\hat a$, it is completely standard that 
$U_t(\hat a):=\widehat{\alpha_t(a)}$ defines a continuous unitary 
representation of $\real$ on $A$ so that $(\pi,U)$ is covariant for
$\alpha.$ Hence, $\{U_t\}$ implements an
uw-continuous extension of $\alpha$ to $\bar\alpha$ acting on 
$\frk A=\pi(A)^{\prime\prime}.$ It is also standard that the cyclic and
separating vector $\xi_0=\hat 1$ gives an extension of the trace
$\bar\omega$ to a faithful uw-continuous trace on $\frk A$. By an 
abuse of notation we will drop the notation ``$\pi$'' for the representation
of $A$ and just assume that $A$ acts directly on $\mathcal H.$ In this way
we will also write the extended scalar trace (given by $\xi_0$) on
$\frk A$ as $\bar\omega.$

With this notation, we invoke \cite{U} to obtain an uw-continuous
conditional expectation $E:\frk A\to\frk Z$ defined by the equation:
$$\bar\omega(E(x)y)=\bar\omega(xy)\;for\;x\in\frk A,\;y\in\frk Z.$$
For $x=a\in A$ and $y=z\in Z$, we have:
$$\bar\omega(\tau(a)z)=\omega(\tau(\tau(a)z))=\omega(\tau(a)z)=\omega(\tau(az))
=\bar\omega(az).$$
Since $Z$ is uw-dense in $\frk Z$ we can replace the $z\in Z$ by any
$y\in\frk Z$ in the previous equation.
That is, for $a\in A$ we have $E(a)=\tau(a)$ and so $E$ is just an extension
of $\tau$ by uw-continuity. We now use the notation $\bar\tau$ in place of $E$,
and observe that since $\tau$ is tracial, so is $\bar\tau.$ To see that
$\bar\tau$ is faithful, suppose $x\in\frk A$ and $\bar\tau(x^*x)=0.$
Then, by the defining equation for $\bar\tau$ we have
$$0=\bar\omega(\bar\tau(x^*x)1)=\bar\omega(x^*x),$$
and since $\bar\omega$ is faithful, $x=0.$

Finally to see that $\bar\alpha$ leaves $\bar\tau$ invariant, we let
$x\in\frk A$ and $t\in\real$. Choose a bounded net $\{a_i\}$ in $A$
which converges to $x$ ultraweakly. Then since $\bar\alpha_t$ is spatial, we
have $\alpha_t(a_i)=\bar\alpha_t(a_i)\to\bar\alpha_t(x)$ ultraweakly. Hence,\\

\hspace{.5in}$\bar\tau(\bar\alpha_t(x))=\lim_i \bar\tau(\alpha_t(a_i))=
\lim_i\tau(\alpha_t(a_i))=\lim_i\tau(a_i)=\lim_i\bar\tau(a_i)=\bar\tau(x).$
\end{proof}

\begin{egs*} {\bf $4$-tuples}\hspace*{4in}

{\bf 1. Kronecker (scalar trace) Example.} Let $A=C({\bf T}^2)$, the $C^{*}$-algebra of 
continuous functions on
the $2$-torus, with the usual scalar trace $\tau_0$ given by integration against
the Haar measure on ${\bf T}^2$.
We let $\alpha^\mu :{\bf R}\to\hbox{Aut}(A)$ be the Kronecker flow on $A$ 
determined by the
real number, $\mu$ (note that $\mu$ is not a power merely a superscript). That is, for $s\in {\bf R}$, $f\in A$, and $
(z,w)\in {\bf T}^2$ we have:
$$(\alpha^\mu_s\,(f))(z,w)=f\left(e^{-2\pi is}\,z,\,e^{-2\pi i\mu
s}\,w\right).$$

In terms of the two commuting unitaries which generate $A=C({\bf T}^2)$, namely $U(z,w)=z$ and $V(z,w)=w$ we have 
$$\alpha^\mu_s(U)=e^{-2\pi is}U,\, \alpha^\mu_s(V)=e^{-2\pi is\mu}V.$$
Of course, this action leaves our scalar trace $\tau_0$ invariant. In this case where $Z={\bf C}$ the faithful state $\omega$ on $Z={\bf C}$ is just the identity mapping and so $\bar\omega:= \omega\circ\tau_0=\tau_0.$ That is, $\mathcal{H}_{\bar\omega}=\mathcal{H}_{\tau_0}=L^2({\bf T}^2)$ with the obvious representation of $A$ on $\mathcal{H}_{\tau_0}$.
In this case, $Z=\frk Z = {\bf C}$ and so $\frk{A}=L^\infty({\bf T}^2).$ 
Clearly $\tau_0$ and $\alpha$ extend to $\bar\tau_0$ and $\bar\alpha$ as required.\\

One easily calculates the winding numbers of the generators:
$$wind_{\alpha^\mu}(U)= -1\;\;\;\rm{and}\;\;\;wind_{\alpha^\mu}(V) = -\mu.$$

{\bf 1.a. Noncommutative Tori.} We quickly observe that the previous construction can be carried over almost verbatim to noncommutative tori. For $\theta\in [0,1)$ let
$$A_\theta = C^*(U,V\,|\,VU=e^{2\pi i\theta}UV)$$
be the universal $C^*$-algebra generated by two unitaries, $U,V$ satisfying the above relation. For $\theta = 0$ the algebra $A_\theta$ is naturally isomorphic to $A=C({\bf T}^2)$ with $U(z,w)=z$ and $V(z,w)=w.$ For $\theta$ irrational, these algebras are of course the irrational rotation algebras which are simple $C^*$-algebras. We let $\alpha^\mu :{\bf R}\to\hbox{Aut}(A)$ be the flow on $A_\theta$ 
determined by the real number, $\mu$. That is, for $s\in {\bf R}$, and $U,V$ the generators of $A_\theta$ we have:
$$\alpha^\mu_s(U)=e^{-2\pi is}U,\, \alpha^\mu_s(V)=e^{-2\pi is\mu}V.$$
Since $\alpha_s(U)$ and $\alpha_s(V)$ satisfy the same relation as $U$ and $V$ this is a well-defined flow on 
$A_\theta.$ 

The scalar trace, $\tau_\theta$ on $A_\theta$ on the dense subalgebra of finite linear combinations of $U^n V^m$
for $m,n$ in $\bf Z$ satisfies:
\[ \tau_\theta(U^n V^m )=\left\{\begin{array}{ll} 0 & \mbox{if $n\neq 0$ or $m\neq 0$}\\
                                           1 & \mbox {if $n=0=m.$}
                                           \end{array} \right.  \]
                                           
Again, one easily calculates the winding numbers of the generators:
$$wind_{\alpha^\mu}(U)= -1\;\;\;\rm{and}\;\;\;wind_{\alpha^\mu}(V) = -\mu.$$                                           
                                           
{\bf 2. Generalized Kronecker and Generalized Noncommutative tori Examples.} We show that any self-adjoint element in any unital commutative $C^*$-algebra (with a faithful state) can be used as a replacement for the scalar $\mu$ in Examples 1 and 1.a to obtain a non-scalar example. Let $Z=C(X)$ be any commutative unital $C^{*}$-algebra 
with a faithful state and let $\eta\in Z_{sa}$ be any self-adjoint 
element in 
$Z$. Let $A=Z\otimes C({\bf T}^2)=C(X,C({\bf T}^2))$ (respectively , $A=Z\otimes A_\theta)=C(X,A_\theta)$) and let $\tau :A\to Z$ be given by the
``slice-map'' $\tau=id_Z\otimes\tau_\theta$ where $\tau_\theta$ for $\theta=0$ is the standard 
trace on $C({\bf T}^2)$ given by Haar measure (respectively, the usual scalar trace $\tau_\theta$ on $A_\theta$ defined above).  Then, $\tau$ is a faithful,
tracial conditional expectation of $A$ onto $Z$. In particular, for $f\in A=Z\otimes C({\bf T}^2)\cong C({\bf T}^2,Z)$ we have
$$\tau(f)=\int_{\bf{T}^2}f(z,w)d(z,w)\in Z.$$ In this case we note that for $A=Z\otimes C({\bf T}^2)$, we have
 $Z(A)=A$ and hence
$Z$ is strictly contained in $Z(A).$ On the other hand, for $\theta$ irrational, $Z(A)=Z(Z\otimes A_\theta)=Z$ since $A_\theta$ is simple.
 In either case we use the element $\eta\in 
Z_{sa}$ to
define a $\tau$-invariant action $\{\alpha^\eta_t\}$ of $\real$ on $A$:
$$\alpha^\eta_t(f)(x)=\alpha^{\eta(x)}_t(f(x)),$$
for $f\in A$, $t\in\real$, $x\in X$ (again, $\eta$ and $\eta(x)$ are not powers, but merely superscripts). It is clear that $(A,Z,\tau,\alpha^\eta)$ is a $4$-tuple.

In both these cases one calculates the following winding operators:
$$wind_{\alpha^\eta}(1\otimes U)= -1\otimes 1\;\;\;\rm{and}\;\;\;wind_{\alpha^\eta}(1\otimes V) = -\eta\otimes 1.$$

Using the faithful state $\omega$ on $Z$, we define a
faithful (tracial) state $\bar\omega$ on $A$ via 
$\bar\omega:=\omega\circ\tau.$ 
By Proposition \ref{extension}, $\bar\omega$ is a faithful (tracial) state on $A$ which is left invariant 
by $\alpha$ and if $(\pi,\mathcal H)$ is the GNS representation of $A$
induced by $\bar\omega$ then there is a continuous
unitary representation $\{U_t\}$ of $\real$ on $\mathcal H$ so that 
$(\pi,U)$ is covariant for $\alpha$ on $A$. Also, $\{U_t\}$ implements an
uw-continuous extension of $\alpha$ to $\bar\alpha$ acting on 
$\frk A:=\pi(A)^{\prime\prime}.$ Morover, 
letting $\frk Z:=\pi(Z)^{\prime\prime},$
there exists a unique faithful unital, uw-continuous 
$\frk Z$-trace $\bar\tau:\frk A \to \frk Z$ extending $\tau,$ 
and $\bar\alpha$ leaves $\bar\tau$ invariant.\\

{\bf 3. $C^*$-algebra of the Integer Heisenberg group}\\
Let $A$ be the $C^*$-algebra $C^*(H)$ of the integer Heisenberg group, $H$:

$$H=\left\{\left[\begin{array}{ccc}
1 & m & p\\
0 & 1 & n\\
0 & 0 & 1 
\end{array}\right]\;\; |\;\; m,n,p\in {\bf Z}\right\}.$$
We view $A=C^*(H)$ as the universal $C^*$-algebra generated by three unitaries $U, V, W$
satisfying:
$$WU=UW,\;\;\;WV=VW,\;\;\;and\;\;\;UV=WVU.$$
Here $U,V,W$ correspond respectively to the three generators of $H$:
$$u=\left[\begin{array}{ccc}
1 & 1 & 0\\
0 & 1 & 0\\
0 & 0 & 1 
\end{array}\right],\;\;\; v=\left[\begin{array}{ccc}
1 & 0 & 0\\
0 & 1 & 1\\
0 & 0 & 1 
\end{array}\right],\;\;\;w=\left[\begin{array}{ccc}
1 & 0 & 1\\
0 & 1 & 0\\
0 & 0 & 1 
\end{array}\right].$$
\begin{prop}
If $H$ is a discrete group with subgroup $C$,
then the map $l^1(H)\to l^1(C)$ defined by $f\mapsto f_{|_C}$ extends to faithful, conditional expectation $\tau$ from ${C_r}^*(H)\to {C_r}^*(C).$ If $C$ is the centre of $H$ then $\tau$ is also tracial. Combining $\tau$ with the canonical $*$-homomorphism: $C^*(H)\to {C_r}^*(H)$ we see that we can also view $\tau$ 
as a trace on $C^*(H).$
\end{prop}
\begin{proof}
Let $f\mapsto \pi_H(f)$ and $g\mapsto \pi_C(g)$
denote the left regular representations of $l^1(H)$ and $l^1(C)$ on $l^2(H)$ and $l^2(C)$
respectively. Then for $\eta\in l^2(C)\subseteq l^2(H)$ we have:
$$\pi_H(f)(\eta)(c)=\sum_{h\in H}f(ch^{-1})\eta(h)=\sum_{h\in C}f(ch^{-1})\eta(h)=\sum_{h\in C}f_{|_{C}}(ch^{-1})\eta(h)=\pi_C(f_{|_{C}})(\eta)(c).$$
In other words, for each $\eta\in l^2(C),$ $\pi_H(f)(\eta)=\pi_C(f_{|_{C}})(\eta)$ so that
$\pi_H(f)_{|_{l^2(C)}}=\pi_C(f_{|{C}}).$ We let $E:l^2(H)\to l^2(C)$ denote the canonical projection. then  all  $\eta\in l^2(C)$ have the form $\eta=E(\xi)$ for $\xi\in l^2(H)$ and we have
$\pi_C(f_{|_C})E(\xi)=E\pi_C(f_{|_C})E(\xi)=E\pi_H(f)E(\xi).$
 We now define $\tau(\pi_H(f))=\pi_C(f_{|_{C}}).$ To see that $\tau$ is bounded in operator norm, 
$$\|\pi_C(f_{|_{C}})\|=\|E\pi_H(f)E \|\leq \|\pi_H(f)\|.$$
Thus $\tau$ extends by continuity to $\tau:{C_r}^*(H)\to {C_r}^*(C).$
For general $x\in {C_r}^*(H)$ we have $\tau(x)= E\pi_H(x)E$ so that the extended $\tau$ is clearly completely positive, onto and has norm $1$: that is, it is a conditional expectation by Tomiyama's theorem.\\

Now for $f\in l^1(H)$ we have $\tau(\pi_H(f))=\pi_C(f_{|_{C}})$ so that, if $C$ is the centre of $H$, then in order to see that $\tau$ is tracial it suffices to see that for $f,g\in l^1(H)$ that 
$(f*g)_{|_{C}}=(g*f)_{|_{C}}.$ So for $c\in C$ we have:
$$(f*g)(c)=\sum_{h\in H} f(ch^{-1})g(h)=\sum_{h\in H} g(h)f(h^{-1}c)=(g*f)(c).$$
\end{proof}
In our example where $H$ is the Heisenberg group, its centre is $C=\{w^n\;|\;n\in {\bf Z}\}.$ In our realization
of $A=C^*(H)$ as a universal $C^*$-algebra, the centre of $A$ is $Z=C^*(W).$ Now the dense $*$-subalgebra of
$A$ generated by $U,V,W$ has as a basis all elements of the form $W^p V^n U^m $ each of which corresponds
uniquely to the group element $w^p v^n u^m=\left[\begin{array}{ccc}
1 & m & p\\
0 & 1 & n\\
0 & 0 & 1 
\end{array}\right]$ in $H.$
In this notation $\tau:A\to Z$ is given by: 
\[ \tau(W^p V^n U^m)=\left\{\begin{array}{ll} 0 & \mbox{if $n\neq 0$ or $m\neq 0$}\\
                                           W^p & \mbox {if $n=0=m.$}
                                           \end{array} \right.  \]
In order to define our action $\alpha:{\bf R}\to Aut(A),$ we first fix an element $\eta\in Z_{sa}$.
For an explicit example, we {\bf arbitrarily} choose $\eta=(\mu/3)(W+1+W^*)$ where $\mu$ is a fixed real number . For this fixed $\eta$ we define the action $\alpha$ via:
$$\alpha_t(U)=e^{-2\pi it}U;\;\;\;\alpha_t(V)=e^{-2\pi it\eta}V;\;\;\;\alpha_t(W)=W.$$
So on the basis elements we get
$$\alpha_t(W^p V^nU^m)=e^{-2\pi int\eta}e^{-2\pi imt}W^p V^nU^m=e^{-2\pi it(n\eta+m)}W^p V^nU^m.$$
One easily checks that for fixed $t$ the operators $U_t:=\alpha_t(U)$, $V_t:=\alpha_t(V)$, and $W_t:=W$,satisfy the same relations as $U,V,W$, namely:
$$W_tU_t=U_tW_t\; ;\;\;W_tV_t=V_tW_t\; ;\;\; U_tV_t=W_tV_tU_t.$$
Hence, $\alpha_t$ defines a $*$-representation of $H$ in $A=C^*(H)$ and so extends to a $*$-representation
of $C^*(H)$ inside $C^*(H)$. Now $W$ is in the range of this $*$-representation and so $C^*(W)$ is in the range of this $*$-representation and hence $e^{2\pi it\eta}$ is in the range of this $*$-representation
for any $t\in {\bf R}.$ Hence $V=e^{2\pi it\eta}V_t$ is in the range also. Similarly, $U$ is in the range
so that $\alpha_t(C^*(H))=C^*(H)$ since it is dense and closed. Since $\alpha_{-t}$ is the inverse of $\alpha_t$, $\alpha_t$ is one-to-one and hence an automorphism of $C^*(H)$. One easily checks that 
$\alpha_{t+s}=\alpha_t\alpha_s$ using the fact that $e^{-2\pi is\eta}$ is in the centre. The point-norm continuity of $t\mapsto \alpha_t(a)$ is clear.\\

Thus we have an action $\alpha:{\bf R}\to Aut(A),$ that fixes $Z=C^*(W)=C^*(C)$ and leaves the $Z$-valued trace 
$\tau$ invariant. That is, $(C^*(H),C^*(C),\tau,\alpha)$ is a $4$-tuple. Now the left regular representation of $C^*(C)$ on $l^2(C)$ gives a faithful vector state $\omega(x)=\langle x(\delta_1),\delta_1\rangle,$ which for $x\in l^1(C)$ is just $\omega(x)=x(1).$ Then the state $\bar\omega$ on $C^*(H)$ is given for $x\in l^1(H)$ by:
$$\bar\omega(x)=(\omega\circ\tau)(x)=\omega(x_{|_C})=x_{|_C}(1)=x(1).$$ Now if $x,y\in l^1(H)$ then the inner product induced by $\bar\omega$ is:
$$\langle x,y\rangle_{\bar\omega}=\bar\omega(x\cdot y^*)=(x\cdot y^*)(1)=\sum_{h\in H}x(1h)y^*(h^{-1})=
\sum_{h\in H}x(h)\overline{y(h)}=\langle x,y\rangle.$$
That is, $\mathcal{H}_{\bar\omega} =l^2(H)$ and the representation of $C^*(H)$ on $\mathcal{H}_{\bar\omega} =l^2(H)$ is just the left regular representation, so in this case, $\frk{A}=W^*_r(H)$
the left regular von Neumann algebra of $H.$\\
\hspace*{.2in} Now $l^2(H)=\bigoplus_X l^2(C\cdot X)$ over all the cosets
$C\cdot X$ of $C$. Moreover, each coset, $C\cdot(W^pV^nU^m)=C\cdot(V^nU^m)$ is uniquely determined by the pair of integers $(n,m)$, so that $l^2(H)=\bigoplus_{(n,m)} l^2(C\cdot V^n U^m)$. Clearly the left action of $C$ (and hence, of $C^*(C)$) on each coset space is unitarily equivalent to the left regular representation of 
$C^*(C)$ on $l^2(C).$ Hence, the left action of $C^*(C)$ on $l^2(H)$ is just a countably infinite multiple of the left regular representation of $C^*(C)$ on $l^2(C).$ That is, $\frk{Z}=1_{\bf Z^2}\otimes W_r^*(C).$\\
\hspace*{.2in} Thus the map $\tau: C^*(H)\to C^*(C)$ with {\bf both} acting on $l^2(H)$ becomes
 $\tau(x)=1_{\bf Z^2}\otimes ExE$ where $E$ is the projection from $l^2(H)$ onto $l^2(C).$ It is clear that this map is $weak-operator$ continuous and so extends by the same formula to 
a tracial expectation $\bar\tau:\frk{A}\to\frk{Z}.$ It is also clear that $\alpha$ extends to $\bar\alpha$
as needed.

In this example one calculates the following winding operators in $Z=C^*(W)$
$$wind_\alpha(U)=-1;\;\;\; wind_\alpha(V)=-\mu/3(W+1+W^*);\;\;\;wind_\alpha(W)=0.$$

\end{egs*}
\begin{egs*} {\bf Morphisms}\hspace*{4in} 

{\bf 1. Generalized Kronecker to Kronecker Morphisms.} We let $A_1=C(X)\otimes C({\bf T}^2)$ and $Z_1=C(X)\otimes 1$. We let $\tau_1=id_{C(X)}\otimes \tau_0$
where $\tau_0:C({\bf T}^2)\to {\bf C}$ is given by integration with respect to Haar measure on ${\bf T}^2.$
We arbitrarily fix a $\eta\in(Z_1)_{sa}=(C(X)\otimes 1)_{sa}$. We also define $\alpha^1:{\bf R}\to Aut(A_1)$ via:
$$\alpha_t^1(h)(x,z,w)=h(x,e^{-2\pi it}z,e^{-2\pi it\eta(x)}w).$$
As before we let $u\in A_1$ be the unitary $u(x,z,w)=w.$\\

We let $A_2=C({\bf T}^2)$ and $Z_2= {\bf C}1$ and $\tau_2=\tau_0: A_2\to Z_2.$ We arbitrarily fix an $x_0\in X$ and define the evaluation $*$-homomorphism $\varphi:A_1\to A_2$ via $\varphi(h)(z,w)=h(x_0,z,w).$ We let $\mu=\eta(x_0)$ and define 
$$\alpha_t^2(h)(z,w)=h(e^{-2\pi it}z,e^{-2\pi it\mu}w).$$
One easily checks that $\varphi$ defines a morphism from $(A_1,Z_1,\tau_1,\alpha^1)$ to $(A_2,Z_2,\tau_2,\alpha^2)$, and that $\varphi(u)=v$ where $v(z,w)=w.$\\

{\bf 1a. Generalized Noncommutative tori to Kronecker Morphisms.} We previously defined $A=C(X)\otimes A_\theta$ and $Z=C(X)\otimes 1$. We also let $\tau_1=id_{C(X)}\otimes \tau_\theta$
where $\tau_\theta:A_\theta\to {\bf C}$ is defined above.
We arbitrarily fixed an $\eta\in(Z)_{sa}=(C(X)\otimes 1)_{sa}$. And then defined $\alpha:{\bf R}\to Aut(A)$ via:
$$(\alpha_t(f))(x)=\alpha_t^{\eta(x)}(f(x))$$
for $f\in A$, $t\in\real$, $x\in X$.
We let $v\in A_1$ be the constant unitary $v(x)=V.$\\

We now consider $A_\theta$ and $Z={\bf C}1$ and $\tau_\theta: A_\theta\to Z.$ We arbitrarily fix an $x_0\in X$ and consider the action of $\real$ on $A_\theta$ defined by the real number $\eta(x_0)$, that is, $\alpha^{\eta(x_0)}$. This gives us a 4-tuple, $(A_\theta, {\bf C}, \tau_\theta,\alpha^{\eta(x_0)})$. 
We now the evaluation $*$-homomorphism $\varphi:A\to A_\theta$ via $\varphi(h)=h(x_0).$ 
One easily checks that $\varphi$ defines a morphism from $(A,Z,\tau_1,\alpha)$ to $(A_\theta,{\bf C},\tau_\theta,\alpha^{\eta(x_0)}).$
Moreover, $\varphi(v)=V.$

{\bf 2. Heisenberg to Kronecker Morphisms.} We let $A_1=C^*(H)$ and $Z_1=C^*(W)=C^*(C)\cong C^*({\bf Z})\cong C({\bf T})$ and recall
\[ \tau_1(W^p V^n U^m)=\left\{\begin{array}{ll} 0 & \mbox{if $n\neq 0$ or $m\neq 0$}\\
                                           W^p & \mbox {if $n=0=m$}
                                           \end{array} \right.  \] 
defines a trace $\tau_1:A_1\to Z_1.$ Recall that we (randomly) chose $\theta =(\mu/3)(W+1+W^*)\in (Z_1)_{sa}$ and defined our automorphism group by 
$$\alpha^1_t(W^p V^nU^m)=e^{-2\pi int\theta}e^{-2\pi imt}W^p V^nU^m=e^{-2\pi it(n\theta+m)}W^pV^nU^m.$$

We let $A_2=C^*(H/C)\cong C^*({\bf Z^2})\cong C({\bf T^2})$ where the two isomorphisms are given by:
$$Coset(W^pV^n U^m)=C\cdot (W^pV^n U^m)=C\cdot (V^n U^m) \mapsto (n,m) \mapsto z^nw^m. $$ 
We let $Z_2={\bf C}1\subset A_2$ and define $\tau_2:A_2\to Z_2={\bf C}1$ to be the composition of these isomorphisms with the trace on $C({\bf T^2})$ given by the Haar integral. This clearly implies that
\[ \tau_2(C\cdot (V^n U^m))=\left\{\begin{array}{ll}  0 & \mbox{if $n\neq 0$ or $m\neq 0$}\\
                                           1 & \mbox {if $n=0=m$}
                                           \end{array} \right.  \] 
We now define $\alpha^2_t\in Aut(A_2)$ via
$$\alpha^2_t((C\cdot V)^n (C\cdot U)^m)=e^{-2\pi itn\mu}(C\cdot V)^n e^{-2\pi itm}(C\cdot U)^m=e^{-2\pi it(n\mu+m)}(C\cdot V)^n (C\cdot U)^m.$$
Clearly, $(A_2,Z_2,\tau_2,\alpha_2)$ is isomorphic to the Kronecker example with scalar $\mu.$ 

We now define a $*$-homomorphism $\varphi:A_1=C^*(H)\to A_2=C^*(H/C)$ as the unique extension of the canonical group homomorphism $H\to H/C.$ So, 
$$\varphi(W^p V^nU^m)=(C\cdot V)^n(C\cdot U)^m\;\;{\rm in\;\;\;particular,}\;\;\;\varphi(W^p)=(C\cdot 1)=1\in H/C.$$
One easily checks that $\varphi$ defines a morphism from $(A_1,Z_1,\tau_1,\alpha^1)$ to $(A_2,Z_2,\tau_2,\alpha^2)$, and that $\varphi(W^p V^nU^m)=(C\cdot V)^n(C\cdot U)^m.$ Hence,
$\varphi(\theta)=\varphi((\mu/3)(W^{-1}+1+W))=\mu$ by our choice of $\theta.$
\end{egs*}

\section{ HILBERT ALGEBRAS OVER an ABELIAN von NEUMANN ALGEBRA}

\begin{tech*}
While centre-valued traces are well-known (eg., the Traces
Op\'{e}ratorielles of \cite{Dix}) a completely general construction of such 
traces suitable for use with crossed-products has not (to our knowledge) been
attempted before now.\\
\indent In this section we combine the theory of Hilbert modules (\cite{Pa}, \cite{R})
with the theory of Hilbert Algebras \cite{Dix} in order to construct
centre-valued traces on certain crossed product von Neumann algebras.
 Although the outline
is similar to the usual Hilbert Algebra theory, the details are rather
subtle. The main difficulties arise because the usual norm completion
of these new ``Hilbert Algebras'' is not self-dual in the sense of Paschke
\cite{Pa}.
\end{tech*}

\begin{defn}\label{Hilbmod}
Let $\mathfrak B$ be a von Neumann algebra. A complex vector space ${\bf X}$
is a (right) {\bf pre-Hilbert $\mathfrak B$-module} if there exists a
$\mathfrak B$-valued inner product $\bra\cdot,\cdot\ket$ which is linear
in the second co-ordinate satisfying:

\noindent(i) $\bra x,x\ket\geq0\;and\;\bra x,x\ket=0 \Longleftrightarrow x=0$
for each $x\in {\bf X}.$

\noindent(ii) $\bra x,y\ket^* = \bra y,x\ket$ for all $x,y \in {\bf X}.$

\noindent(iii) $\bra x,ya\ket=\bra x,y\ket a$ for all $x,y \in {\bf X}$ and
$a \in \mathfrak B.$

\noindent(iv) $span\{\bra x,y\ket\,|\,x,y \in {\bf X}\}$ is uw-dense in
$\mathfrak B.$

\end{defn}

\begin{tech*}
In the following we do {\bf not} assume that our bounded module mappings are
adjointable: as pointed out by Lance \cite{L} this yields a rather trivial
result that for Hilbert modules all such maps arise from inner products. However, 
most Hilbert modules are not self-dual: self-dual modules $Y$ have the property that
$\cal L(Y)$ is a von Neumann algebra.  In the examples that we use later, 
the Paschke dual $X^\dagger$ of a pre-Hilbert $\frk B$-module $X$ is a self-dual module
that is usually much larger than $X$. We need these
self-dual modules in order to work in the von Neumann algebra, $\cal L(X^\dagger)$.

\end{tech*}

\begin{defn}\label{dual}
We follow Paschke \cite{Pa} by defining the {\bf dual} of a pre-Hilbert 
$\mathfrak B$-module $\bf X$ to be the space:
$${\bf X}^{\dagger}=\{\theta :{\bf X} \to \mathfrak B\,|\,\theta\;is\;a\;
bounded\;
\mathfrak B{\text -}module\; map\}.$$
In order to make the embedding of $\bf X$ into ${\bf X}^{\dagger}$ linear,
Paschke defines scalar multiplication on
${\bf X}^{\dagger}$ by:
$$(\lambda\theta)(x):=\bar\lambda\theta(x)\;\;\;for\;\lambda \in \comp,\;
\theta \in {\bf X}^{\dagger},\;and\;x\in {\bf X}.$$
Similarly, module multiplication on ${\bf X}^{\dagger}$ is given by:
$$(\theta\cdot a)(x):=(a^*\theta(x))\;\;\;for\;\theta \in {\bf X}^{\dagger},\;a \in
\mathfrak B,\;and\;x \in {\bf X}.$$
\end{defn}

Therefore, we can identify $\bf X$ in ${\bf X}^{\dagger}$ via
$x \mapsto \hat x$ where $\hat x(y)=\bra x,y\ket$ for $x,y \in {\bf X}.$
Since $\mathfrak B$ is a von Neumann algebra, Paschke shows how to extend
the $\mathfrak B$-valued inner product on $\bf X$ to an inner product on
${\bf X}^{\dagger}$ so that ${\bf X}^{\dagger}$ becomes self-dual \cite{Pa} 
Theorem 3.2. This theorem is {\bf not} trivial.

We recall Paschke's construction on page 450 of \cite{Pa}.
Let $\mathfrak B_*$ be the space of
ultraweakly continuous linear functionals on $\mathfrak B$: that is,
the predual of $\mathfrak B.$ Now for each positive functional $\omega$ in
$\mathfrak B_*$ we have that for
$N_{\omega}=\{x\in {\bf X}\,|\,\omega(\bra x,x\ket)=0\}$, the space 
${\bf X}/N_{\omega}$
is a pre-Hilbert space with inner product: $\bra x+N_{\omega},y+N_{\omega}\ket
_{\omega} = \omega(\bra x,y\ket).$ Moreover, for each $\theta \in 
{\bf X}^{\dagger}$,
the mapping $x+N_{\omega}\mapsto\omega (\theta (x))$ 
is a well-defined bounded linear functional on ${\bf X}/N_{\omega}$ satisfying
$|\omega(\theta(x))|\leq \|\omega\|^{1/2}\|\theta\|\,\|x+N_{\omega}\|_\omega.$
Hence, there exists a unique vector $\theta_{\omega}$ in $\mathcal{H}_{\omega},$
the Hilbert space completion of ${\bf X}/N_{\omega},$ with
$$\omega(\theta(x))=\bra\theta_{\omega},x+N_{\omega}\ket_{\omega}\; for\; all\;
x \in {\bf X},\,and$$ 
$$\|\theta_{\omega}\|_{\omega}\leq \|\omega\|^{1/2}\|\theta\|.$$
Thus, $\|x\|_\omega := \omega (\bra x,x\ket )^{1/2}$ is a 
well-defined seminorm on $\bf X$ which extends naturally to 
${\bf X}^\dagger$ via
$\|\theta\|_\omega=\bra\theta_{\omega},\theta_{\omega}\ket_{\omega} ^{1/2}.$
Moreover, for all $\omega\in\mathfrak B_*^+ ,\;\theta\in {\bf X}^{\dagger},\; 
x\in {\bf X}$
we have: 
\begin{eqnarray*}
|\bra \theta_{\omega},x+N_{\omega}\ket_{\omega}| &\leq&
\|\theta_{\omega}\|_{\omega}\|x+N_{\omega}\|_\omega \\
&\leq& \|\omega\|^{1/2}\|\theta\|\,\|\omega\|^{1/2}\|x\|=\|\omega\|\,\|\theta\|
\,\|x\|.
\end{eqnarray*}

We recall from Proposition 3.8 of \cite{Pa} that ${\bf X}^\dagger$ is a dual
space with the $weak^*$-topology given by the linear functionals:
$$\theta \mapsto \omega(\bra \tau,\theta\ket)\;for\;\omega\in\mathfrak B_*\;
\tau\in{\bf X}^\dagger.$$

\begin{prop}\label{density}
Let $\mathfrak B$ be a von Neumann algebra and let $\bf X$ be a pre-Hilbert 
$\mathfrak B$-module. Then, 

\noindent(i) the unit ball of ${\bf X}^\dagger$ is complete
in the topology given by the family of seminorms,\\ 
\indent \;$\{\|\cdot\|_\omega\,|\,
\omega \in \mathfrak B_*^+\};$

\noindent(ii) $\bf X$ is dense in ${\bf X}^\dagger$ in this topology; and hence

\noindent(iii) $\bf X$ is $weak^*$ dense in ${\bf X}^\dagger.$

\noindent(iv)\;For each $\omega\in \frk B_*^+$,\; $\theta\in {\bf X}^\dagger$,
and $\epsilon>0$ there exists an $x\in {\bf X}$ with:
$$\|\theta - x\|_{\omega}^2 = \omega(\bra\theta -x,\theta -x\ket) 
< \epsilon^2.$$
\end{prop}

\begin{proof}[\bf Proof]
(i) Let $\{\theta^\alpha\}$ be a Cauchy net in the unit ball of 
${\bf X}^\dagger.$
Then, for a fixed $\omega \in\mathfrak B_*^+$, the net 
$\{(\theta^{\alpha})_\omega
\}$ is a Cauchy net in the norm $\|\cdot\|_\omega$ on $\mathcal{H}_\omega$ by definition.
Hence, there exists an element $\theta_\omega \in\mathcal{H}_\omega$ with
$\|(\theta^{\alpha})_{\omega} - \theta_{\omega}\| \to 0.$ Moreover,
$$\|\theta_{\omega}\| \leq \limsup_{\alpha} \|(\theta^{\alpha})_{\omega}\|
\leq \|\omega\|^{1/2}\|\theta^{\alpha}\| \leq \|\omega\|^{1/2}.$$ 

Now, for fixed $x \in {\bf X}$, $\{\theta^{\alpha}(x)\}$ is a bounded net in
$\mathfrak B.$ Moreover, for each $\omega\in \mathfrak B_*^+$
$$\lim_{\alpha} \omega(\theta^{\alpha}(x))=\lim_{\alpha}
\bra (\theta^{\alpha})_{\omega},x+N_{\omega}\ket_\omega =\bra \theta_{\omega},
x+N_{\omega}\ket_\omega .$$
Thus for $every$ $\omega \in \mathfrak B_*$, the net 
$\{\omega(\theta^{\alpha}(x))\}$ converges in $\comp.$ Clearly,
this limit is linear in $\omega$: that is, the bounded net 
$\{\theta^{\alpha}(x)\}$ of linear functionals on $\mathfrak B_*$
converges pointwise to a linear functional on $\mathfrak B_*$ which is 
therefore bounded by the same bound, $\|x\|.$ That is, the pair 
$(x,\{\theta_{\omega}\,|\,\omega\in\mathfrak B_*^+ \})$ defines
an element in $(\mathfrak B_* )^* = \mathfrak B$ via
$\omega \mapsto \bra\theta_{\omega}, x+N_{\omega}\ket_\omega .$ If we call 
this element $\theta(x)$, then by definition,
$$\omega(\theta(x))=\bra\theta_{\omega},x+N_{\omega}\ket_\omega =\lim_{\alpha}
\omega(\theta^{\alpha}(x)),$$
$$and\; \|\theta(x)\|\leq \|x\|.$$ 

By this formula, $\theta(x)$ is clearly linear in $x$, and so
$\theta :{\bf X} \to \mathfrak B$ is linear. By construction, 
$\theta^{\alpha}(x)$
converges ultraweakly to $\theta(x)$ and since each $\theta^\alpha$ is
a $\mathfrak B$-module map, so is $\theta$. Clearly, $\|\theta\|\leq 1$,
so $\theta$ is in the unit ball of ${\bf X}^\dagger ,$ 
and $\theta^\alpha$ converges to $\theta$. That is, the unit ball of
${\bf X}^\dagger$ is complete as claimed.

(ii) To see that $\bf X$ is dense in ${\bf X}^\dagger$,
fix $\theta \in {\bf X}^\dagger$ and $\epsilon > 0$. Let 
$\{\omega_1,\omega_2,...,\omega_m\}$ be a finite set of functionals 
in $\mathfrak B_*^+.$ Given this data we let $\omega=\omega_1+\cdots
+\omega_m.$ Now,
$\omega\geq\omega_i$ for each $i=1,2,...,m$ and so by Proposition 
3.1 of \cite{Pa}, the map $x+N_{\omega}
\mapsto x+N_{\omega_i}$ is a well-defined contraction which extends to 
a contraction $\mathcal{H}_{\omega} \to \mathcal{H}_{\omega_i}$ carrying $\theta_{\omega}$
to $\theta_{\omega_i}.$ We choose $x 
\in {\bf X}$ so that 
$\|(x+N_{\omega})-\theta_{\omega}\|_{\omega} < \epsilon.$
Then, for each $i=1,2,...,m$, we have:
$$\|x-\theta\|_{\omega_i}:=\|(x+N_{\omega_i})-\theta_{\omega_i}\|_{\omega_i}
\leq \|(x+N_{\omega}) -\theta_\omega \|_\omega <\epsilon.$$

(iii) Now fix $\theta\in {\bf X}^\dagger$ and let $\epsilon>0$, $\{\tau_1,...
\tau_n\}\subseteq {\bf X}^\dagger$, $\{\omega_1,...,\omega_m\}\subseteq
\mathfrak B_*$ define a basic $weak^*$-neighbourhood of $\theta.$ Since 
every element of $\mathfrak B_*$ is expressible as a linear combination of four
elements in $\mathfrak B_*^+$ we can assume that $\omega_1,...,\omega_m$
are positive. Let $\omega=\omega_1+\cdots+\omega_m$ and choose $x\in {\bf X}$
with $$\|(x+N_{\omega})-\theta_{\omega}\|_{\omega} < \frac{\epsilon}{\|\tau_1\|
+\cdots+\|\tau_n\|}.$$
Then, for each $i=1,...,m$ and $k=1,...,n$, we have:
\begin{eqnarray*}
 |\omega_i \bra\tau_k,x-\theta\ket| &=&|\bra\tau_k,x-\theta\ket_{\omega_i}|
 \leq\|\tau_k\|_{\omega_i}\|x-\theta\|_{\omega_i}\\
 &\leq &\|\tau_k\|_{\omega}\|x-\theta\|_{\omega}
 \leq\|\tau_k\|\,\|(x+N_{\omega})-\theta_{\omega}\|_\omega < \epsilon.
\end{eqnarray*}

(iv) This is just a restatement of the fact that ${\bf X}/N_\omega$ is dense
in its Hilbert space completion $\mathcal{H}_\omega$ as described above in the remarks
after Definition \ref{dual}. 
\end{proof}

\begin{rem*}
In the following class of examples we can more or less explicitly calculate $X^\dagger.$
\end{rem*}

\begin{eg}\label{tensormod}
Let $\mathcal H$ be a Hilbert space with orthonormal basis
$\{\xi_n\}$, let $\mathfrak B$ be a von Neumann algebra,
 and let $X$ be the algebraic tensor product 
$X=\mathcal H\otimes\mathfrak B$, with the obvious $\mathfrak B$-valued 
inner product. Then, $X$ is a pre-Hilbert $\mathfrak B$-module and we can
identify $X^\dagger$ as:
$$X^\dagger=\left\{\sum_n \xi_n\otimes b_n\;|\; b_n\in\mathfrak B\;and\;
\exists M>0\;with\;
\|\sum_{n\in F}b_n^*b_n\|\leq M,\; \forall\;finite\;F\right\}.$$
Such a formal sum defines a bounded $\mathfrak B$-module mapping $\theta$ on $X$
as follows:
$$\theta\left(\sum_{k=1}^N\eta_k\otimes a_k\right)=
\sum_{k=1}^N\sum_n\bra\xi_n,\eta_k\ket b_n^*a_k,$$
where the right hand side converges in norm.
\end{eg}

\begin{proof}
First, let $\theta$ denote an arbitrary element in $X^\dagger.$
Define $b_n^*:=\theta(\xi_n\otimes 1).$ Since $\theta$ is also defined on the
norm closure of $X$, we see that $\theta$ is defined on each element
of the form,
$\sum_n \xi_n\otimes a_n$ where $\sum_n a_n^*a_n$ converges in norm in
$\mathfrak B.$ In particular, if $\eta\in\mathcal H$, so that
$\eta=\sum_n\bra\xi_n,\eta\ket\xi_n$ converges in norm then, 
$\eta\otimes a=\sum_n\xi_n\otimes\bra\xi_n,\eta\ket a$ converges in norm, 
and so
$$\theta(\eta\otimes a)=\sum_n\theta(\xi_n\otimes\bra\xi_n,\eta\ket a)=
\sum_n\bra\xi_n,\eta\ket\theta(\xi_n\otimes 1)a=\sum_n\bra\xi_n,\eta\ket b_n^*a=
\sum_n\bra\xi_n,\eta\ket b_n^*a.$$
Hence for any element $x=\sum_{k=1}^N\eta_k\otimes a_k\in X$ we have
$x=\sum_{k=1}^N\sum_n\xi_n\otimes\bra\xi_n,\eta_k\ket a_k$ converges in norm
and:
$$\theta\left(\sum_{k=1}^N\eta_k\otimes a_k\right)=
\sum_{k=1}^N\theta(\eta_k\otimes a_k)=
\sum_{k=1}^N\sum_n\bra\xi_n,\eta_k\ket b_n^*a_k,$$
as claimed. To see that the $b_n$'s satisfy the boundedness condition,
let $F$ be any finite set of indices. Then,
\begin{eqnarray*}
&&\|\sum_{n\in F}b_n^*b_n\|=\|\theta(\sum_{n\in F}\xi_n\otimes b_n)\|\leq
\|\theta\|\cdot\|\sum_{n\in F}\xi_n\otimes b_n\|\\
&=&\|\theta\|\cdot\|\bra\sum_{n\in F}\xi_n\otimes b_n,
\sum_{n\in F}\xi_n\otimes b_n\ket_{\mathfrak B}\|^{1/2}=\|\theta\|\cdot
\|\sum_{n\in F}b_n^*b_n\|^{1/2}.
\end{eqnarray*}
That is, $\|\sum_{n\in F}b_n^*b_n\|^{1/2}\leq \|\theta\|$ for all finite $F$,
so we can choose $M=\|\theta\|^2.$

On the other hand if we have such a formal sum,
$\sum_n \xi_n\otimes b_n$, then we will show that the finite partial sums
$\sum_{n\in F}\xi_n\otimes b_n$
form a Cauchy net (in the family of seminorms of Prop. \ref{density})
in the ball of radius $\sqrt M$ in
$X$, and invoke the previous proposition to conclude that they converge
pointwise ultraweakly to an element in $X^\dagger$ of norm at most $\sqrt M.$

To this end let $\omega\in\mathfrak B_*^+$ and let $\epsilon>0$. Since the 
finite sums, $\left\{\sum_{n\in F}b_n^*b_n\right\}_F$ form a bounded increasing 
net of positive operators in $\mathfrak B$, they converge strongly to an 
element of $\mathfrak B$. Hence the net 
$\left\{\sum_{n\in F}\omega(b_n^*b_n)\right\}_F$ converges to a 
finite nonnegative number. Thus, there exists a large finite set $F_0$ so that 
if $F_0\cap F=\phi$ then $\sum_{F}\omega(b_n^*b_n)<\epsilon/2.$

Thus if $F_0\subset F_1$ and $F_0\subset F_2$, we have
\begin{eqnarray*}
&&\|\sum_{F_1}\xi_n\otimes b_n-\sum_{F_2}\xi_n\otimes b_n\|^2_{\omega}=
\|\sum_{F_1\sim F_2}\xi_n\otimes b_n
-\sum_{F_2\sim F_1}\xi_n\otimes b_n\|^2_{\omega}\\
&=&\omega\left(\left\bra(\sum_{F_1\sim F_2}\xi_n\otimes b_n
-\sum_{F_2\sim F_1}\xi_n\otimes b_n),
(\sum_{F_1\sim F_2}\xi_n\otimes b_n
-\sum_{F_2\sim F_1}\xi_n\otimes b_n)\right\ket_{\mathfrak B}\,\right)\\
&=&\omega\left(\sum_{F_1\sim F_2}b_n^*b_n\right) + 
\omega\left(\sum_{F_2\sim F_1}b_n^*b_n\right) < \epsilon/2 +\epsilon/2=\epsilon.
\end{eqnarray*}
Hence, the finite sums $\sum_F\xi_n\otimes b_n$ converge to an element 
$\theta\in X^\dagger:$ that is, for each $x\in X$,
$\theta(x)=uw-lim_F\bra\sum_F\xi_n\otimes b_n,x\ket.$
Now, for $x=\sum_{k=1}^N\eta_k\otimes a_k\in X$ we have by the first part of the 
proof that
$x=\sum_{k=1}^N\sum_n\xi_n\otimes\bra\xi_n,\eta_k\ket a_k$ converges in norm.
Since $\theta$ is bounded, $\theta(x)=\sum_{k=1}^N\sum_n\bra\xi_n,\eta_k\ket 
\theta(\xi_n\otimes a_k)$ also converges in norm. But then,
$$\theta(\xi_n\otimes a_k)=uw-lim_{F}\left\bra\sum_{m\in F}\xi_m\otimes b_m,
\xi_n\otimes a_k\right\ket_{\mathfrak B}=b_n^*a_k.$$
And so, indeed,  $\theta(\sum_{k=1}^N\eta_k\otimes a_k)=\sum_{k=1}^N\sum_n 
\bra\xi_n,\eta_k\ket b_n^*a_k$
converges in norm.

\end{proof}

\begin{tech*} In the definition below of a {\bf $\mathfrak Z$-Hilbert algebra},
$\cal A,$ a key idea is the use of the topology given by the seminorms
in Proposition \ref{density} to replace the norm topology on
$\mathcal{H}_{\mathcal A}:={\mathcal A}^\dagger$ when $\mathfrak Z$ is not $\comp.$\\
\indent Hence, axiom (viii) below seems to us the most natural replacement 
for the usual
axiom of the norm-density of $\cal A^2$ in $\cal A$. When we come to apply
this axiom to the crossed product examples that we construct we 
are actually able to
show that a stronger condition holds. However, in order to prove that
the algebra of bounded elements $\cal A_b$ also satisfies axiom (viii)
we need the weaker version below. Moreover, in the converse construction of 
a $\frk Z$-Hilbert
Algebra from a {\bf given} $\frk Z$-trace one also needs the weaker 
version of axiom (viii) below. 
\end{tech*}

\begin{defn}\label{Hilbalg}
Let $\mathfrak Z$ be an abelian von Neumann algebra. A complex $*$-algebra
$\mathcal A$ is called a {\bf $\mathfrak Z$-Hilbert algebra} if $\mathcal A$
is a right pre-Hilbert $\mathfrak Z$-module which satisfies the further four
axioms:

\noindent (v) $\bra a^*,b^*\ket=\bra b,a\ket\;for\;a,b\in \mathcal A.$

\noindent (vi) $\bra ab,c\ket=\bra b,a^*c\ket\;for\;a,b,c\in\mathcal A.$

\noindent (vii) $b \mapsto ab:\mathcal A \to \mathcal A$ is bounded in 
the $\mathfrak Z$-module norm for each fixed $a\in\mathcal A.$

\noindent (viii) The space $\mathcal A^2=span\{ab\,|\,a,b\in\mathcal A\}$
is dense in $\mathcal A$ in the topology given by the family of seminorms
$\{\|\cdot\|_{\omega}\,|\,\omega\in\ \mathfrak Z_*^+\},$
defined above.
\end{defn}

\begin{rem*}
It is easy to see that if $\mathcal A^2$ is norm-dense in $\mathcal A$ in the
$\mathfrak Z$-module norm,\\
 $\|a\|^2=\|\bra a,a\ket\|$ then axiom {\it (viii)} is
satisfied.
\end{rem*}

\begin{eg}\label{Hilbex}
Let $\mathfrak A$ be a von Neumann algebra and let 
$\mathfrak Z$ be a von Neumann subalgebra of the centre of $\mathfrak A.$
Suppose $\tau : \mathfrak A \to \mathfrak Z$ is a faithful, unital
uw-continuous $\frk Z$-trace. Then, for $a,b \in 
\mathfrak A$, the
following inner product makes $\mathfrak A$ into a  
$\mathfrak Z$-Hilbert algebra:
$$\bra a,b\ket_{\mathfrak Z}:=\tau(a^*b).$$
\end{eg}

\begin{proof}[\bf Proof]
The only axioms that are not completely trivial are (iii) and (vii).
Axiom (iii) follows from lemma 1.1,
while Axiom (vii) follows from the calculation:
\begin{eqnarray*}
\|ab\|^2_{\mathfrak A} &=& \|\bra ab,ab\ket_{\mathfrak Z}\|_{\mathfrak Z}
= \|\tau(b^*a^*ab)\|_{\mathfrak Z}\\
&\leq & \|\tau(\|a^*a\|_{op}b^*b)\|_{\mathfrak Z}
= \|a\|^2_{op}\|b\|^2_{\mathfrak A}.
\end{eqnarray*}
Since $\tau$ is unital, it is easy to see that $\|1\|_{\mathfrak A} = 1$
and so $\|a\|_{\mathfrak A} \leq \|a\|_{op}$ for all $a \in {\mathfrak A}.$
\end{proof}  
Of course, even if $\frk Z=\comp$ one usually has strict containment
$\frk A\subset {\frk A}^\dagger:=\mathcal H_{\mathfrak A}.$
\begin{rems*}
We denote by $\pi(a)$ the operator ``left multiplication by $a$'' and note that
by axioms (vi) and (vii) $\pi(a)$ is adjointable with adjoint $\pi(a^*)$
and hence $\pi(a)$ is a $\mathfrak Z$-module mapping. That is,
$$a(bz)=(ab)z\;for\;a,b\in\mathcal A\, ,\,z\in\mathfrak Z.$$

We denote by $\pi^{\prime}(a)$ the operator ``right multiplication by $a$''
and note that by axioms (v),(vi), and (vii) that $\pi^{\prime}(a)$ is
also bounded and adjointable with adjoint $\pi^{\prime}(a^*)$ and therefore 
is also a $\mathfrak Z$-module mapping. That is, 
$$(bz)a=(ba)z\;for\;a,b\in\mathcal A\, ,\,z\in\mathfrak Z.$$
A little playing with the axioms and using the fact that $\mathfrak Z$
is abelian yields the further useful identity:
$$(az)^*=a^*z^*\;for\;a\in \mathcal A\, ,\,z\in\mathfrak Z.$$

Whenever $\cal A$ is a $\mathfrak{Z}$-Hilbert algebra,
 we will use the suggestive notation $\mathcal{H}_{\mathcal A}$
in place of $\mathcal A^{\dagger}$
for the Paschke dual of $\mathcal A.$ That is,
$$\mathcal{H}_{\mathcal A}=\mathcal A^{\dagger}=
\{\theta :\mathcal A \to \mathfrak Z\,|\,\theta\;is\;a\;
bounded\;\mathfrak Z{\text -}module\; map\}.$$
By Theorem 3.2 of \cite{Pa}, $\mathcal{H}_{\mathcal A}$ is a self-dual Hilbert
$\mathfrak Z$-module. For $\xi\in \mathcal{H}_{\mathcal A}$ and $a\in\mathcal A$
we have $\xi(a)=\bra\xi,\hat a\ket$ where $\hat a\in \mathcal{H}_{\mathcal A}$ is
given by $\hat a(b)=\bra a,b\ket$ for $b\in\mathcal A.$ We identify
$a$ with $\hat a\in \mathcal{H}_{\mathcal A}$ so that $\mathcal A\subseteq
\mathcal{H}_{\mathcal A}$ and so, of course, ${\mathcal A}^{-}\subseteq
\mathcal{H}_{\mathcal A}.$  By Corollary 3.7 of \cite{Pa} each $\pi(a)$ (respectively, 
$\pi^{\prime}(a)$) extends 
uniquely to an element of ${\mathcal L}(\mathcal{H}_{\mathcal A})$ which we will 
also denote by
$\pi(a)$ (respectively, $\pi^{\prime}(a)$) and moreover, the map:
$$\mathcal A\stackrel{\pi}{\to}
{\mathcal L}(\mathcal{H}_{\mathcal A})$$ is a $*$-monomorphism. Similarly, the
map:
$$\mathcal A\stackrel{\pi^{\prime}}{\to}
{\mathcal L}(\mathcal{H}_{\mathcal A})$$ is a $*$-anti-monomorphism.

We note that with this notation, axiom (viii) ensures that $\mathcal A^2$ is
$weak^*$-dense in $\mathcal{H}_{\mathcal A}$ by Proposition \ref{density} part (iii).  
\end{rems*}

\begin{prop}\label{TypeI}
Let $\mathcal A$ be a $\mathfrak Z$-Hilbert algebra where $\mathfrak Z$
is an abelian von Neumann algebra.
For $z\in \mathfrak Z$ and $\xi\in \mathcal{H}_{\mathcal A}$ the mapping
$\xi \mapsto z\cdot\xi:=\xi z$ embeds $\mathfrak Z$ into ${\mathcal L}
(\mathcal{H}_{\mathcal A})$.
With this embedding we have $$\mathfrak Z = Z({\mathcal L}(\mathcal{H}_{\mathcal A})),$$
the centre of ${\mathcal L}(\mathcal{H}_{\mathcal A}).$ Moreover, ${\mathcal L}
(\mathcal{H}_{\mathcal A})$ is a Type $I$
von Neumann algebra.
\end{prop}

\begin{proof}[\bf Proof]
It is easy to check that this mapping embeds $\mathfrak Z$ into 
${\mathcal L}(\mathcal{H}_{\mathcal A})$ and since each $T\in {\mathcal L}
(\mathcal{H}_{\mathcal A})$ is
$\mathfrak Z$-linear we have that $\mathfrak Z \hookrightarrow
Z({\mathcal L}(\mathcal{H}_{\mathcal A})).$ Now by Corollary 7.10 of \cite{R},
$\mathfrak Z$ and ${\mathcal L}(\mathcal{H}_{\mathcal A})$ are Morita equivalent 
in the sense of \cite{R} and so by
Theorem 8.11 of \cite{R}, ${\mathcal L}(\mathcal{H}_{\mathcal A})$ is a Type $I$ 
von Neumann algebra.

Now by the construction of Corollary 7.10 of \cite{R}, $\mathcal{H}_{\mathcal A}$
becomes a left Hilbert ${\mathcal L}(\mathcal{H}_{\mathcal A})$-module with the 
inner product:
$$\bra\xi ,\eta\ket_{{\mathcal L}(\mathcal{H}_{\mathcal A})}(\mu)=\xi\bra\eta ,
\mu\ket_{\mathfrak Z}
\;for\;\xi,\eta,\mu\in \mathcal{H}_{\mathcal A}.$$ That is,
$\bra\xi ,\eta\ket_{{\mathcal L}(\mathcal{H}_{\mathcal A})}$ is the ``finite-rank'' 
operator
$\xi\otimes\overline{\eta}$ in ${\mathcal L}(\mathcal{H}_{\mathcal A}).$ Then, for 
$T\in Z({\mathcal L}(\mathcal{H}_{\mathcal A})),$
\begin{eqnarray*}
\bra T\xi ,\eta\ket_{{\mathcal L}(\mathcal{H}_{\mathcal A})}&=&
(T\xi)\otimes\overline{\eta}=
T(\xi\otimes\overline{\eta})\\
&=&(\xi\otimes\overline{\eta})T=\xi\otimes \overline{T^*\eta}
=\bra\xi ,T^*\eta\ket_{{\mathcal L}(\mathcal{H}_{\mathcal A})}.
\end{eqnarray*}
Thus, such a $T$ is adjointable and clearly ${\mathcal L}
(\mathcal{H}_{\mathcal A})$-linear.
By Corollary 7.10 of \cite{R}, $T$ must be of the form $T\xi=\xi z=
z\cdot\xi$ for some $z\in\mathfrak Z.$ That is, 
$\mathfrak Z = Z({\mathcal L}(\mathcal{H}_{\mathcal A})).$
\end{proof}

\begin{tech*}
The fact that $\cal L(\mathcal{H}_{\cal A})$ is a type I von Neumann algebra with
centre $\frk Z$ is one key idea which makes the theory of $\mathfrak{Z}$-Hilbert 
algebras possible. That is, if
$\mathfrak R$ is a $*$-subalgebra of ${\mathcal L}(\mathcal{H}_{\mathcal A})$ 
which contains 
$\mathfrak Z$, then $\mathfrak R$ is uw-closed
if and only if $\mathfrak R=\mathfrak R^{\prime\prime}$ where $^{\prime}$
denotes commutant {\bf within} ${\mathcal L}(\mathcal{H}_{\mathcal A})$. This follows from
compl\'{e}ment 13, III.7 of \cite{Dix} and allows us to use commutation
(pure algebra) to determine inclusion or equality of certain algebras.
\end{tech*}

\section{ COMMUTATION THEOREM for $\mathfrak Z$-HILBERT ALGEBRAS}

Throughout this section $\mathfrak Z$ is an Abelian von Neumann algebra
and $\mathcal A$ is a $\mathfrak Z$-Hilbert Algebra with Paschke dual
$\mathcal{H}_{\mathcal A}.$ Given the machinery we have developed for  $\mathfrak Z$-Hilbert Algebras, the proof of the commutation theorem below follows the outline of the classical case quite closely. 

\begin{lemma}
If $T$ is a nonzero operator in ${\mathcal L}(\mathcal{H}_{\mathcal A})$ then there 
exists $a \in \mathcal A$ with $T\pi(a) \neq 0.$
\end{lemma}

\begin{proof}[\bf Proof]
If $T(\mathcal A^2) = \{0\},$ then for all $\xi \in \mathcal{H}_{\mathcal A},$
$\bra T^*\xi,ab\ket=\bra\xi,T(ab)\ket =0.$
Hence, for each positive $\omega \in \mathfrak Z_*$ we have
$$0=\omega(\bra ab,T^*\xi\ket)=\bra ab,T^*\xi \ket_\omega.$$
Then by Definition \ref{Hilbalg} part (viii) and Proposition \ref{density} part (ii) we must
have $T^*\xi =0$ for all $\xi\in \mathcal{H}_{\mathcal A}.$ That is, $T^*=0$ and hence
$T=0.$

Therefore, there exists $a,b\in\mathcal A$ with
$$0\neq T(ab)=T(\pi(a)b)=(T\pi(a))(b),\;so\;T\pi(a)\neq 0.$$ 
\end{proof}

Since $\mathcal L (\mathcal{H}_{\mathcal A})$ is a von Neumann algebra it has a God-given 
ultraweak (uw) topology. This is the topology we refer to in the
following lemma.

\begin{lemma}\label{vNd}
With the standing assumptions of this section, we have

(i) $(\pi(\mathcal A))^{-uw}=(\pi(\mathcal A))^{\prime\prime}$
and 

(ii) $\mathfrak Z \subseteq (\pi(\mathcal A))^{-uw}.$
\end{lemma}

\begin{proof}[\bf Proof]
Since $\mathfrak Z$ is the centre of $\mathcal L(\mathcal{H}_{\mathcal A})$ 
by Proposition \ref{TypeI} we see that
$$(\pi(\mathcal A))^{\prime} = [alg\{\pi(\mathcal A),\mathfrak Z\}]^\prime.$$
Moreover, since $\mathcal L(\mathcal{H}_{\mathcal A})$ is type $I$ with centre 
$\mathfrak Z$ and $\mathfrak Z \subseteq alg(\pi(\mathcal A),\mathfrak Z)$,
we have by compl\'{e}ment 13, III.7 of \cite{Dix} that
$$[alg(\pi(\mathcal A),\mathfrak Z)]^{\prime\prime} =
[alg(\pi(\mathcal A),\mathfrak Z)]^{-uw}.$$
Hence,

\noindent (1)\hspace{1.3in}$(\pi(\mathcal A))^{\prime\prime}=
[alg(\pi(\mathcal A),\mathfrak Z)]
^{\prime\prime} =
[alg(\pi(\mathcal A),\mathfrak Z)]^{-uw}.$

Now, $\pi(\mathcal A)$ is a $*$-ideal in the $*$-algebra $alg(\pi(\mathcal A),
\mathfrak Z)$ so that $(\pi(\mathcal A))^{-uw}$ is a $*$-ideal in
$[alg(\pi(\mathcal A),\mathfrak Z)]^{-uw}$ so that there exists a central
projection $E$ in $[alg(\pi(\mathcal A),\mathfrak Z)]^{-uw}$ with
$$(\pi(\mathcal A))^{-uw}=E[alg(\pi(\mathcal A),\mathfrak Z)]^{-uw}.$$
If $E\neq 1$ then $1-E \neq 0$ but $(1-E)\pi(\mathcal A) = \{0\}$,
contradicting the previous lemma. Hence,

\noindent(2)\hspace{1.9in}$(\pi(\mathcal A))^{-uw}=[alg(\pi(\mathcal A),
\mathfrak Z)]^{-uw}.$

Equations (1) and (2) imply part (i). Part (ii) follows since $\mathfrak Z$
is contained in any commutant.
\end{proof}

\begin{lemma}\label{*map}
The map $*$ extends to a conjugate-linear isometry of $\mathcal{H}_\mathcal A$
(also denoted by $*$) by defining $\xi^*(a):=(\xi(a^*))^*$ for $\xi\in \mathcal{H}_
{\mathcal A}$ and $a\in \mathcal A.$ This extension satisfies
$$\bra\xi ,\eta\ket^*= \bra\xi^* ,\eta^*\ket=\bra\eta ,\xi\ket,$$ 
for all $\xi,\eta\in \mathcal{H}_\mathcal A.$
\end{lemma}

\begin{proof}[\bf Proof]
It is easy to see that $\xi^*$ is a bounded $\mathfrak Z$-module map and
that $\|\xi^*\|\leq \|\xi\|.$ Since $\xi^{**}=\xi$ we see that $*$ is isometric
on $\mathcal{H}_\mathcal A.$ By axioms (ii) and (v) we have for $a,b\in\mathcal A,$
$$(\hat{b})^*(a)=(\hat{b}(a^*))^*=\bra b,a^*\ket^*=\bra a^*,b\ket=\bra b^*,a\ket
=\widehat{b^*}(a),$$
so that this $*$ really is an extension from $\mathcal A$ to $\mathcal{H}_\mathcal A.$
Moreover, using the definition of module multiplication given in Definition \ref{dual}
it is easy to check that $(\xi z)^*=\xi^*z^*$ for all $z\in \mathfrak Z$
and $\xi\in \mathcal{H}_\mathcal A.$ 

We observe that $\mathfrak Z$ is a self-dual Hilbert $\mathfrak Z$-module
with the inner product $\bra z_1 ,z_2\ket=z_1^*z_2:$ for, if $\theta:
\mathfrak Z \to \mathfrak Z$ is a bounded $\mathfrak Z$-module map
then $\theta(z)=\theta(1)z=\bra\theta(1)^*,z\ket.$

Now if $\xi\in \mathcal{H}_\mathcal A$, then by Proposition 3.6 of \cite{Pa},
$\xi$ extends uniquely to a bounded $\mathfrak Z$-module mapping:
$\mathcal{H}_{\mathcal A} \to \mathfrak Z.$ But, using the first paragraph of the
proof one checks that $\eta\mapsto\bra\xi ,\eta\ket$
and $\eta \mapsto\bra\xi^* ,\eta^*\ket^*$ are two such extensions. Hence,
$$\bra\xi ,\eta\ket=\bra\xi^* ,\eta^*\ket^*$$
as claimed.

The equality, $\bra\xi ,\eta\ket^*=\bra\eta ,\xi\ket$ follows from axiom (ii)
since $\mathcal{H}_\mathcal A$ is a (self-dual) Hilbert $\mathfrak Z$-module by 
Theorem 3.2 of \cite{Pa}. 
\end{proof}

\begin{defn}\label{J}
The isometry $\eta\mapsto\eta^*: \cal{H}_{\cal{A}}\to\cal{H}_{\cal{A}}$ of the previous lemma will be 
denoted by $J.$ That is, $J(\eta)=\eta^*$ for all $\eta\in\cal{H}.$
\end{defn}

\begin{rems*}
The unique extension of Proposition 3.6 of \cite{Pa} used 
in the previous proof will be used several more times in this paper under 
the name ``unique extension property.''
\end{rems*}

\begin{lemma}
With the standing assumptions of this section,

(1) $\mathfrak Z \subseteq (\pi^{\prime}(\mathcal A))^{-uw}=
(\pi^{\prime}(\mathcal A))^{\prime\prime},$

(2) $\pi(\mathcal A) \subseteq (\pi^{\prime}(\mathcal A))^{\prime}\;and$

(3) $\pi^{\prime}(\mathcal A) \subseteq (\pi(\mathcal A))^{\prime}.$
\end{lemma}

\begin{proof}[\bf Proof]
(1) This is the same proof as Lemma \ref{vNd}.

(2) and (3) By the unique extension property, it suffices to see that
$\pi^{\prime}(a)\pi(b) = \pi(b)\pi^{\prime}(a)$
on the space $\mathcal A \subseteq \mathcal{H}_{\mathcal A}$. This is trivial to check.
\end{proof}

\subsection {Bounded elements in $\mathcal{H}_{\mathcal A}$}

Let $\xi\in \mathcal{H}_{\cal A}$ and suppose that the map 
$$a\mapsto\pi^{\prime}(a)\xi:\cal A \to \cal{H}_{\cal A}$$ is bounded. 
We note that by the remarks
following example \ref{Hilbex}, $\pi(az)=\pi(a)z=z\pi(a)$ and 
$\pi^{\prime}(az)=\pi^{\prime}(a)z=z\pi^{\prime}(a),$ for all $a\in\cal A\;
and\;z\in\frk Z.$ Therefore,
$$(az)\mapsto\pi^{\prime}(az)\xi=z\pi^{\prime}(a)\xi=(\pi^{\prime}(a)\xi)z$$
so that this bounded map is also $\frk Z$-linear. Hence by the unique 
extension property this map extends uniquely to a bounded module mapping
$\cal{H}_{\cal A} \to \cal{H}_{\cal A}$ which we denote by $\pi(\xi).$ That is,
$\pi(\xi)a=\pi^{\prime}(a)\xi$ for all $a\in \cal A.$ By Proposition 3.4
of \cite{Pa} $\pi(\xi)$ is adjointable and 
$\pi(\xi)\in\cal L(\cal{H}_{\cal A}).$ Such an element $\xi\in \cal{H}_\cal A$ is called
$left-bounded$ and the set of all such elements is denoted $\cal A_l$. Clearly,
$\cal A \subseteq \cal A_l.$

Similarly, we let $\cal A_r = \{\eta\in \cal{H}_{\cal A}\,|\,\pi^{\prime}(\eta)\in
\cal L(\cal{H}_{\cal A})\}.$ Where, of course, $\pi^{\prime}(\eta)a=\pi(a)\eta$
for all $a\in\cal A.$

\begin{prop}\label{lbdd}
With the standing assumptions of this section,

(1) $\pi(\cal A_l) \subseteq (\pi^{\prime}(\cal A))^\prime$ and similarly
$\pi^{\prime}(\cal A_r) \subseteq (\pi(\cal A))^\prime ,$

(2) $\pi(\cal A_l)$ is a left ideal in $(\pi^{\prime}(\cal A))^{\prime}$ and
$T\pi(\xi)=\pi(T\xi) \;for \;\xi\in\cal A_l\;and\;T\in 
(\pi^{\prime}(\cal A))^{\prime}.$ In particular, $\pi(\eta)\pi(\xi)=
\pi(\pi(\eta)\xi)\;for\;
\eta,\xi\in\cal A_l.$ Similarly, $\pi^{\prime}(\cal A_r)$ is a left ideal 
in $(\pi(\cal A))^{\prime}$, etc.

(3) $\cal A_l$ is an associative algebra with the multiplication
$\xi\eta=\pi(\xi)\eta$ and $\pi:\cal A_l \to \cal L(\cal{H}_\cal A)$ is a 
monomorphism. Similarly, $\cal A_r$ is an associative algebra with the
multiplication $\xi\eta=\pi^{\prime}(\eta)\xi,$ and $\pi^{\prime}$ is
an anti-monomorphism.

(4) $\cal A_l$ is invariant under $*$ and $\pi(\xi^*)=\pi(\xi)^*$ so that
$\pi(\cal A_l)$ is a $*$-ideal in $(\pi^{\prime}(\cal A))^\prime$
and $\pi$ is a $*$-monomorphism. A similar statement holds for $\cal A_r$.
\end{prop}

\begin{proof}[\bf Proof]
(1) By the unique extension property, it suffices to check that if 
$\xi\in\cal A_l,\;and\;b\in\cal A$ then 
$\pi(\xi)\pi^{\prime}(b)=\pi^{\prime}(b)\pi(\xi)$ on the space $\cal A.$
To this end let $a\in\cal A$, then:
$$(\pi(\xi)\pi^{\prime}(b))(a)=\pi(\xi)(ab)=\pi^{\prime}(ab)(\xi)=
\pi^{\prime}(b)\pi^{\prime}(a)(\xi)=\pi^{\prime}(b)\pi(\xi)(a),$$
as required.

(2)If $\xi\in\cal A_l$, $T\in (\pi^{\prime}(\cal A))^{\prime}$ 
and $a\in \cal A$, then:
$$\pi(T\xi)a=\pi^{\prime}(a)T\xi=T\pi^{\prime}(a)\xi=T\pi(\xi)a.$$
That is, $T\xi\in\cal A_l$ and $\pi(T\xi)=T\pi(\xi)$ by the unique 
extension property.

(3)By (2), $\xi\eta:=\pi(\xi)\eta$ is in $\cal A_l$ if $\xi,\eta\in\cal A_l.$
Moreover, by (2) $\pi(\xi\eta)=\pi(\xi)\pi(\eta)$. Since $\pi:\cal A_l\to
\cal L(\cal{H}_{\cal A})$ is clearly linear, it suffices to see that $\pi$ is
also one-to-one. But if $\pi(\xi)=0$, then for all $a,b\in\cal A$ we have
$$0=\bra\pi(\xi)a,b\ket_{\omega}=\bra\pi^{\prime}(a)\xi,b\ket_{\omega}=
\bra\xi,ba^*\ket_{\omega}$$ for all positive $\omega\in\frk Z_*.$ That is,
$\xi=0$ by axiom (viii) and Proposition \ref{density}.

(4)Let $\xi\in\cal A_l$ and let $a,b\in\cal A.$ Using Lemma \ref{*map}  and the
fact that $\cal{H}_{\cal A}$ is a Hilbert $\frk Z$-module, we get the
following calculation:
\begin{eqnarray*}
\bra\pi(\xi)^*a,b\ket&=&\bra b,\pi(\xi)^*a\ket^*=\bra\pi(\xi)b,a\ket^*
=\bra\pi^{\prime}(b)\xi,a\ket^*\\
&=&\bra\xi,ab^*\ket^*=\bra\xi^*,ba^*\ket=
\bra\xi^*,\pi^{\prime}(a^*)b\ket=\bra\pi^{\prime}(a)\xi^*,b\ket\\
&=&\bra\pi(\xi^*)a,b\ket.
\end{eqnarray*}
Thus, as module maps $\pi(\xi)^*a$ and $\pi(\xi^*)a$ agree for all $b\in\cal A$
and so $\pi(\xi)^*a=\pi(\xi^*)a$ for all $a\in\cal A.$ That is, $\xi^*$ is
left-bounded and $\pi(\xi^*)=\pi(\xi)^*.$ Moreover, for $\xi,\eta\in\cal A_l$
$$\pi((\xi\eta)^*)=[\pi(\xi\eta)]^*=[\pi(\xi)\pi(\eta)]^*=\pi(\eta)^*
\pi(\xi)^*=\pi(\eta^*)\pi(\xi^*)=\pi(\eta^*\xi^*)$$
and so $(\xi\eta)^*=\eta^*\xi^*$ as $\pi$ is one-to-one.
\end{proof}

\begin{cor}\label{comm1}With the standing assumptions of this section ,

(1)\;$(\pi(\cal A_l))^{\prime\prime}=\pi(\cal A_l)^{-uw}=
(\pi^{\prime}(\cal A))^{\prime},$ and

(2)\;$(\pi^{\prime}(\cal A_r))^{\prime\prime}=\pi^{\prime}(\cal A_r)
^{-uw}=(\pi(\cal A))^{\prime}.$
\end{cor}

\begin{proof}[\bf Proof]
(1)\;By Proposition \ref{lbdd}, $\pi(\cal A_l)^{-uw}$ is a 
$*$-ideal in $(\pi^{\prime}(\cal A))^{\prime}.$ But by Lemma 4.2,\\
$1\in\frk Z\subseteq\pi(\cal A)^{-uw}\subseteq\pi(\cal A_l)^{-uw}$
and so $\pi(\cal A_l)^{-uw}=(\pi^{\prime}(\cal A))^{\prime}.$
Now, since $\frk Z\subseteq(\pi(\cal A_l))^{-uw}$ we have by
compl\'{e}ment 13, III.7 of \cite{Dix} that 
$$(\pi(\cal A_l)^{-uw})^{\prime\prime}=\pi(\cal A_l)^{-uw}.$$
But then, since commutants are always ultraweakly closed:
$$(\pi(\cal A_l))^{\prime\prime}=(\pi(\cal A_l)^{\prime\prime})^{-uw}
\supseteq(\pi(\cal A_l))^{-uw}=(\pi(\cal A_l)^{-uw})^{\prime\prime}
\supseteq(\pi(\cal A_l))^{\prime\prime}.$$
The proof of (2) is similar.
\end{proof}

\begin{prop}\label{comm2}
With the standing assumptions of this section, $\cal A_l =\cal A_r$ and

(1)\;$\pi^{\prime}(\xi)a=[\pi(\xi^*)a^*]^*$ for $\xi\in\cal A_l$, $a\in\cal A.$

(2)\;$\pi(\xi)a=[\pi^{\prime}(\xi^*)a^*]^*$ for $\xi\in\cal A_r$, $a\in\cal A.$
\end{prop}

\begin{proof}[\bf Proof]
(1)\;Let $\xi\in\cal A_l.$ Then for $a,b\in\cal A,$
\begin{eqnarray*}
\bra\pi^{\prime}(\xi)a,b\ket&=&\bra\pi(a)\xi,b\ket=\bra\xi,a^*b\ket\\
&=&\bra\xi^*,b^*a\ket^*=\bra\pi^{\prime}(a^*)\xi^*,b^*\ket^*=
\bra\pi(\xi^*)a^*,b^*\ket^*\\
&=&\bra[\pi(\xi^*)a^*]^*,b\ket.
\end{eqnarray*}
Therefore, $\xi\in\cal A_r$ so that $\cal A_l\subseteq\cal A_r$ and (1) holds.
Similarly, $\cal A_r\subseteq\cal A_l$ and (2) holds.
\end{proof}

\begin{cor}\label{comm3}
For all $\xi\in\cal A_l =\cal A_r$ and $\eta\in \cal{H}_\cal A,$

(1)\;$\pi^{\prime}(\xi)\eta=[\pi(\xi^*)\eta^*]^*$ and

(2)\;$\pi(\xi)\eta=[\pi^{\prime}(\xi^*)\eta^*]^*.$
\end{cor}

\begin{proof}[\bf Proof]
(1)\;Recall $J:\cal{H}_{\cal A}\to \cal{H}_{\cal A}$ is the conjugate-linear isometry 
$J\eta=\eta^*.$ As noted in the proof of Lemma \ref{*map},
$J(\eta z)=(J\eta)z^*$ for $z\in\frk Z.$ Now, by part (1) of the previous
proposition, we see that for $\xi\in\cal A_l=\cal A_r$, $\pi^{\prime}(\xi)$
and $J\pi(\xi^*)J$ agree on $\cal A.$ Since both of these maps are bounded
$\mathfrak Z$-module maps they agree on $\cal{H}_\cal A$ by uniqueness. This proves
part (1). The proof of part (2) is similar.
\end{proof}
\begin{prop}\label{comm4}
Let $\xi,\eta\in\cal A_l=\cal A_r$, then we have:

(1)\;$\pi(\xi)\eta=\pi^{\prime}(\eta)\xi$ so that the two multiplications 
of Proposition \ref{lbdd}
agree, and

(2)\;$\pi(\xi)\pi^{\prime}(\eta)=\pi^{\prime}(\eta)\pi(\xi).$
\end{prop}

\begin{proof}[\bf Proof]
(1)\;Fix $a\in \cal A$, then:
\begin{eqnarray*}
\bra\pi(\xi)\eta,a\ket &=& \bra(\pi(\xi)\eta)^*,a^*\ket^*
= \bra\pi^{\prime}(\xi^*)\eta^*,a^*\ket^*=
\bra\eta^*,\pi^{\prime}(\xi)a^*\ket^*=\bra\eta^*,\pi(a^*)\xi\ket^*\\
&=& \bra\pi(a)\eta^*,\xi\ket^*=\bra\pi^{\prime}(\eta^*)a,\xi\ket^*=
\bra a,\pi^{\prime}(\eta)\xi\ket^*
= \bra\pi^{\prime}(\eta)\xi,a\ket
\end{eqnarray*}
so that (1) holds.

(2)\;Again fix $a\in\cal A$ then,
\begin{eqnarray*}
\pi(\xi)\pi^{\prime}(\eta)a &=& \pi(\xi)\pi(a)\eta
= \pi(\pi(\xi)a)\eta\;\;by\;\ref{lbdd}(2)\\
&=& \pi^{\prime}(\eta)(\pi(\xi)a)
=\pi^{\prime}(\eta)\pi(\xi)a.
\end{eqnarray*}
\end{proof}

\begin{not*}
Since $\cal A_l=\cal A_r$ (even as $*$-algebras) we now use the notation
$\cal A_b$ to denote the $*$-algebra of {\bf bounded} elements in $\cal{H}_\cal A.$ 
\end{not*}

\begin{thm}\label{Comm}{\bf [Commutation Theorem]}
Let $\cal A$ be a $\frk Z$-Hilbert Algebra over the abelian von Neumann
algebra $\frk Z.$ Then,

(1)\;$\pi(\cal A)^{-uw}=(\pi(\cal A))^{\prime\prime}=
(\pi(\cal A_b))^{\prime\prime}=\pi(\cal A_b)^{-uw}=
(\pi^{\prime}(\cal A_b))^\prime =(\pi^{\prime}(\cal A))^\prime$ and

(2)\;$\pi^{\prime}(\cal A)^{-uw}=(\pi^{\prime}(\cal A))^{\prime\prime}=
(\pi^{\prime}(\cal A_b))^{\prime\prime}=\pi^{\prime}(\cal A_b)^{-uw}=
(\pi(\cal A_b))^\prime =(\pi(\cal A))^\prime.$
\end{thm}

\begin{proof}[\bf Proof]
(1)\;By part (1) of Corollary \ref{comm1}, we have 
$$(\pi(\cal A_b))^{-uw}=(\pi(\cal A_b))^{\prime\prime}=
(\pi^{\prime}(\cal A))^{\prime}\supseteq (\pi^{\prime}(\cal A_b))^\prime.$$ 
However, by part (2) of the previous corollary, we have
$$(\pi(\cal A_b))^{\prime\prime}\subseteq (\pi^{\prime}(\cal A_b))^
{\prime\prime\prime}=(\pi^{\prime}(\cal A_b))^\prime.$$
Hence, 
$$(\pi(\cal A_b))^{-uw}=(\pi(\cal A))^{\prime\prime}
=(\pi^{\prime}(\cal A))^{\prime}=(\pi^{\prime}(\cal A_b))^\prime.$$
On the other hand, by part (2) of Corollary \ref{comm1}:
$$(\pi(\cal A))^{\prime\prime}=(\pi^{\prime}(\cal A_b))^{\prime\prime\prime}=
(\pi^{\prime}(\cal A_b))^\prime.$$
Since $\pi(\cal A)^{-uw}=(\pi(\cal A))^{\prime\prime}$ by Lemma \ref{vNd},
we are done.

The proof of (2) is similar.
\end{proof}

\begin{defn}
We define the {\bf left von Neumann algebra of $\cal A$} to be 
$$\cal U(\cal A):=(\pi(\cal A))^{\prime\prime}.$$ We define the 
{\bf right von Neumann algebra of $\cal A$} to be
$$\cal V(\cal A):=(\pi^{\prime}(\cal A))^{\prime\prime}.$$
\end{defn}

\begin{cor}
Let $\cal A$ be a $\frk Z$-Hilbert algebra over the abelian von Neumann
algebra $\frk Z.$ Then, for all $\xi,\eta\in\cal A_b,$ with $J$ as in 
Definition \ref{J}

\noindent(1)\hspace{1.6in}$J\pi(\xi)J=\pi^{\prime}(J\xi)\;
and\;J\pi^{\prime}(\xi)J=\pi(J\xi).$

\noindent(2)\hspace{1.6in}$J\cal U(\cal A)J=\cal V(\cal A)\;
and\;J\cal V(\cal A)J=\cal U(\cal A).$
\end{cor}

\begin{proof}[\bf Proof]
Item (1) is just Corollary \ref{comm1}. 

To see item (2), let $T\in\cal U(\cal A)=(\pi^{\prime}(\cal A_b))^\prime.$
Then for $\xi\in\cal A_b$ and $\eta\in \cal{H}_\cal A$ we get:
\begin{eqnarray*}
JTJ\pi(\xi)\eta &=& JTJ\pi(\xi)J\eta^*\\
&=& JT\pi^{\prime}(J\xi)\eta^*=J\pi^{\prime}(J\xi)T\eta^*
=J\pi^{\prime}(J\xi)JJTJ\eta\\
&=& \pi(\xi)JTJ\eta.
\end{eqnarray*}
Therefore, $J\cal U(\cal A)J\subseteq(\pi(\cal A_b))^{\prime}=\cal V(\cal A).$
Similarly, $J\cal V(\cal A)J\subseteq\cal U(\cal A).$ Since $J^2=1$, we're
done.
\end{proof}

\begin{rems*}
At this point we could show that $\cal A_b$ is a $\frk Z$-Hilbert
algebra satisfying\\$\cal{H}_{\cal A_b} =\cal{H}_\cal A$, $\cal U(\cal A_b)=\cal U(\cal A)$,
and $\cal V(\cal A_b)=\cal V(\cal A).$ Since we don't appear to need this now,
we defer the statement and proof to Proposition \ref{full}.
\end{rems*}

\section{ CENTRE-VALUED TRACES}

With the same hypotheses and notation of the previous section we show how
to construct a natural $\frk Z$-valued trace on the von Neumann algebra, 
$\cal U(\cal A).$ We first remind the reader of Paschke's results that
both $\cal{H}_\cal A$ and $\cal L(\cal{H}_\cal A)$ are dual spaces, and that since
$\cal L(\cal{H}_\cal A)$ is a von Neumann algebra, its $weak^*$-topology must also
be its $uw$-topology since pre-duals for von Neumann algebras are unique.

\begin{tech*}
The problem of convergence is one of our main headaches. 
The topology of Proposition \ref{density} (closely related to a topology introduced
by Paschke \cite{Pa}) and Proposition 3.10 of \cite{Pa} are
exactly what is needed to prove the following result which is used
several times in the remainder of this paper.
\end{tech*}
\begin{prop}\label{conv}
If $\cal A$ is a pre-Hilbert $\frk Z$-module (not necessarily a 
$\frk Z$-Hilbert Algebra) with Paschke dual $\cal{H}_{\cal A}$, then:

\noindent(1)\;A bounded net $\{\xi_\alpha\}$ in $\cal{H}_\cal A$ converges $weak^*$ 
to $\xi\in \cal{H}_\cal A \Longleftrightarrow$\\
\indent$\bra\eta,\xi_\alpha\ket\to\bra\eta,\xi\ket$ ultraweakly in
$\frk Z$ for all $\eta\in \cal{H}_\cal A.$

\noindent(2)\;A net $\{T_\alpha\}$ in $\cal L(\cal{H}_\cal A)$ converges
ultraweakly to $T\in\cal L(\cal{H}_\cal A) \Longleftrightarrow$\\ 
\indent$\bra T_{\alpha}\xi,\eta\ket\to\bra T\xi,\eta\ket$ ultraweakly in 
$\frk Z$ for all $\xi,\eta\in \cal{H}_\cal A.$

\noindent(3)\;A {\bf bounded} net $\{T_\alpha\}$ in $\cal L(\cal{H}_\cal A)$ converges
ultraweakly to $T\in\cal L(\cal{H}_\cal A) \Longleftrightarrow$\\ 
\indent$\bra T_{\alpha}a,b\ket\to\bra Ta,b\ket$ ultraweakly in 
$\frk Z$ for all $a,b\in \cal A.$

\end{prop}

\begin{proof}[\bf Proof]

Item (1) is just Remark 3.9 of \cite{Pa} and works for any self-dual Hilbert
module over a von Neumann algebra.

Item (2) follows immediately from the definition of the $weak^*$ topology on
$\cal L(\cal{H}_\cal A)$ in Remark 3.9 and the proof of Proposition 3.10 of
\cite{Pa}. This result also holds for any self-dual Hilbert
module over a von Neumann algebra.

Item (3) follows from item (2) and the usual $\epsilon/3$-argument using item
(iv) of Proposition \ref{density}.
\end{proof}

Since $\pi(\cal A_b^2)$ is going to be the domain of definition of our
$\frk Z$-valued trace on $\cal U(\cal A)$, we need a condition on an
operator $T\in\cal U(\cal A)$ (involving $\frk Z$-valued inner products)
to be an element of $\pi(\cal A_b).$

\begin{rem*}
In Example \ref{Hilbex} where our $\frk Z$-Hilbert algebra is itself a von Neumann
algebra $\frk A$ with $\frk Z\subseteq Z(\frk A)$ and a faithful, tracial,
uw-continuous $\frk Z$-trace $\tau:\frk A\to \frk Z$, one can use item (3)
in Proposition \ref{conv} to show that $\pi(\frk A)=(\pi(\frk A))^{\prime\prime},$
as expected.
\end{rem*}

\begin{prop}\label{Tbdd}
If $T\in\cal U(\cal A)$ then $T\in\pi(\cal A_b)$ if and only if 
$$\{\bra T\xi,T\xi\ket\,|\,\xi\in\cal A_b\;and\;\|\pi(\xi)\|\leq 1\}
\;is\;bounded\;above\;in\;\frk Z_{+}.$$
In this case, $T=\pi(\eta)$ where $z=\bra\eta,\eta\ket,$ and $z$ is the 
supremum of this set in $\frk Z_{+}.$
\end{prop}

\begin{proof}[\bf Proof]
$(\Longleftarrow)$ Let $z$ be an upper bound for this set in $\frk Z_{+}.$
Let $\{\pi(\xi_{\alpha})\}$ be a net in $\pi(\cal A_b)$
converging ultraweakly to $1$ and norm bounded by $1$. Then, 
$$\|T\xi_{\alpha}\|=\|\bra T\xi_{\alpha},T\xi_{\alpha}\ket\|^{1/2}\leq
\|z\|^{1/2}$$
so that $\{T\xi_{\alpha}\}$ is a bounded net in the dual space
$\cal{H}_\cal A$ and so we can assume that it converges $weak^*$ to some 
$\eta\in \cal{H}_\cal A.$
That is,
$$T\xi_{\alpha}\stackrel{w^*}{\to}\eta\;and\;\pi(T\xi_{\alpha})=
T\pi(\xi_{\alpha})\stackrel{uw}{\to}T.$$
By Proposition \ref{conv} we see that for all $a\in\cal A$ and all $\mu\in \cal{H}_\cal A$:
\begin{eqnarray*}
\bra Ta,\mu\ket &=& \lim_{\alpha}\bra\pi(T\xi_{\alpha})a,\mu\ket
= \lim_{\alpha}\bra\pi^{\prime}(a)T\xi_{\alpha},\mu\ket
= \lim_{\alpha}\bra T\xi_{\alpha},\pi^{\prime}(a^*)\mu\ket\\
&=& \bra\eta,\pi^{\prime}(a^*)\mu\ket
= \bra\pi(\eta)a,\mu\ket.
\end{eqnarray*}
So, $Ta=\pi(\eta)a$ for all $a\in\cal A$ and hence $T=\pi(\eta)$ where 
$\eta\in\cal A_b.$

$(\Longrightarrow)$ On the other hand, if $T=\pi(\eta)$ for some 
$\eta\in\cal A_b$, then for all $\xi\in\cal A_b$ with $\|\pi(\xi)\|\leq 1$
we get by Proposition 2.6 of \cite{Pa}: 
\begin{eqnarray*}
\bra T\xi,T\xi\ket &=& \bra\eta\xi,\eta\xi\ket=\bra\xi^*\eta^*,\xi^*\eta^*\ket\\
&=& \bra\pi(\xi\xi^*)\eta^*,\eta^*\ket\leq \|\pi(\xi\xi^*)\|\bra\eta,\eta\ket
\leq\bra\eta,\eta\ket\in\frk Z. 
\end{eqnarray*}

Now, since $\frk Z$ is abelian, the supremum of any finite set of
self-adjoint elements exists and so the supremum of the bounded set, 
$\{\bra T\xi,T\xi\ket\,|\,\xi\in\cal A_b\;and\;\|\pi(\xi)\|\leq 1\}$ 
can be written as the limit of a bounded increasing net of elements in 
$\frk Z_+$ 
which exists (in $\frk Z_+$) by Vigier's Theorem. We let $z_0$ be this 
supremum. Then, if $T=\pi(\eta)$ for $\eta\in\cal A_b$ we see by the second 
part of the above argument that $z_0\leq\bra\eta,\eta\ket.$

On the other hand, If we choose the net $\{\xi_{\alpha}\}$ as in the first part 
of the above argument to also satisfy $\xi_{\alpha}^*=\xi_{\alpha}$, then:
\begin{eqnarray*}
\bra T\xi_{\alpha},T\xi_{\alpha}\ket &=& 
\bra\eta\xi_{\alpha},\eta\xi_{\alpha}\ket
=\bra\xi_{\alpha}\eta^*,\xi_{\alpha}\eta^*\ket\\
&=& \bra\pi(\xi_{\alpha})^2\eta^*,\eta^*\ket\stackrel{uw}{\longrightarrow}
\bra\eta^*,\eta^*\ket
=\bra\eta,\eta\ket.
\end{eqnarray*}
That is $\bra\eta,\eta\ket\geq z_0,$ and we're done.
\end{proof}

\begin{lemma}\label{Ideal}
Let $\cal I=\pi(\cal A_b)^2:=span\{\pi(\xi)\pi(\eta)\,|\,\xi,\eta\in\cal A_b\}.$
Then $\cal I$ is an uw dense $*$-ideal in $\cal U(\cal A)$ and
$\cal I_+ = \{\pi(\xi^*)\pi(\xi)\,|\,\xi\in\cal A_b\}$.
\end{lemma}

\begin{proof}[\bf Proof]
It follows from Proposition \ref{lbdd} and Theorem \ref{Comm} that $\cal I$ is an uw dense 
$*$-ideal in $\cal U(\cal A)$. Let $\cal I_0=\{\pi(\xi^*)\pi(\xi)\,|\,
\xi\in\cal A_b\}.$ We verify that $\cal I_0$ satisfies the conditions of 
Lemme 1 of I.1.6 of \cite{Dix}.

(i)\;$\cal I_0$ is unitarily invariant in $\cal U(\cal A)$ since $\pi(\cal A_b)$
is an ideal in $\cal U(\cal A).$

(ii)\;Let $\eta\in\cal A_b$ and let $T\in \cal U(\cal A)_+$ with
$0\leq T\leq\pi(\eta^*)\pi(\eta).$ Then for each $\xi\in\cal A_b$ with
$\|\pi(\xi)\|\leq 1$ we get:
\begin{eqnarray*}
\bra T^{1/2}\xi, T^{1/2}\xi\ket &=& \bra T\xi,\xi\ket\leq
\bra\pi(\eta^*)\pi(\eta)\xi,\xi\ket\\
&=& \bra\eta\xi,\eta\xi\ket=\bra\xi^*\eta^*,\xi^*\eta^*\ket
\leq \|\pi(\xi^*)\|^2\bra\eta^*,\eta^*\ket
\leq \bra\eta,\eta\ket.
\end{eqnarray*}
By Proposition \ref{Tbdd}, $T^{1/2}=\pi(\mu)$ for some $\mu\in\cal A_b.$ That is,
$T=\pi(\mu^*)\pi(\mu)\in\cal I_0.$ 

(iii)\;If $S=\pi(\eta^*\eta)$ and $T=\pi(\mu^*\mu)$ are in $\cal I_0$, 
then for all $\xi\in\cal A_b$ with $\|\pi(\xi)\|\leq 1$ we have:
\begin{eqnarray*}
\bra (S+T)^{1/2}\xi,(S+T)^{1/2}\xi\ket &=& \bra S\xi,\xi\ket +
\bra T\xi,\xi\ket
=\bra\pi(\eta^*\eta)\xi,\xi\ket + \bra\pi(\mu^*\mu)\xi,\xi\ket\\
&\leq& \cdots\leq\bra\eta,\eta\ket + \bra\mu,\mu\ket.
\end{eqnarray*}
Again by Proposition \ref{Tbdd}, $(S+T)^{1/2}=\pi(\gamma)$ for some 
$\gamma\in\cal A_b,$ and so $S+T=\pi(\gamma^*\gamma)\in \cal I_0.$

Hence, $\cal I_0=\cal J_+$ the positive part of an ideal $\cal J$ and
$\cal J = span\cal I_0.$ Clearly, $\cal J\subseteq\cal I.$
On the other hand, if $\xi,\eta\in\cal A_b$ then
$$\pi(\xi)\pi(\eta^*)=\frac{1}{4}\sum_{k=0}^{3}i^k\pi(\xi+i^k\eta)
\pi((\xi+i^k\eta)^*)\;is\;in\;\cal J.$$
Thus, $\cal I\subseteq \cal J,$ and so they are equal. That is,
$$\{\pi(\xi^*)\pi(\xi)\,|\,\xi\in\cal A_b\}=\cal I_0=\cal J_+=\cal I_+.$$ 
\end{proof}

\begin {cor}\label{Ispan}
With the above hypotheses, 
$$\cal I:=span\{\pi(\xi)\pi(\eta)\,|\,\xi,\eta\in\cal A_b\}
=\{\pi(\xi)\pi(\eta)\,|\,\xi,\eta\in\cal A_b\}.$$
\end{cor}

\begin{proof}[\bf Proof]
Let $T\in\cal I$ and let $T=V|T|$ be the polar decomposition of $T$ in
$\cal U(\cal A).$ Then $|T|=V^*T\in\cal I_+.$ Hence, 
$$T=V|T|=V\pi(\xi)\pi(\xi^*)=\pi(V\xi)\pi(\xi^*)$$
by part (2) of Proposition \ref{lbdd}.
\end{proof}

\begin{rems*}
At this point we can define a ``trace'' on the ideal $\cal I$ in the usual
way:
$$\tau(\pi(\xi\eta)):=\bra\xi^*,\eta\ket,$$
as in the following theorem. However, in order to connect this up with
Dixmier's ``trace op\'eratorielle'' \cite{Dix} which includes unbounded
operators affiliated with $\frk Z$ in its range (and also includes
a notion of normal) we are forced to work a little harder. 
\end{rems*}

\begin{thm}\label{Tr}
Let $\cal A$ be a $\frk Z$-Hilbert algebra over the abelian von Neumann
algebra $\frk Z.$ Let $\cal I=\pi(\cal A_b^2)$ be the canonical uw dense
$*$-ideal in $\cal U(\cal A)= (\pi(\cal A))^{\prime\prime}$, the left
von Neumann algebra of $\cal A.$ Then, $\tau:\cal I\to\frk Z$ defined by
$$\tau(\pi(\xi\eta))=\bra\xi^*,\eta\ket$$
is a well-defined positive $\frk Z$-linear mapping which is:

\noindent(1)\;{\bf faithful}, i.e., $\tau(T)=0$ and 
$T\geq 0 \Longrightarrow T=0\;and,$

\noindent(2)\;{\bf tracial}, i.e., 
$\tau(TS)=\tau(ST)\;for\;T\in\cal U(\cal A)\;and\;
S\in\cal I.$
\end{thm}

\begin{proof}[\bf Proof]
To see that $\tau$ is well-defined, fix a net $\{\xi_{\alpha}\}$ in
$\cal A_b$ with  $\pi^{\prime}(\xi_{\alpha})\to 1$ ultraweakly. Let 
$T=\pi(\xi\eta)\in\cal I$. Then the element $\xi\eta\in\cal A_b^2$ is 
unique since $\pi$ is one-to-one (of course, its representation as a product
is not unique). Now,
$$\tau(T)=\bra\xi^*,\eta\ket=uw\lim_{\alpha}
\bra\pi^{\prime}(\xi_{\alpha})
\xi^*,\eta\ket=uw\lim_{\alpha}\bra\xi_{\alpha},\xi\eta\ket.$$
That is, $\tau(T)$ is uniquely determined by $T$. Thus, $\tau(T)$ is
well-defined and $\frk Z$-linear.

If $T\in\cal I_+$, then $T=\pi(\xi^*\xi)$ by Lemma \ref{Ideal} and  
$\tau(T)=\bra\xi,\xi\ket\geq 0$ so that $\tau$ is positive. Clearly,
$\tau(T)=0 \Longrightarrow \xi=0 \Longrightarrow \pi(\xi)=0 \Longrightarrow
T=0.$ That is, $\tau$ is faithful.

To see that $\tau$ is tracial, let $S=\pi(\xi\eta)\in\cal I$ and let
$T\in \cal U(\cal A)$. Then,
\begin{eqnarray*}
\tau(TS) &=& \tau(T\pi(\xi)\pi(\eta))=\tau(\pi(T\xi)\pi(\eta))
=\bra(T\xi)^*,\eta\ket = \bra T\xi,\eta^*\ket^*\\
&=& \bra\xi,T^*(\eta^*)\ket^* = \bra\xi^*,(T^*(\eta^*))^*\ket
= \tau(\pi(\xi)\pi(T^*(\eta^*))^*) = \tau(\pi(\xi)[T^*\pi(\eta^*)]^*)\\
&=& \tau(\pi(\xi)\pi(\eta)T) = \tau(ST).
\end{eqnarray*}
\end{proof}

\section{ TRACES OP\'ERATORIELLES}

We recall here J. Dixmier's definition of a ``$\frk Z$-trace''
\cite{Dix}. We begin by paraphrasing (and translating) Dixmier's discussion
of the formal set-up. 

Let $\frk A$ be a von Neumann algebra and let $\frk Z$ be a von Neumann
subalgebra of the centre of $\frk A.$ In this section we fix a locally 
compact Hausdorff space $X$, a positive measure $\nu$ on $X$, and an 
isomorphism of $L^\infty(X,\nu)$ with $\frk Z$ (see th\'eor\`eme 1 of I.7
of \cite{Dix}). Then $\frk Z_+$ is embedded in the set, $\hat{\frk Z}_+$,
of nonnegative measureable functions on $X$ which are not necessarily 
finite-valued. Of course, we identify functions in $\hat{\frk Z}_+$ which
are equal $\nu$-almost everywhere. As mentioned before, any bounded 
increasing net in $\frk Z_+$ has a supremum in $\frk Z_+$. It is clear 
that the same thing holds for the set $\hat{\frk Z}_+$.

\begin{defn}\label{ZTr}
With the above notation, we define a $\frk Z$-trace on $\frk A_+$
to be a mapping\\$\phi:\frk A_+\to\hat{\frk Z}_+$ which satisfies:

(i)\;If $S,T\in\frk A_+$ then $\phi(S+T)=\phi(S)+\phi(T),$

(ii)\;If $S\in\frk A_+$ and $T\in\frk Z_+$ then
$\phi(TS)=T\phi(S),$ and

(iii)\;If $S\in\frk A_+$ and $U$ is a unitary in $\frk A$ then
$\phi(USU^*)=\phi(S).$

\noindent We call $\phi$ {\bf faithful} if $S\in\frk A_+\;and\;\phi(S)=0
\Longrightarrow S=0.$

\noindent We call $\phi$ {\bf finite} if $\phi(S)\in\frk Z_+$ for all 
$S\in\frk A_+.$

\noindent We call $\phi$ {\bf semifinite} if for each nonzero $S\in\frk A_+$ 
there exists a nonzero $T\in\frk A_+$ with $T\leq S$ and $\phi(T)\in\frk Z_+.$

\noindent We call $\phi$ {\bf normal} if for every bounded increasing net 
$\{S_{\alpha}\}$ in $\frk A_+$ with supremum $S\in\frk A_+$, $\phi(S)$ is the 
supremum of the increasing net $\{\phi(S_{\alpha})\}$ in $\hat{\frk Z}_+.$
\end{defn}

We now show that if $\cal A$ is a $\frk Z$-Hilbert algebra then there is a 
natural $\frk Z$-trace on the von Neumann algebra $\cal U(\cal A)$ constructed
in the usual way.

\begin{thm}\label{HalgtoDix}(cf., Th\'eor\`eme 1, I.6.2 of \cite{Dix})
Let $\cal A$ be a $\frk Z$-Hilbert algebra over the abelian von Neumann
algebra $\frk Z$ and let $\tau:\cal I=\pi(\cal A_b^2)\to\frk Z$ be the 
tracial mapping defined in Theorem \ref{Tr}. Then $\tau$ restricted to $\cal I_+$
extends to a mapping $\bar\tau:\cal U(\cal A)_+\to\hat{\frk Z}_+$ via:
$$\bar\tau(T)=\sup\{\tau(S)\,|\,S\in\cal I_+,\,S\leq T\}.$$
This extension is a faithful, normal, semifinite $\frk Z$-trace in
the sense of Dixmier and moreover, $$\{T\in\cal U(\cal A)_+\,|\,\bar\tau(T)\in
\frk Z_+\}=\cal I_+.$$
Clearly, $\bar\tau$ is the unique normal extension of $\tau.$
\end{thm}

\begin{proof}[\bf Proof]
This proof is similar in outline to Th\'eor\`eme 1, I.6.2 of \cite{Dix}.
However, there are many complications (some subtle) in this degree of
generality. At least it is clear that $\bar\tau$ extends $\tau.$

(i) {\bf $\bar\tau$ is additive.} Trivially, we have for $T_1,T_2\in\cal U(\cal A)_+$
$$\bar\tau(T_1)+\bar\tau(T_2)\leq \bar\tau(T_1+T_2).$$
On the other hand, let $T=T_1+T_2$ for $T_1,T_2\in\cal U(\cal A)$. Then by
p. 86 of \cite{Dix}, $T_1^{1/2}=AT^{1/2}$ and $T_2^{1/2}=BT^{1/2}$ for
$A,B\in\cal U(\cal A)$ and $E=A^*A+B^*B$ is the range projection of $T$.
Now, if $0\leq S\leq T$ with $S\in\cal M_+$ then
$$ASA^*\leq ATA^*=(AT^{1/2})(AT^{1/2})^*=T_1^{1/2}T_1^{1/2}=T_1,$$
and similarly, $BSB^*\leq T_2.$ Since $\cal I$ is an ideal, $ASA^*$ and $BSB^*$
are in $\cal I_+$. Thus, since $ES=S$,
\begin{eqnarray*}
\tau(S) &=& \tau(ES)=\tau(A^*AS) +\tau(B^*BS)\\
&=& \tau(ASA^*)+\tau(BSB^*)\\
&\leq& \bar\tau(T_1)+\bar\tau(T_2).
\end{eqnarray*}
Taking the supremum over all such $S$ yields the other inequality:
$$\bar\tau(T)\leq\bar\tau(T_1)+\bar\tau(T_2).$$

(ii) {\bf $\bar\tau$ is $\frk Z_+$-linear.} Unlike the scalar case this is not
completely trivial.

If $E$ is a projection in $\frk Z_+$ and $T\in\cal U(\cal A)_+$, then one
easily checks that:
$$(S\in \cal I_+\;and\;S\leq ET)\Longleftrightarrow(S=ER\;for\;R\in\cal I_+\;
with\;R\leq T).$$
Applying the definition of $\bar\tau$, we get $\bar\tau(ET)=E\bar\tau(T).$

Now, if $z_0\in\frk Z_+$ and if there exists $z_1\in\frk Z_+$ with 
$z_1z_0=E$ the range projection of $z_0$ then again one shows that:
$$(S\in\cal I_+\;and\;S\leq z_0T)\Longleftrightarrow(S=z_0R\;for\;R\in\cal I_+\;
with\;R\leq T).$$
Hence, $\bar\tau(z_0T)=z_0\bar\tau(T)$ if $z_0$ is bounded away from $0$
on its range projection.

Now for an arbitrary $z_0\in\frk Z_+$ and $T\in\cal U(\cal A)_+$ we work
pointwise on $X$ where we have identified $\frk Z=L^\infty(X.\nu).$ So, fix
$x\in X$. There are two cases. If $z_0(x)=0$, then 
$[z_0\bar\tau(T)](x)=z_0(x)\bar\tau(T)(x)=0.$ On the other hand, if $S\leq z_0T$
and $S\in\cal I_+$ then $S=ES$ where $E$, the range projection of $z_0$, 
satisfies $E(x)=0,$ then:
$$\tau(S)(x)=\tau(ES)(x)=(E\tau(S))(x)=E(x)\tau(S)(x)=0.$$
Taking the supremum over such $S$ we get $\bar\tau(z_0T)(x)=0$ That is,
$$if\;z_0(x)=0,\; then\; \bar\tau(z_0T)(x)=[z_0\bar\tau(T)](x)=0.$$

In the second case, $z_0(x)>0,$ so that we can write $z_0=z_1+z_2$ in 
$\frk Z_+$ where $z_1$ is bounded away from $0$ on its support (which
contains $x$) and $z_2(x)=0.$ Then:
\begin{eqnarray*}
\bar\tau(z_0T)(x) &=& [\bar\tau(z_1T)+\bar\tau(z_2T)](x)
=[z_1\bar\tau(T)+\bar\tau(z_2T)](x)\\
&=& z_1(x)\bar\tau(T)(x)+\bar\tau(z_2T)(x)
= z_0(x)\bar\tau(T)(x)+0
= [z_0\bar\tau(T)](x).
\end{eqnarray*}
Hence, $\bar\tau(z_0T)=z_0\bar\tau(T).$ 

(iii) {\bf $\bar\tau$ is unitarily invariant.} This follows easily from
Theorem \ref{Tr} part (2).

(iv) {\bf $\bar\tau$ is faithful.} If $\bar\tau(T)=0$, then the only 
$S\in\cal I_+$ with $S\leq T$ is $S=0.$ However, if $\{\pi(\xi_{\alpha})\}$
is a net in $\pi(\cal A_b)$ converging ultraweakly to $1$ and having norm
$\leq 1$ then:
$$0\leq T^{1/2}\pi(\xi_{\alpha}\xi_{\alpha}^*)T^{1/2} \leq T.$$
But, $T^{1/2}\pi(\xi_{\alpha}\xi_{\alpha}^*)T^{1/2}$ is in $\cal I_+$ and
converges ultraweakly to $T$. Hence, $T=0.$

(v) {\bf $\bar\tau$ is semifinite.} This is the same argument as in part (iv).

(vi) {\bf $\{T\in\cal U(\cal A)_+\,|\,\bar\tau(T)\in\frk Z_+\}=\cal I_+.$}
Clearly, $\cal I_+$ is contained in this set. So, suppose $\bar\tau(T)
=z\in\frk Z_+.$ We apply Proposition \ref{Tbdd}. That is, let $\xi\in \cal A_b$
satisfy $\|\pi(\xi)\|\leq 1.$ Then,
$$\pi\left[(T^{1/2}(\xi))(T^{1/2}(\xi))^*\right]=T^{1/2}\pi(\xi\xi^*)T^{1/2}
\leq T$$
and so,
$$\tau\left(\pi\left[(T^{1/2}(\xi))(T^{1/2}(\xi))^*\right]\right)\leq 
\bar\tau(T)=z.$$
But,
$$\tau\left(\pi\left[(T^{1/2}(\xi))(T^{1/2}(\xi))^*\right]\right)=
\bra (T^{1/2}(\xi))^*,(T^{1/2}(\xi))^*\ket=\bra T^{1/2}(\xi),T^{1/2}(\xi)\ket.$$
Therefore, by Proposition \ref{Tbdd}, $T^{1/2}=\pi(\eta)$ for some 
$\eta\in\cal A_b$ and so $T=\pi(\eta^*\eta)\in\cal I_+.$

(vii) {\bf $\bar\tau$ is normal.} We first show that $\bar\tau$ satisfies the
normality condition when the relevant operators are all in $\cal I_+.$ That is,
suppose that $\{\pi(\xi_{\alpha}^*\xi_{\alpha})\}$ is an increasing net in
$\cal I_+$ with least upper bound $\pi(\xi^*\xi)$ also in $\cal I_+.$ 
Now for any $\eta\in\cal A_b$ we have by the polar decomposition theorem
that $|\pi(\eta)|=V\pi(\eta)=\pi(V\eta)$ and that $V\eta\in\cal A_b.$ Hence,
for any $\eta\in\cal A_b,$ 
$$\pi(\eta^*\eta)=|\pi(\eta)|^2=\pi((V\eta)^2)\;and\;\pi(V\eta)\geq 0.$$
Thus we can assume that $\xi_{\alpha}$ and $\xi$ are self-adjoint
and that $\pi(\xi_{\alpha})\geq 0$ and $\pi(\xi)\geq 0.$ Then,
$\pi(\xi_{\alpha})=(\pi(\xi_{\alpha}^*\xi_{\alpha}))^{1/2}$ and
$\pi(\xi)=(\pi(\xi^*\xi))^{1/2}.$

Now, $\pi(\xi_{\alpha}^2)\to\pi(\xi^2)$ in the strong operator topology
by Vigier's theorem and by the proof of Th\'eor\`eme 1 of I.6.2 of \cite{Dix}
we also have $\pi(\xi_{\alpha})\to\pi(\xi)$ in the strong operator topology.
As the square root function is operator monotone, this implies that
$\pi(\xi)=\sup_{\alpha}\pi(\xi_{\alpha}).$

It easily follows that $\|\xi_{\alpha}\|\leq \|\xi\|$ for all $\alpha.$
Since $\cal{H}_{\cal A}$ is a dual space, we can find a subnet $\{\xi_{\beta}\}$
which converges weak$*$ to some $\zeta\in \cal{H}_{\cal A}.$ To see that
$\zeta=\xi$, let $\lambda,\mu\in\cal A_b$ then by Proposition \ref{conv}:
$$\bra\zeta,\lambda\mu\ket=\lim_{\beta}\bra\xi_{\beta},\lambda\mu\ket
=\lim_{\beta}\bra\pi(\xi_{\beta})\mu^*,\lambda\ket=
\bra\pi(\xi)\mu^*,\lambda\ket=\bra\xi,\lambda\mu\ket.$$
Thus, $\zeta$ and $\xi$ define the same $\frk Z$-valued mapping
on $\cal A_b^2\supseteq\cal A^2$ and therefore the same mapping on $\cal A$.
That is, $\zeta=\xi.$

Now, since $\tau$ is positive we have
$$\tau(\pi(\xi^*\xi))\geq\sup_{\alpha}\tau(\pi(\xi_{\alpha}^*\xi_{\alpha})).$$
On the other hand, by Kaplansky's Cauchy-Schwarz inequality \cite{K}
(which holds since $\frk Z$ is abelian) we have:
$$|\bra\xi_{\beta},\xi\ket|\leq
\bra\xi_{\beta},\xi_{\beta}\ket^{1/2}\bra\xi,\xi\ket^{1/2}\;for\;all\;\beta.$$
Since $\xi$ and $\xi_{\beta}$ are self-adjoint it is seen
that $\bra\xi_{\beta},\xi\ket$ is also self-adjoint and so in fact
$$\bra\xi_{\beta},\xi\ket\leq
\bra\xi_{\beta},\xi_{\beta}\ket^{1/2}\bra\xi,\xi\ket^{1/2}\;for\;all\;\beta.$$
Hence,
\begin{eqnarray*}
\bra\xi,\xi\ket &=& uw\lim_{\beta}\bra\xi_{\beta},\xi\ket\leq\sup_{\beta}
\bra\xi_{\beta},\xi_{\beta}\ket^{1/2}\bra\xi,\xi\ket^{1/2}\\
&\leq&(\sup_{\alpha}\bra\xi_{\alpha},\xi_{\alpha}\ket^{1/2})
\bra\xi,\xi\ket^{1/2}.
\end{eqnarray*}
Since $\frk Z$ is abelian this implies that 
$$\bra\xi,\xi\ket^{1/2}\leq\sup_{\alpha}\bra\xi_{\alpha},\xi_{\alpha}\ket^{1/2}
\;and\;so\;\bra\xi,\xi\ket\leq\sup_{\alpha}\bra\xi_{\alpha},\xi_{\alpha}\ket.$$
That is,
$$\tau(\pi(\xi^*\xi))\leq
\sup_{\alpha}\tau(\pi(\xi_{\alpha}^*\xi_{\alpha})),\;and\;so\;they\;are\;
equal.$$

Now, we let $\{T_{\alpha}\}$ be an increasing net in $\cal U(\cal A)_+$ with
supremum $T\in\cal U(\cal A)_+.$ We define 
$f=\sup_{\alpha}(\bar\tau(T_{\alpha}))$, in $\hat{\frk Z}_+.$ 
Let $E=\{x\in X\,|\,f(x)=+\infty\}.$ 
Since $\bar\tau(T_{\alpha})\leq\bar\tau(T)$ for all $\alpha$,
we have $f\leq\bar\tau(T).$ Hence $f$ agrees with $\bar\tau(T)$ on the 
measureable set $E.$ The complement of $E$ is the countable union of the
measureable sets $E_N:=\{x\in X\,|\,f(x)\leq N\}$, so it suffices to see that
$f$ agrees with $\bar\tau(T)$ (almost everywhere) on each $E_N.$ To this end,
let $z_N$ be the characteristic function of $E_N.$ Clearly, $z_N\in
\frk Z_+$ and $z_{N}T=\sup_{\alpha}z_{N}T_{\alpha}$ in $\cal U(\cal A)_+.$
Now, for each $\alpha$, 
$$\bar\tau(z_{N}T_{\alpha})=z_{N}\bar\tau(T_{\alpha})\leq z_{N}f\leq Nz_N\in
\frk Z_+.$$
So, by an earlier part of the proof, there exists $\xi_{\alpha}=\xi_{\alpha}^*
\in\cal A_b$ with 
$z_{N}T_{\alpha}=\pi(\xi_{\alpha}^*\xi_{\alpha})$ and 
$\bra\xi_{\alpha},\xi_{\alpha}\ket\leq Nz_N.$
Now, for each $\eta\in\cal A_b$ with $\|\pi(\eta)\|\leq 1$ we have:
\begin{eqnarray*}
\bra z_{N}T^{1/2}\eta,z_{N}T^{1/2}\eta\ket &=& \bra z_{N}T\eta,\eta\ket
=\lim_{\alpha}\bra z_{N}T_{\alpha}\eta,\eta\ket=
\lim_{\alpha}
\bra\xi_{\alpha}\eta,\xi_{\alpha}\eta\ket\\
&=& \lim_{\alpha}\bra\eta^*\xi_{\alpha},\eta^*\xi_{\alpha}\ket=
\lim_{\alpha}\bra\pi(\eta\eta^*)\xi_{\alpha},\xi_{\alpha}\ket
\leq \sup_{\alpha}\bra\xi_{\alpha},\xi_{\alpha}\ket\leq Nz_N.
\end{eqnarray*}
Therefore, by Proposition \ref{Tbdd} there exists a $\zeta\in\cal A_b$ with
$z_{N}T^{1/2}=\pi(\zeta).$ Moreover,
$$\sup_{\alpha}\pi(\xi_{\alpha}^*\xi_{\alpha})=\sup_{\alpha}z_{N}T_{\alpha}=
z_{N}T=\pi(\zeta^*\zeta).$$
Hence by the first part of the proof of normality of $\bar\tau$,
$$\bar\tau(z_{N}T)=\bar\tau(\pi(\zeta^*\zeta))=\sup_{\alpha}\bar\tau(\pi(
\xi_{\alpha}^*\xi_{\alpha}))=\sup_{\alpha}\bar\tau(z_{N}T_{\alpha}).$$
That is, for $x\in E_N$ we have:
\begin{eqnarray*}
f(x) &=& (z_{N}f)(x)=(z_{N}\sup_{\alpha}\bar\tau(T_{\alpha}))(x)\\
&=& (\sup_{\alpha}\bar\tau(z_{N}T_{\alpha}))(x)=(\bar\tau(z_{N}T))(x)\\
&=& (z_{N}\bar\tau(T))(x)=\bar\tau(T)(x)\; \text {as required}.
\end{eqnarray*}
\end{proof}

\begin{rems*}
In the above setting we want to observe that $\cal A_b$ is also a 
$\frk Z$-Hilbert algebra and that $\cal U(\cal A)=\cal U(\cal A_b)$, etc. 
It turns out that the only subtle point is the fact that 
$\cal{H}_{\cal A}=\cal{H}_{\cal A_b}!$
\end{rems*}

\begin{lemma}\label{complete}
Suppose ${\bf X}\subseteq{\bf Y}\subseteq{\bf X}^{\dagger}$ as pre-Hilbert
$\frk B$-modules where $\frk B$ is a von Neumann algebra. Then, in fact,
${\bf X}^{\dagger}={\bf Y}^{\dagger}.$
\end{lemma}

\begin{proof}[\bf Proof]
If $\theta\in{\bf X}^{\dagger}$ then  
$y\mapsto\bra\theta,y\ket_{{\bf X}^{\dagger}}:{\bf Y}\to\frk B$ is a bounded 
$\frk B$-module
map and so there is a unique $\tilde\theta\in{\bf Y}^{\dagger}$ so that:

\hspace{1.9in}$\bra\tilde\theta,\hat{y}\ket_{{\bf Y}^{\dagger}}=
\bra\theta,y\ket_{{\bf X}^{
\dagger}}\;for\;all\;y\in{\bf Y}.$\hspace{1.8in} (1)

\noindent That is, $\theta\mapsto\tilde\theta$ embeds ${\bf X}^{\dagger}$ in
${\bf Y}^{\dagger}$. We first show that this embedding preserves inner products.

Now, given $\eta\in{\bf X}^{\dagger}$, then 
$\theta\mapsto\bra\tilde\eta,\tilde\theta\ket_{{\bf Y}^{\dagger}}:
{\bf X}^{\dagger}\to\frk B$ is an element of ${\bf X}^{\dagger\dagger}=
{\bf X}^{\dagger}$ and so there exists a unique $\gamma\in{\bf X}^{\dagger}$
so that 

\hspace{1.9in}$\bra\gamma,\theta\ket_{{\bf X}^{\dagger}}=
\bra\tilde\eta,\tilde\theta\ket_{{\bf Y}^{\dagger}}
\;for\;all\;\theta\in{\bf X}^{\dagger}.$\hspace{1.8in} (2)

\noindent In particular, for all $x\in{\bf X}$ we get
$$\bra\gamma,x\ket_{{\bf X}^{\dagger}}=\bra\tilde\eta,\hat x\ket_
{{\bf Y}^{\dagger}}=\bra\eta,x\ket_{{\bf X}^{\dagger}}\;by\;equation\;(1).$$
Hence, $\gamma=\eta$, and equation (2) becomes:
$$\bra\eta,\theta\ket_{{\bf X}^{\dagger}}=
\bra\tilde\eta,\tilde\theta\ket_{{\bf Y}^{\dagger}}
\;for\;all\;\eta,\theta\in{\bf X}^{\dagger}.$$
That is, ${\bf X}^{\dagger}$ is a pre-Hilbert $\frk B$-submodule of
${\bf Y}^{\dagger}$ and we have:
$${\bf Y}\subseteq{\bf X}^{\dagger}\subseteq{\bf Y}^{\dagger}$$
as pre-Hilbert $\frk B$-modules.

Now, for each $\mu\in{\bf Y}^{\dagger}$ the map 
$\theta\mapsto\bra\mu,\tilde\theta\ket_{{\bf Y}^{\dagger}}:
{\bf X}^{\dagger}\to\frk B$ defines a unique element $\check\mu\in
{\bf X}^{\dagger}$ satisfying:
$$\bra\mu,\tilde\theta\ket_{{\bf Y}^{\dagger}}=\bra\check\mu,\theta\ket_
{{\bf X}^{\dagger}}=
\bra\tilde{\check{\mu}},\tilde\theta\ket_{{\bf Y}^{\dagger}}
\;for\;all\;\theta\in{\bf X}^{\dagger}.$$
But since ${\bf Y}\subseteq{\bf X}^{\dagger}$ we must have
$$\mu=\tilde{\check{\mu}}.$$
That is, $\tilde{}:{\bf X}^{\dagger}\to{\bf Y}^{\dagger}$ is onto.
\end{proof}

\begin{prop}\label{full}
Let $\cal A$ be a $\frk Z$-Hilbert algebra over the abelian von Neumann
algebra $\frk Z$. Then, $\cal A_b$ is also a $\frk Z$-Hilbert algebra and

\noindent (1) $\cal{H}_{\cal A_b} = \cal{H}_{\cal A},$

\noindent (2) $\cal U(\cal A_b)=\cal U(\cal A)$ and $\cal V(\cal A_b)=
\cal V(\cal A),$

\noindent (3) $(\cal A_b)_b=\cal A_b.$
\end{prop}

\begin{proof}[\bf Proof]
Since $\frk Z \subseteq \cal L(\cal{H}_{\cal A})$ and $\pi(\cal A_b)$ is a left 
ideal in $\cal L(\cal{H}_{\cal A})$, we see that $\cal A_b$ is a pre-Hilbert 
$\frk Z$-submodule of $\cal{H}_{\cal A}$ containing $\cal A.$ Hence, by the previous
lemma, $\cal{H}_{\cal A_b} = \cal{H}_{\cal A}.$ 

Thus, axioms (i), (ii), (iii), and (iv) are automatically satisfied.

That $\cal A_b$ is a $*$-algebra follows from Proposition \ref{lbdd}. Now,
axiom (v) follows from Lemma \ref{*map}. Axiom (vi) follows from part (4) of
Proposition \ref{lbdd} since $\pi(\xi^*) = \pi(\xi)^*$ for $\xi\in\cal A_b.$
Axiom (vii) follows from the definition of $\cal A_b$ and part (3) of
Proposition \ref{lbdd}.

To see axiom (viii), we first note that 
$$\cal A^2\subseteq\cal A_b^2\subseteq \cal A_b\subseteq \cal{H}_{\cal A_b} 
= \cal{H}_{\cal A}.$$
Since $\cal A^2$ is dense in $\cal A$ by definition and $\cal A$ is dense in 
$\cal{H}_{\cal A}$ by Proposition \ref{density}, it follows that $\cal A_b^2$ is dense in 
$\cal{H}_{\cal A_b}$ and hence in $\cal A_b.$

Thus, $\cal A_b$ is also a $\frk Z$-Hilbert algebra, and items (2) and (3)
follow easily.
\end{proof}

\section{$\frk Z$-HILBERT ALGEBRAS from $\frk Z$-TRACES}

Here we suppose that $\phi$ is a faithful normal semifinite $\frk Z$-trace
(in Dixmier's sense) on the von Neumann algebra $\frk A,$ where 
$\frk Z$ is a von Neumann 
subalgebra of the centre of $\frk A.$ We abuse notation and also let $\phi$
denote the unique linear extension of the original $\phi$ from 
$$\cal I_+=\{x\in\frk A\,|\,\phi(x)\in\frk Z_+\}$$
to the ideal $\cal I = span \cal I_+,$ defined in Proposition 1 of 
III.4.1 of \cite{Dix}. Then, by I.1.6 of \cite{Dix} the space 
$\cal A = \{x\in\frk A\,|\,\phi(x^*x)\in\frk Z_+\}$ is an ideal in $\frk A$
with $\cal A^2=\cal I.$

\begin{prop}\label{prods}
With the above hypotheses, the ideal $$\cal A=\{x\in\frk A\,|\,\phi(x^*x)\in
\frk Z_+\}$$ is a $\frk Z$-Hilbert algebra, with the $\frk Z$-valued inner
product $\bra x,y\ket = \phi(x^*y).$
\end{prop}

\begin{proof}[\bf Proof]
Since $\cal A$ is an ideal in $\frk A$ it is certainly a right $\frk Z$-module.
Axiom (i) is just the statement that $\phi$ is faithful. Axiom (ii) follows 
since the extended $\phi$ is clearly self-adjoint. Axiom (iii) follows as 
the original $\phi$ is $\frk Z_+$-linear.

To see that Axiom (iv) holds requires a little thought. First, it is clear that 
$span(\phi(\cal A^2))$ is an ideal in $\frk Z.$ Therefore, its u.w.-closure
is an ideal in $\frk Z$ of the form $E\frk Z$ for some projection $E\in\frk Z.$
If $(1-E)\neq 0$ then since $\phi$ is semifinite there exists $x\in\frk A_+$
with $0\neq x\leq(1-E)$ and $\phi(x)\in\frk Z_+$ so that $x^{1/2}\in \cal A$. 
But then,
$$0\neq\phi(x)=\phi((1-E)x)=(1-E)\phi(x)$$
lies in $E\frk Z$, a contradiction. Hence $E=1$ and the span of the inner
products is u.w.-dense in $\frk Z.$

Axiom (v) follows from the tracial property of Proposition 1 of III.4.1 of 
\cite{Dix}. Axiom (vi) is trivial, and Axiom (vii) is proved as in Example \ref{Hilbex}.

To see Axiom (viii) we first show that $\cal A$ is u.w.-dense in $\frk A.$
Now the ultraweak closure of $\cal A$ is an u.w. closed ideal in $\frk A$
and so has the form $F\frk A$ for some projection $F$ in $Z(\frk A)$.
If $(1-F)\neq 0$ then since $\phi$ is semifinite there exists $y\in\frk A_+$
with $0\neq y\leq(1-F)$ and $\phi(y)\in \frk Z_+$ so that $y^{1/2}\in\cal A$.
But then $y\in\cal A$ and so $y\leq F$, a contradiction as $y\neq 0.$
Thus $F=1$ and $\cal A$ is u.w.-dense in $\frk A.$

Now, given $\omega\geq 0$ in the predual of $\frk Z$, we have that 
$\phi_{\omega}:=\omega\circ\phi$ is a normal, semifinite trace on $\frk A$ by 
Proposition 2 of III.4.3 of \cite{Dix}. Moreover, the GNS Hilbert space of the 
normal representation $\pi_{\omega}$ of $\frk A$ induced by $\phi_{\omega}$ is 
the same as the Hilbert space $\cal{H}_{\omega}$ of section 3. For $a,b\in\cal A$,
we have $\pi_{\omega}(a)(b+N_{\omega})=ab+N_{\omega}.$ Since $\pi_{\omega}$
is normal, $\pi_{\omega}(\cal A)$ is u.w.-dense in $\pi_{\omega}(\frk A).$
Therefore, it is also s.o.-dense and hence given any $b\in\cal A$ and 
$\epsilon>0$ there exists $a\in\cal A$ with
$$\|\pi_{\omega}(a)(b+N_{\omega}) - (b+N_{\omega})\|_{\omega} < \epsilon.$$
That is, $\|ab-b\|_{\omega} < \epsilon,$ and Axiom (viii) is satisfied.
\end{proof}

In this setting, each $x\in\frk A$ defines an operator, $\tilde x$, on the ideal
$\cal A = \{a\in \frk A\,|\,\phi(a^*a)\in\frk Z_+\}$ via $\tilde x(a)=xa.$
Clearly, $\tilde x$ is $\frk Z$-linear, and it is easy to check that
$\tilde x$ is a bounded $\frk Z$-module map on $\cal A,$ and therefore 
extends uniquely to a bounded module map on $\cal{H}_{\cal A},$ also denoted by
$\tilde x.$ As left multiplications commute with right multiplications,
we see that $\tilde x\in(\pi^{\prime}(\cal A))^{\prime} = \cal U(\cal A),$
by the Commutation Theorem \ref{Comm}.

\begin{lemma}\label{semifin}
Let $\cal A$ be an u.w.-dense $*$-ideal in the von Neumann algebra $\frk A$.
Then, each $T\in\frk A_+$ is the increasing limit of a net in $\cal A_+.$ 
\end{lemma} 

\begin{proof}[\bf Proof]
It follows from the proof of Theorem 1.4.2 of \cite{Ped} that
$\{a\in\cal A_+\,|\,\|a\|<1\}$ is an increasing net in the usual ordering
of positive elements and hence converges in $\frk A_+$ by Vigier's Theorem. 
By the Kaplansky Density Theorem there is a subnet 
of this one converging ultraweakly to the identity in $\frk A$, and therefore
this net converges ultraweakly to $1\in\frk A.$
  
Thus, if $T\in\frk A_+$, the net $\{T^{1/2}aT^{1/2}\,|\,a\in\cal A_+\;
{\rm and}\;\|a\|<1\}$ is an
increasing net in $\cal A_+$ converging ultraweakly to $T$.
\end{proof}

\begin{thm}\label{isom}
Let $\phi$ be a faithful normal semifinite $\frk Z$-trace
on the von Neumann algebra $\frk A,$ where $\frk Z$ is a von Neumann 
subalgebra of the centre of $\frk A.$ Let $\cal A = 
\{a\in \frk A\,|\,\phi(a^*a)\in\frk Z_+\}$ be the corresponding $\frk Z$-Hilbert
algebra. Then the mapping $x\mapsto\tilde x :\frk A\to\cal U(\cal A)$ is
an isomorphism of von Neumann algebras.
\end{thm}

\begin{proof}[\bf Proof]
It is clear the the mapping is a $*$-homomorphism. Since $\cal A$ is u.w.-dense
in $\frk A$, the mapping is also one-to-one. Hence, it suffices to see that the
mapping is onto $\cal U(\cal A).$ So, let $T\in\cal U(\cal A)_+.$ Since
$\pi(\cal A)$ is an u.w.-dense $*$-ideal in $\cal U(\cal A)$, there
is a net, $\{b_{\alpha}\}$ in $\cal A_+$ with $\pi(b_{\alpha})$
increasing to $T$ in $\cal U(\cal A)\subseteq\cal L(\cal{H}_{\cal A}).$
Since, $\{b_{\alpha}\}$ is an increasing net in $\cal A_+\subseteq\frk A_+$
bounded by $\|T\|$, it converges to an element $x\in\frk A_+.$ To see that
$\tilde x=T$ it suffices to see that $\omega(\bra Ta,c\ket)=
\omega(\bra xa,c\ket)$
for all $a,c\in\cal A$ and $\omega\geq 0$ in $\frk Z_*.$

Now, since $\omega\circ\phi$ is a normal scalar trace on $\frk A$ by
Proposition 2 of III.4.3 of \cite{Dix} and since $ca^*\in\cal A^2=\cal I$
is contained in the ideal of definition of this normal scalar trace,
the map $$y\mapsto\omega\circ\phi(yca^*):\frk A\to\comp$$
is a normal (and so u.w.-continuous) linear functional on $\frk A.$
Hence,
\begin{eqnarray*}
\omega(\bra xa,c\ket) &=& \omega(\phi(a^*xc)) = \omega(\phi(xca^*))\\
&=& \lim_{\alpha}\omega(\phi(b_{\alpha}ca^*)) = \lim_{\alpha}\omega
(\bra\pi(b_{\alpha})a,c\ket).
\end{eqnarray*}
But, by Proposition \ref{conv} part (2) this last term equals $\omega(\bra Ta,c\ket)$
since $\pi(b_{\alpha})\stackrel{uw}{\to} T.$ 
\end{proof}

\section{The $\frk Z$-TRACE on the CROSSED PRODUCT von NEUMANN ALGEBRA}

Let $(A,Z,\tau,\alpha)$ be a $4$-tuple as in Section 1. We also assume that $Z$ has a faithful 
state, $\omega$ to apply Proposition \ref{extension} so that 
$\bar\omega=\omega\circ\tau$ is a faithful tracial state on $A$
and representing $A$ on the GNS Hilbert space $\cal{H}_{\bar\omega}$ we obtain
$\frk A= A^{\prime\prime}$ and $\frk Z=Z^{\prime\prime}$
and a $\frk Z$-trace 
$\bar\tau : \frk A \to \frk Z$ extending $\tau$ and an extension of
$\alpha$ to an ultraweakly continuous action $\bar\alpha :\real \to
Aut(\frk A)$ which leaves $\bar\tau$ invariant.

\begin{rem}\label{AZ}
The following construction of the $\frk Z$-trace on the crossed product algebra works
in much greater generality: the action of $\real$ on $A$ leaving $\tau$
invariant can be replaced by an action of a unimodular locally compact group
$G$ on $A$ leaving $\tau$ invariant. We leave the minor modifications to 
the interested reader. All the results up to the end of section 8.5 work
in this generality.
\end{rem}
 
We let $A_{\frk Z}$ denote the $C^*$-subalgebra of $\frk A$ generated by $A$ and
$\frk Z$. Clearly,
$$A_{\frk Z}=
\left\{\sum_{i=1}^n a_{i}z_{i}| a_{i}\in A,
 z_{i} \in \frk Z\right\}^{-\|\cdot\|}.$$
 
It is clear that:
\vspace{.1in}\\
(1)\;$A_{\frk Z}$ contains $A$ and $\frk Z$ and is therefore ultraweakly dense in 
$\frk A$.\vspace{.1in}\\
(2)\; $\bar\tau : A_{\frk Z} \to \frk Z$ is a faithful, unital 
$\frk Z$-trace, and \vspace{.1in}\\
(3)\; $\bar\alpha : \real\to Aut(A_{\frk Z})$ is a {\bf norm}-continuous
action on $A_{\frk Z}$ leaving $\bar\tau$ invariant and leaving $\frk Z$ pointwise fixed.

\begin{tech*}
The introduction of this hybrid algebra $A_{\frk Z}$ allows us to treat $\frk Z$
as scalars and use {\bf norm}-continuity in most of our calculations. This
permits the use of $C^*$-algebra crossed products and is a considerable 
simplification. We note also that one cannot simply use the {\bf space}
of {\bf norm}-continuous functions $C_c(\real,\frk A)$ below since
$\bar\alpha$-twisting the multiplication might take us out of the realm of
{\bf norm}-continuity. However, as a vector space (and pre-Hilbert 
$\frk Z$-module), $C_c(\real,\frk A)$ will have its uses.
\end{tech*}

With this set-up and notation, we define:
\begin{defn}\label{ZHilb}
$$\cal A = C_c(\real,A_{\frk Z}),$$
the space of {\bf norm}-continuous compactly supported functions from $\real$ to
$A_{\frk Z}$. We require {\bf norm}-continuity so that $\cal A$ becomes 
a $*$-algebra with the usual 
$\bar\alpha$-twisted multiplication:
$$x\cdot y(s) = \int x(t)\bar\alpha_t(y(s-t))dt,$$
and involution:
$$x^*(s) = \bar\alpha_s((x(-s))^*).$$
\end{defn}

Moreover, $\cal A$ becomes a (right) pre-Hilbert $\frk Z$-module with the inner 
product:
$$\bra x,y\ket = \int \bar\tau(x(s)^*y(s))ds$$ and $\frk Z$-action:
$$(xz)(s)=x(s)z.$$ 

Axioms (i), (ii), and (iii) are routine calculations. To see axiom (iv)
we observe that the set of inner products $\{\bra x,y\ket\; |\; x,y\in \cal A\}$ is
exactly equal to $\frk Z$. It comes as no surprise that $\cal A$ is, in fact,
a $\frk Z$-Hilbert algebra.

\begin{rem*}
We will also have occasion to use the completion of $\cal A$ in the vector-valued Banach
$L^2$ norm:
$$\|x\|_2 =\left(\int\|x(s)\|^2ds\right)^{1/2}.$$
We define this completion to be $L^2({\bf R},A_{\mathfrak{Z}})$ and observe that since 
$\|x\|_{\cal{A}} \leq \|x\|_2 $, we have a natural inclusion: 
$$L^2({\bf R},A_{\mathfrak{Z}})\hookrightarrow 
\cal{A}^{-\|\cdot\|_{\cal{A}}}\subset \cal{H}_\cal{A}.$$ 
\end{rem*}

\begin{prop}\label{ZHilb2}
With the above inner product and $\frk Z$-action, the $*$-algebra $\cal A$
is a $\frk Z$-Hilbert algebra.
\end{prop}

\begin{proof}[\bf Proof]
Axioms (v) and (vi) are routine calculations. Since $\cal A$ contains all the 
scalar-valued functions in $C_c(\real)$, it is easy to see that $\cal A^2$
is dense in $\cal A$ in the vector-valued $L^2$ norm: 

Since $\|x\|_{\cal{A}}\leq \|x\|_2$ , $\cal A^2$
is dense in $\cal A$ in the $\frk Z$-Hilbert algebra norm and so axiom (viii) is
satisfied by the Remark after Definition \ref{Hilbalg}.

Axiom (vii) requires a little more thought. We will show that the left regular
representaion of the $*$-algebra $\cal A$ on the pre-Hilbert $\frk Z$-module
$\cal A$ is the integrated form of a covariant pair of representations 
$(\pi_{\cal A},U)$ of the system $(A_{\frk Z},\real,\bar\alpha)$ inside the 
von Neumann algebra, $\cal L(\cal{H}_{\cal A}).$ To this end we represent $A_{\frk Z}$
on the $\frk Z$-module $\cal A = C_c(\real,A_{\frk Z})$ via:
$$[\pi_{\cal A}(a)x](s) = ax(s)\;\;for\; a\in A_{\frk Z},\;x\in\cal A,\;
s\in\real.$$

Similarly, we represent $\real$ on $\cal A$ via:
$$[U_t(x)](s) = \bar\alpha_t(x(s-t))\;\;for\;t,s\in \real,\;x\in\cal A.$$

One easily checks that these are representations as bounded, adjointable 
$\frk Z$-module mappings. Now, for fixed $x\in\cal A$ the map $t\mapsto U_t(x)$ 
is $\|\cdot\|_2$-norm continuous and so $\|\cdot\|_{\cal A}$-norm continuous:
by item (3) of Proposition \ref{conv} this easily implies that 
$$t\mapsto U_t:\real\to\cal L(\cal{H}_{\cal A})$$
is an ultraweakly continuous representation. Morever, the following are easily 
verified:

(1)\;$\|\pi_{\cal A}(a)\|\leq \|a\|$ for $a\in A_{\frk Z},$

(2)\;$\bra U_t(x),U_t(y)\ket = \bra x,y\ket$ for $t\in\real,\;\;x,y\in\cal A,$

(3)\;$\pi_{\cal A}(a)^* = \pi_{\cal A}(a^*)$ and $U_t^* = U_{-t}$ for
$a\in A_{\frk Z},\;\;t\in\real,$ and

(4)\;$U_t\pi_{\cal A}(a)U_t^* = \pi_{\cal A}(\bar\alpha_t(a))$ for $t\in\real$
and $a\in\ A_{\frk Z}.$ This is the covariance condition.

Combining this covariant pair of representations of the system,
$(A_{\frk Z},\real,\bar\alpha)$ in $\cal L(\cal A)$ with the $*$-monomorphism
embedding $\cal L(\cal A)\hookrightarrow \cal L(\cal{H}_{\cal A})$ (by Corollary 3.7 
of \cite{Pa}) we obtain a representation $\pi_{\cal A}\times U$ of the 
$C^*$-algebra 
$A_{\frk Z}\rtimes\real$ in the von Neumann algebra $\cal L(\cal{H}_{\cal A})$. One 
then easily checks that for $x\in\cal A\subset A_{\frk Z}\rtimes\real$ and 
$y\in\cal A\subset \cal{H}_{\cal A}$ that:
$$\left[(\pi_{\cal A}\times U)(x)(y)\right](s) = 
\int x(t)\bar\alpha_t(y(s-t))dt = (x\cdot y)(s).$$
That is, left-multiplication by $x$ on the $\frk Z$-module $\cal A$ is bounded
in the $\frk Z$-module norm and axiom (vii) is satisfied.
\end{proof}

\begin{lemma}\label{technical} 
If $\cal A = C_c(\real, A_{\frk Z})$ as above, then the following
hold.\\
(1) The norm-decreasing embedding: 
$(\cal A,\|\cdot\|_2)\to (\cal{H}_{\cal A},\|\cdot\|_{\frk Z})$ extends by continuity
to a norm-decreasing embedding of  
$L^2(\real,A_{\frk Z})$ into $\cal{H}_{\cal A}.$ 
Moreover, $L^2(\real,A_{\frk Z})$ is a $\frk Z$-module and the $\frk Z$-valued inner product on
$\cal{H}_{\cal A}$ restricts to $L^2(\real,A_{\frk Z})$ so that it is, in fact,  a pre-Hilbert 
$\frk Z$-module.  \\
(2) If
$x\in L^2(\real,A_{\frk Z})\subseteq \cal{H}_{\cal A}$ and $y\in \cal A$ then in the
$\frk Z$-Hilbert algebra notation, the element:
$$\pi(x)y:=\pi^{\prime}(y)x \in \cal{H}_{\cal A}$$ is identical to the element
$x\cdot y\in L^2(\real,A_{\frk Z})$ given by the twisted convolution:
$$(x\cdot y)(s) = \int x(t)\bar\alpha_t(y(s-t))dt.$$
(3) If $x , y \in L^2(\real,A_{\frk Z})$ and if $\pi(x)$ and $\pi(y)$
are bounded, then the operator $\pi(x)^*\pi(y)$ is in the ideal of
definition of the $\frk Z$-trace, $\sigma$ on $\cal U(\cal A)$, and
$$\sigma[\pi(x)^*\pi(y)]=\bra x,y\ket=
\int\bar\tau(x(t)^*y(t))dt.$$
\end{lemma}

\begin{proof}[\bf Proof]
The first statement of item (1) follows trivially from the inequality
$\|x\|_{\cal{A}}\leq \|x\|_2.$\\
\hspace*{.2in} To see the second statement of item (1),
suppose $\{x_n\}$ is a sequence in $\cal A$ which is Cauchy in the $\|\cdot\|_2$ norm and that $z\in\frk Z.$
Then $\|x_n z-x_m z\|_2 \leq\|x_n-x_m\|_2 \|z\|\to 0,$ so that $L^2(\real,A_{\frk Z})$ is a $\frk Z$-module.
Similarly, if  $\{x_n\}$ and  $\{y_n\}$ are sequences in $\cal A$ which are Cauchy in the $\|\cdot\|_2$ norm,
then by the Cauchy-Schwarz inequality:
\begin{eqnarray*}\|\bra x_n,y_n\ket - \bra x_m,y_m\ket\|&=& \|\bra x_n -x_m,y_n\ket - \bra x_m,y_m -y_n\ket\|\\
&\leq& \|x_n-x_m\|_{\cal A} \|y_n\|_{\cal A} + \|x_m\|_{\cal A} \|y_m -y_n\|_{\cal A}\\
&\leq& \|x_n-x_m\|_2 \|y_n\|_2 + \|x_m\|_2 \|y_m -y_n\|_2.
\end{eqnarray*}
Therefore, the $\frk Z$-valued inner product on $\cal{H}_{\cal{A}}$ restricts to a $\frk Z$-valued inner product on
$L^2(\real,A_{\frk Z}).$\\ 
\hspace*{.2in}To see the item (2), 
let $\{x_n\}$
be a sequence in $\cal A$ with $\|x_n - x\|_2\to 0$. Then:
$$\|x_n\cdot y - x\cdot y\|_{\cal A}\leq \|x_n\cdot y - x\cdot y\|_2\leq
\|x_n - x\|_2\|y\|_1\to 0.$$
On the other hand, since $x_n$ and $y$ are both in  
$\cal A$ we have that $\pi^{\prime}(y)x_n = x_n\cdot y$ by definition and so:
$$\|x_n\cdot y - \pi(x)y\|_{\cal A} = 
\|\pi^{\prime}(y)x_n-\pi^{\prime}(y)x\|_{\cal A}\leq\|\pi^{\prime}(y)\|\,
\|x_n-x\|_{\cal A}\leq \|\pi^{\prime}(y)\|\, \|x_n-x\|_2\to 0.$$
So, $\pi(x)y = x\cdot y.$

Item (3) follows from from the definition of the trace (Theorem \ref{Tr}) and 
item (1).
\end{proof}

\begin{lemma}\label{extendrep}
The representation $\pi_{\cal A}:A_{\frk Z}\to\cal L(\cal{H}_{\cal A})$ extends
to an ultraweakly continuous representation (also denoted $\pi_{\cal A}$)
of $\frk A$ in $\cal L(\cal{H}_{\cal A})$.
\end{lemma}

\begin{proof}[\bf Proof]
We first observe that the space of {\bf norm}-continuous functions, 
$C_c(\real,\frk A)\subset \cal{H}_{\cal A}$ in a natural way.
That is if $x\in C_c(\real,\frk A)$, then for $y\in \cal A$ the map:
$$y\mapsto \int\bar\tau((x(t))^*y(t))dt$$
is a bounded $\frk Z$-module mapping from $\cal A$ to $\frk Z$ and so defines
a unique element in $\cal{H}_{\cal A}$. If we abuse notation and denote this element
in $\cal{H}_{\cal A}$ by $x$, then we get the formula:
$$\bra x,y\ket = \int\bar\tau((x(t))^*y(t))dt.$$ 
Clearly, $\cal A = C_c(\real,A_{\frk Z})\subset C_c(\real,\frk A)
\subset \cal{H}_{\cal A}$. The extension of $\pi_{\cal A}$ to $\frk A$
is now obvious:
$$[\pi_{\cal A}(a)x](s) = ax(s)\;\;for\; a\in \frk A,\;x\in C_c(\real,\frk A),
\;s\in\real.$$
It is easy to check that this is a well-defined extension to $\frk A$ as
$\frk Z$-module mappings on the $\frk Z$-submodule  $C_c(\real,\frk A)
\subset \cal{H}_{\cal A}$. These $\pi_{\cal A}(a)$ extend uniquely to 
$\frk Z$-module mappings on $\cal{H}
_{\cal A}$ since $\cal{H}_{\cal A}$ is also the Paschke 
dual of $C_c(\real,\frk A)$ by Lemma \ref{complete}.

To see that $\pi_{\cal A}: \frk A \to \cal L(\cal{H}_{\cal A})$ is normal, it suffices
to see that $\pi_{\cal A}(\frk A)$ is ultraweakly closed in $\cal L(\cal{H}_{\cal A})$
by Cor. I.4.1 of \cite{Dix}. To this end, it suffices to see that the unit
ball in $\pi_{\cal A}(\frk A)$ is ultraweakly closed. So, let $\{a_n\}$
be a net in $\frk A$ with $\|a_n\|=\|\pi_{\cal A}(a_n)\|\leq 1$ and
$$\pi_{\cal A}(a_n)\to T \;{\rm ultraweakly\;in}\; \cal L(\cal{H}_{\cal A}).$$
Since the unit ball in $\frk A$ is ultraweakly compact we can assume 
(by choosing a subnet if necessary) that there is an $a\in\frk A$ such that
$a_n\to a$ ultraweakly. By Proposition \ref{conv} part (3), we have for all
$x,y\in C_c(\real,\frk A)$
$$\bra x,\pi_{\cal A}(a_n)y\ket\to \bra x,Ty\ket\;{\rm ultraweakly\; in}\; \frk Z.$$

On the other hand, if $x=cf$ and $y=bg$ for $c,b\in\frk A$ and 
$f,g\in C_c(\real)$ then one easily calculates that:
$$\bra x,\pi_{\cal A}(a_n)y\ket=\bar\tau(a_nbc^*)\int \bar f(t)g(t)dt$$
which converges ultraweakly in $\frk Z$ to $\bra x,\pi_{\cal A}(a)y\ket$. Thus,
for all such $x,y$ we have:
$$\bra x,\pi_{\cal A}(a)y\ket=\bra x,Ty\ket.$$
Clearly, the same equation holds for all finite linear combinations of such 
$x$ and $y$. Since such combinations are $\|\cdot\|_2$-dense in 
$C_c(\real,\frk A)$ (and so $\|\cdot\|_{\frk Z}$-dense) we have the equation 
holding for all $x,y\in C_c(\real,\frk A)$. Hence, for all 
$y\in C_c(\real,\frk A)$ we have:
$$\pi_{\cal A}(a)y=Ty.$$
Since $\pi_{\cal A}(a)$ leaves the pre-Hilbert $\frk Z$-module
$C_c(\real,\frk A)$ invariant, Proposition 3.6 of \cite{Pa} implies that 
$T=\pi_{\cal A}(a)$ as required.

\end{proof}

\begin{tech*}
Now, the natural embedding of the $\frk Z$-module, 
$L^2(\real)\otimes_{alg}A_{\frk Z}$ into $L^2(\real,A_{\frk Z})$ induces an
embedding: $L^2(\real,A_{\frk Z})\hookrightarrow L^2(\real)\otimes_{\frk Z} A_{\frk Z}$
where the latter is {\bf defined} to be the completion of the algebraic tensor
product in the pre-Hilbert ${\frk Z}$-module norm, \cite{L}. Thus we get a series of 
inclusions of pre-Hilbert $\frk Z$-modules each of which is strict unless $A$ is 
finite-dimensional:
$$L^2(\real)\otimes_{alg}A_{\frk Z}\subset L^2(\real,A_{\frk Z})\subset
L^2(\real)\otimes_{\frk Z} A_{\frk Z}\subset \cal{H}_{\cal A}.$$
One could insert another (generally strict) series of containments:
$$L^2(\real)\otimes_{\frk Z} A_{\frk Z}\subset L^2(\real)\otimes_{\frk Z} \frk A\subset 
\cal{H}_{\cal A}.$$
Or, even the diagram of containments:
\begin{eqnarray*}
C_c(\real,A_{\frk Z})\;\; =&\hspace{.3in}\cal A\hspace{.3in} &=\;\;
C_c(\real,A_{\frk Z})\\
\cup \hspace{.5in}&\hspace{.5in}&\hspace{.5in}\cap\\
C_c(\real)\otimes_{alg}A_{\frk Z} \subset&L^2(\real)\otimes_{alg}A_{\frk Z}
&\subset L^2(\real,A_{\frk Z})
\end{eqnarray*}
In general, one might be able to realize $\cal{H}_{\cal A}$ as some sort of collection
of measurable $L^2$-functions from $\real$ into the $\frk Z$-module
$\cal{H}_{A_{\frk Z}}=\cal{H}_{\frk A}$; however, 
this does not seem particularly useful, so
we refrain from exploring this idea further. The important point is that each
of these $\frk Z$-modules has the {\bf same} Paschke dual $\cal{H}_{\cal A}$ and
so we can define operators in $\cal L(\cal{H}_{\cal A})$ by defining bounded
adjointable $\frk Z$-module mappings on any one of them by Corollary 3.7
of \cite{Pa}. Of course any one such operator may or may not leave the other
$\frk Z$-modules invariant.
\end{tech*}

\begin{prop}\label{lvN}
Let $\cal A= C_c(\real,A_{\frk Z})$. Then,

\noindent (1)\;For $x\in\cal A$ we have 
$\pi(x)=(\pi_{\cal A}\times U)(x) = \int\pi_{\cal A}(x(t))U_t dt$, where the 
integral converges in the norm of $\cal L(\cal{H}_{\cal A})$.

\noindent (2)\;$\cal U(\cal A)=
[(\pi_{\cal A}\times U)(A_{\frk Z}\rtimes\real)]^{\prime\prime}=
[\pi_{\cal A}(\frk A) \cup \{U_t\}_{t\in\real}]^{\prime\prime}$.

\noindent (3)\;$\cal U(\cal A)=
[(\pi_{\cal A}\times U)(A\rtimes\real)]^{\prime\prime}.$ 
\end{prop}

\begin{proof}[\bf Proof]
To see item (1) we note that in the proof of Proposition \ref{AZ} it was shown 
that for $x,y\in \cal A$:
$$\pi(x)y=(\pi_{\cal A}\times U)(x)y.$$
By Proposition 3.6 of \cite{Pa} this implies that 
$\pi(x)=(\pi_{\cal A}\times U)(x)$ as elements of $\cal L(\cal{H}_{\cal A})$.
The second equality in item (1) is true for any crossed product when $x$ is
a compactly supported continuous function from the group into the $C^*$-algebra.

\noindent To see item (2) we first note that by item (1):
\begin{eqnarray*}
(\pi_{\cal A}\times U)(A_{\frk Z}\rtimes\real) &=& (\pi_{\cal A}\times U)
(C_c(\real,A_{\frk Z}))^{-\|\cdot\|}\\
&=& (\pi_{\cal A}\times U)(\cal A)^{-\|\cdot\|}\\
&=& \pi(\cal A)^{-\|\cdot\|}.
\end{eqnarray*}
Hence,
$$\cal U(\cal A)=[\pi(\cal A)]^{\prime\prime}=
[\pi(\cal A)^{-\|\cdot\|}]^{\prime\prime}=
[(\pi_{\cal A}\times U)(A_{\frk Z}\rtimes\real)]^{\prime\prime}.$$
Now, by the Commutation Theorem (\ref{Comm}):
$$\cal U(\cal A) = \left (\pi^{\prime}(\cal A)\right )^{\prime}.$$
and it is an easy calculation that $\pi_{\cal A}(A_{\frk Z})\subset
\left (\pi^{\prime}(\cal A)\right )^{\prime}.$ Since the representation
$\pi_{\cal A}$ is ultraweakly continuous on $\frk A$ and $A_{\frk Z}$ is
ultraweakly dense in $\frk A$ we see that:
$$\pi_{\cal A}(\frk A)=\pi_{\cal A}(A_{\frk Z})^{-u.w.}\subset
\left (\pi^{\prime}(\cal A)\right )^{\prime}=\cal U(\cal A).$$
It is a straightforward calculation (since the operators $U_t$ leave $\cal A$
invariant) that :
$$ \{U_t\}_{t\in\real}\subset
\left (\pi^{\prime}(\cal A)\right )^{\prime}=\cal U(\cal A).$$
Thus, $$[\pi_{\cal A}(\frk A) \cup \{U_t\}_{t\in\real}]^{\prime\prime}\subset
\cal U(\cal A).$$

On the other hand, if 
$T\in [\pi_{\cal A}(\frk A) \cup \{U_t\}_{t\in\real}]^{\prime},$ then
$T\in [\pi_{\cal A}(A_{\frk Z}) \cup \{U_t\}_{t\in\real}]^{\prime}$ and
by the full force of item (1), we see that
$T\in \left (\pi(\cal A)\right )^{\prime}=\cal U(\cal A)^{\prime}$ by
Theorem \ref{Comm}. That is,
$$[\pi_{\cal A}(\frk A) \cup \{U_t\}_{t\in\real}]^{\prime}\subset
\cal U(\cal A)^{\prime}\;or$$
$$[\pi_{\cal A}(\frk A) \cup \{U_t\}_{t\in\real}]^{\prime\prime}\supset
\cal U(\cal A)$$
as required.

To see item (3), we observe that since $A$ is ultraweakly dense in $\frk A$,
Lemma \ref{extendrep} implies that $\pi_{\cal A}(\frk A)=\pi_{\cal A}(A)^{\prime\prime}
\subset [(\pi_{\cal A}\times U)(A\rtimes\real)]^{\prime\prime}.$ Since
$\{U_t\}_{t\in\real} \subset 
[(\pi_{\cal A}\times U)(A\rtimes\real)]^{\prime\prime}$, we have by item (2) 
that $\cal U(\cal A)\subset 
[(\pi_{\cal A}\times U)(A\rtimes\real)]^{\prime\prime}$. The other containment
is trivial.
\end{proof}

\begin{defn}\label{Indrep}{\bf The Induced Representation.}
Now, there is another representaion of $\cal A= C_c(\real,A_{\frk Z})$ (and
hence $A_{\frk Z}\rtimes\real$) on $\cal{H}_{\cal A}$ which is unitarily 
equivalent to $\pi = \pi_{\cal A}\times U$. In the remainder of the paper we will
use the standard notation for this representation, namely $Ind:$ see below.
Later when we define the notion of index, we will use the notation $Index$ to avoid confusion.
To define the representation $Ind$ we first define a single 
unitary $V\in \cal L(\cal{H}_{\cal A})$ via:
$$(V\xi)(t)= 
\bar\alpha_{t}^{-1}(\xi(t))\;\;for\;\;\xi\in L^2(\real,A_{\frk Z}).$$
One easily checks that $V$ is a bounded, adjointable, $\frk Z$-module
mapping on the $\frk Z$-module $L^2(\real,A_{\frk Z})$ and therefore on 
$\cal{H}_{L^2(\real,A_{\frk Z})} = \cal{H}_{\cal A}$ by the previous remarks. One easily 
checks that for $a\in A_{\frk Z}$ , $t\in\real$ and $\xi\in 
L^2(\real,A_{\frk Z})$
$$V\pi_{\cal A}(a)V^*=\tilde\pi(a)\hspace{.4in}and\hspace{.4in}VU_tV^*=
\lambda_t,$$
where 
$$(\tilde\pi(a)\xi)(s)=\bar\alpha_s^{-1}(a)\xi(s)\hspace{.3in}and\hspace{.3in}
(\lambda_t\xi)(s)=\xi(s-t).$$
Another straightforward calculation shows that for $x,\xi\in\cal A$ 
$$(V\pi(x)V^*\xi)(s)=\int\bar\alpha_s^{-1}(x(t))\xi(s-t)dt,$$
and that this formula easily extends to $\xi\in L^2(\real,A_{\frk Z})$.

Now, if $x\in L^2(\real,A_{\frk Z})$, $\pi(x)$ is bounded and $\xi\in\cal A$,
then using the formula of item (2) in lemma \ref{technical} one easily calculates that
we obtain the same formula, namely
$$(V\pi(x)V^*\xi)(s)=\int\bar\alpha_s^{-1}(x(t))\xi(s-t)dt.$$
Since this representation of $A_{\frk Z}\rtimes\real$, $x\mapsto V\pi(x)V^*$ is 
induced from the left
multiplication of $A_{\frk Z}$ on itself via the action of $\real$ on 
$A_{\frk Z}$, we denote it by $Ind(x)$. That is,
$$Ind(x):=V\pi(x)V^*.$$
Now, the von Neumann algebra, $\cal U(\cal A)$ contains the
representations $\pi_{\cal A}$ of $A_{\frk Z}$ and $U$ of $\real$ which 
integrate to give the representation $\pi=\pi_{\cal A}\times U$
of $\cal A$ (and hence of $A_{\frk Z}\rtimes \real$) in $\cal U(\cal A)$. We 
define the von Neumann algebra $$\cal M=V\cal U(\cal A)V^*$$
in $\cal L(\cal{H}_{\cal A})$ which also has centre $\frk Z$ and is unitarily 
equivalent to $\cal U(\cal A)$ but for which 
the machinery of $\frk Z$-Hilbert algebras is not directly applicable. 
$\cal M$ is generated by the representations, 
$\tilde\pi(\cdot):=V\pi_{\cal A}(\cdot)V^*$ of $A_{\frk Z}$ and
$\lambda_{(\cdot)}:=VU_{(\cdot)}V^*$ of $\real$. The integrated representation
$\tilde\pi\times\lambda$ is, of course, $Ind$. The trace on $\cal M$
is denoted by $\hat\tau$ and is defined on 
$\cal M^{\hat\tau}:=V\cal U(\cal A)^{\sigma}V^*$ via:
$$\hat\tau(T):=\sigma(V^*TV).$$
It follows from item (3) of Lemma \ref{technical} that if $x , y \in L^2(\real,A_{\frk Z})$ 
and if $\pi(x)$ and $\pi(y)$
are bounded, then the operator $Ind (x)^*Ind (y)$ is in the ideal of
definition of the $\frk Z$-trace, $\hat\tau$ on $\cal M$, and
$$\hat\tau[Ind (x)^*Ind (y)]=\hat\tau[V\pi(x)^*\pi(y)V^*]=\bra x,y\ket=
\int\bar\tau(x(t)^*y(t))dt.$$

\end{defn}

\begin{defn}\label{Hilbtran}{\bf The Hilbert Transform.}
The Hilbert Transform, $H_{\real}$ on $L^2(\real)$ is defined for $\xi\in 
L^2(\real)$ by:
$$H_{\real}(\xi) = (\hat\xi sgn\check),$$
where $\hat{}$, $\check{}$ are the usual Fourier transform and inverse 
transform and $sgn$ is the usual signum function on $\real$.

Then, $H_{\real}$ is a self-adjoint unitary, so that $H_{\real}^2=1$ and
$P_{\real}:=\frac{1}{2}(H_{\real}+1)$ is the projection onto the Hardy
space, $\cal{H}^2(\real)$. By \cite{L}, $H:=H_{\real}\otimes 1$ and 
$P:=P_{\real}\otimes 1$ define bounded adjointable $\frk Z$-module
maps on $L^2(\real)\otimes_{alg}A_{\frk Z}$ (and therefore on $\cal{H}_{\cal A}$)
with the same properties. That is, $H^2=1$ and $P=\frac{1}{2}(H+1)$ satisfies
$P=P^*=P^2$.
\end{defn}

In the lemma below, we identify $L^2(\real)$ with $L^2(\real)\cdot 1_A$
inside $L^2(\real,A_{\frk Z})$.

\begin{lemma}\label{Hilbconvo}
The operators $H$ and $P$ are in $\cal M$. In fact, if we define for
$\epsilon >0$ the function $f_{\epsilon}$ in 
$L^2(\real)\subset L^2(\real,A_{\frk Z})\subset \cal{H}_{\cal A}$ via:
$$f_{\epsilon}(t)=\frac{1}{\pi it}\;\;for\;\;|t|\geq\epsilon$$
then the $\pi(f_{\epsilon})$ (technically, $\pi(f_{\epsilon}\cdot 1_{A})$)
are uniformly bounded and as $\epsilon\to 0$
$$Ind(f_{\epsilon})=V\pi(f_{\epsilon})V^*\to H\;\;strongly \;\; on\;\; 
L^2(\real)\otimes \bar{A}_{\frk Z},$$
and so
$$Ind(f_{\epsilon})=V\pi(f_{\epsilon})V^*\to H\;\;ultraweakly \;\; on\;\;
\cal{H}{\cal A}.$$
\end{lemma}

\begin{proof}[\bf Proof]
It follows from \cite{DM} that left convolution by the functions $f_{\epsilon}$,
$\lambda(f_{\epsilon})$,
are uniformly bounded on $L^2(\real)$ and converge strongly to $H_{\real}$. It
is trivial then that $\lambda(f_{\epsilon})\otimes 1$ converges strongly
to $H_{\real}\otimes 1$ on $L^2(\real)\otimes_{alg} A_{\frk Z}$. Since these 
operators are all uniformly bounded, adjointable $\frk Z$-module maps by
\cite{L}, we see by the usual $\delta/3$-argument, that their extensions to the 
completion, $L^2(\real)\otimes_{\frk Z} 
A_{\frk Z}$ satisfy:
$$\lambda(f_{\epsilon})\otimes 1\to H_{\real}\otimes 1 = H
\;\;strongly\;\;on\;\;L^2(\real)\otimes_{\frk Z} 
A_{\frk Z}.$$ It now follows from item (3) of Lemma \ref{conv} 
(with $L^2(\real)\otimes_{\frk Z} A_{\frk Z}$ in place of $\cal A$) and 
Key Problem 8 that 
$$\lambda(f_{\epsilon})\otimes 1\to H
\;\;ultraweakly\;\;on\;\;\cal{H}_{L^2(\real)\otimes_{\frk Z} A_{\frk Z}}=\cal{H}_{\cal A}.$$
It remains to see that $\lambda(f_{\epsilon})\otimes 1 = Ind(f_{\epsilon})$ on
$\cal{H}_{\cal A}.$ Now the former is initially defined on 
$L^2(\real)\otimes_{alg} A_{\frk Z}$ while the latter is initially defined on
$V(\cal A)=\cal A$. Since they are both defined on the common dense domain
$C_c(\real)\otimes A_{\frk Z}$, it suffices to check equality there. This is a
trivial calculation.
\end{proof}

\begin{rem}\label{formalconvo}
It follows from the previous lemma that for $\xi\in \cal A$
$$H(\xi) = norm\lim_{\epsilon\to 0} V\pi(f_{\epsilon})V^*\xi.$$
And since
$$V\pi(f_{\epsilon})V^*\xi(s) = 
\int f_{\epsilon}(t)\xi(s-t)dt
=\int_{|t|\geq 0} \frac{1}{\pi it}\xi(s-t)dt\;\;for\;\;s\in\real,$$
we can formally write:
$$(H\xi)(s)=\int\frac{1}{\pi it}\xi(s-t)dt\;\;for\;\;
\;\;\xi\in\cal A\;\;and\;\;s\in\real$$
where we understand the integral to be the principal-value integral converging
in the norm of $\cal{H}_{\cal A}.$
\end{rem}

\section{\bf The INDEX THEOREM}

We quickly recap for the benefit of the reader what we've done so far. 

We begin with a unital 
$C^*$-algebra $A$ and a unital $C^*$-subalgebra, $Z$ of the centre of $A.$ 
We assume that we have a faithful, unital $Z$-trace $\tau$ and a 
continuous action 
$\alpha:\real\to Aut(A)$ leaving $\tau$ and hence $Z$ invariant. In short, the 4-tuple $(A,Z,\tau,\alpha)$
is our object of study. As {\bf Standing Assumptions}, we will assume that we have
a concrete $*$-representation of $A$ on a Hilbert space $\mathcal H$
which carries a {\bf faithful}, unital u.w.-continuous $\frk Z$-trace
$\bar\tau:\frk A \to \frk Z$ extending $\tau$ where as before $\frk A$ 
and $\frk Z$ denote respectively, the
ultraweak closures of $A$ and $Z$ on $\mathcal H$. Since $A$ is 
concretely represented on this Hilbert space, we do not carry
a special notation for this representation. Moreover there is an 
ultraweakly continuous action 
$\bar\alpha:\real\to Aut(\frk A)$ extending
$\alpha$ and leaving $\bar\tau$ and $\frk Z$ invariant. If $Z$ has a faithful
state, $\omega$ then the GNS representation of the state
$\bar\omega=\omega\circ\tau$ gives us a 
representation of $A$ satisfying the {\bf Standing Assumptions} by Proposition \ref{extension}.

We defined $A_{\frk Z}$ to be the $C^*$-subalgebra of $\frk A$ generated by
$A$ and $\frk Z$, so that $\bar\alpha$ restricts to a norm-continuous action of 
$\real$ on $A_{\frk Z}$ and $\bar\tau$ restricts to a faithful, unital 
$\frk Z$-trace on $A_{\frk Z}$ . We defined 
$\cal A = C_c(\real,A_{\frk Z})$ to be a $*$-algebra with the usual
$\bar\alpha$-twisted convolution multiplication. There is a natural (right)
pre-Hilbert $\frk Z$-module structure on $\cal A$ making it into a
$\frk Z$-Hilbert algebra as defined in section 3. We defined $\cal{H}_{\cal A}$
to be the Paschke dual of all bounded $\frk Z$-module mappings from $\cal A$
to $\frk Z$ (i.e., all $\frk Z$-linear ``$\frk Z$-valued functionals'' on 
$\cal A$). Then $\cal L(\cal{H}_{\cal A})$ is a type I von Neumann algebra with
centre $\frk Z$. The point of this set-up is that the von Neumann subalgebra
$\cal U(\cal A)$ of $\cal L(\cal{H}_{\cal A})$ 
generated by the left multiplications 
$\pi(x)$ of $\cal A$ on $\cal{H}_{\cal A}$ contains $\frk Z$ in its centre and
has a faithful, normal semifinite 
$\frk Z$-trace $\sigma$, defined on the two-sided ideal, 
$\cal U(\cal A)^{\sigma}=\pi(\cal A_b^2)$ via:
$$\sigma(\pi(\xi\eta))=\bra\xi^*,\eta\ket,$$
for $\xi,\eta\in\cal A_b$ the (full) $\frk Z$-Hilbert algebra of (left) bounded
elements in $\cal{H}_{\cal A}$.

At this point we look at a von Neumann algebra $$\cal M=V\cal U(\cal A)V^*$$
in $\cal L(\cal{H}_{\cal A})$ which also contains $\frk Z$ in its centre. 
$\cal M$ is generated by representations, 
$\tilde\pi(\cdot):=V\pi_{\cal A}(\cdot)V^*$ of $A_{\frk Z}$ and
$\lambda_{(\cdot)}:=VU_{(\cdot)}V^*$ of $\real$. The integrated representation
$\tilde\pi\times\lambda$ is denoted by $Ind$. The canonical trace on $\cal M$
is denoted by $\hat\tau$ and has domain of definition: 
$$\cal M^{\hat\tau}=\{S\in\cal M| S=
V\pi(\xi\eta)V^*\;\; some\;\;\xi,\eta\in \cal A_b\}.$$ 
And for $S=V\pi(\xi\eta)V^*$,
$$\hat\tau(S)=\bra\xi^*,\eta\ket.$$
In particular, 
if $x , y \in L^2(\real,A_{\frk Z})$ with $\pi(x)$ and $\pi(y)$
bounded, then the operator $Ind (x)^*Ind (y)$ is in the ideal of
definition of the $\frk Z$-trace, $\hat\tau$ on $\cal M$, and
$$\hat\tau[Ind (x)^*Ind (y)]=
\int\bar\tau(x(t)^*y(t))dt.$$

\begin{defn}\label{Toeplitz}
We consider the semifinite von Neumann algebra,
$$\cal N:= P\cal M P$$
with the faithful, normal, semifinite $\frk Z$-trace obtained by restricting
$\hat\tau$. For $a\in A$ we define the {\bf Toeplitz} operator 
$$T_a:=P\tilde\pi(a)P\in\cal N.$$
\end{defn}

We recall from Section 1 that $\delta$ is the infinitesimal generator of
$\alpha$ on $A$ and that $$a\mapsto \frac{1}{2 \pi i}\tau(\delta(a)a^{-1}):
dom(\delta)^{-1} \to Z_{sa}$$
is a group homomorphism which is constant on connected components
and so extends uniquely to a group homomorphism $A^{-1} \to Z_{sa}$
which is constant on connected components and is $0$ on $Z^{-1}.$
With this convention and all the above notation, we state our index theorem.
Much of the work that we have done so far is to make sense of the
the statement of the following theorem and to make sense of the 
index calculations of \cite{CMX} and \cite{PhR} in this generality. It is 
interesting that the conclusions of the theorem are insensitive to the
choice of a suitable representation of $A$ satisfying the standing 
assumptions. In particular, if the representation is chosen using 
Proposition \ref{extension}, the conclusions of the theorem are insensitive to the
choice of a faithful state on $Z$.

\begin{thm}\label{index}
Let $A$ be a unital $C^*$-algebra and let $Z\subseteq Z(A)$ be a unital
$C^*$-subalgebra of the centre of $A.$ Let $\tau:A \to Z$ be a faithful,
unital $Z$-trace which is invariant under a continuous
action $\alpha$ of $\real$. Then for any $a\in A^{-1}\cap dom(\delta)$, 
the Toeplitz operator
$T_a$ is Fredholm relative to the trace $\hat\tau$ on 
$\cal N=P(Ind(A\rtimes\real)^{\prime\prime})P$, and
$$\hat\tau{\text -}Index(T_a)=\frac{-1}{2\pi i}\tau(\delta(a)a^{-1}).$$
\end{thm}

We follow the second proof of \cite{CMX}, Section 25.2 (cf section 3 of \cite{PhR}. 
Now relative to the decomposition $1=P+(1-P)$ we see that 
$$\tilde\pi(a)=\left[\begin{array}{cc} T_a & B\\C & D\end{array}\right],$$
where
$$B=P\tilde\pi(a)(1-P)=P[P,\tilde\pi(a)]=\frac{1}{2}P[H,\tilde\pi(a)],$$
and similarly,
$$C=\frac{1}{2}[H,\tilde\pi(a)]P.$$
Thus, we are led to calculate the general commutator $[H,\tilde\pi(a)]$
for $a\in dom(\delta).$

\begin{lemma}\label{commutator}
For any $a\in dom(\delta)$, $[H,\tilde\pi(a)]$ belongs to $\cal M_2^{\hat\tau}.$
In fact, $[H,\tilde\pi(a)]=Ind(x),$ where 
$x\in C_0(\real,A_{\frk Z})\cap L^2(\real,A_{\frk Z})$ is given by
$$x(t)=\frac{\alpha_t(a)-a}{\pi it}.$$
\end{lemma}

\begin{proof}[\bf Proof]
Now, $Ind(f_{\epsilon})$ converges strongly on $\cal A$ to $H$,
so we easily compute for $\xi\in \cal A$:
$$[Ind(f_{\epsilon}),\tilde\pi(a)]\xi=Ind(x_{\epsilon})\xi$$
where $$x_{\epsilon}(t)=\left\{\begin{array}{cc} \frac{\alpha_t(a)-a}{\pi it} &
|t|\geq \epsilon\\0 & else\end{array}\right..$$
So, the $Ind(x_{\epsilon})$ are uniformly bounded operators that converge
pointwise on $\cal A$ to $[H,\tilde\pi(a)]$. Now, since
$x(t)\to (\pi i)^{-1}\delta(a)$ as $t\to 0$ and
$$\|x(t)\|^2\leq \frac{4\|a\|^2}{\pi^2 t^2},$$
we see that $x\in C_0(\real,A_{\frk Z})\cap L^2(\real,A_{\frk Z})$. One easily
calculates that for $\xi\in\cal A$
$$\|Ind(x)\xi - Ind(x_{\epsilon})\xi\|_{\frk Z}\leq
\|Ind(x)\xi - Ind(x_{\epsilon})\xi\|_2\to 0,$$ and so
$Ind(x)$ and $[H,\tilde\pi(a)]$ agree on $\cal A$. That is, by the discussion
in 8.6, $\pi(x)=V^*Ind(x)V$ is left bounded and 
$$Ind(x)=[H,\tilde\pi(a)]\;\;in\;\; \cal L(\cal{H}_{\cal A}).$$ 
\end{proof}

We want to use the $\frk Z$-trace version of H\"{o}rmander's formula
(Theorem A3 and Corollary A4
in the Appendix) to calculate the $\hat\tau$-index
of the Toeplitz operator $T_a$ as $\hat\tau([T_a,T_{a^{-1}}])$. So we are led
to examine such commutators in the hopes that they are in fact trace-class
(they are).

\begin{cor}\label{commutator2}
If $a,b\in dom(\delta)$ we have $T_aT_b-T_{ab}\in\cal M^{\hat\tau}\cap\cal N=
\cal N^{\hat\tau}.$ In particular, if $b=a^{-1}$ then $T_a$ and $T_b$ are
$\hat\tau$-Fredholm operators in $\cal N$. In general, if $ab=ba$, then 
$[T_a,T_b]\in\cal N^{\hat\tau}.$ 
\end{cor}

\begin{proof}[\bf Proof]
We easily calculate (see cor.3.3 of \cite{PhR}): 

\begin{eqnarray}
T_aT_b-T_{ab}&=&P\tilde\pi(a)(P-1)\tilde\pi(b)P\\
=\cdots&=&\frac{1}{4}P[H,\tilde\pi(a)][H,\tilde\pi(b)]P
\end{eqnarray}

which is in $\cal M^{\hat\tau}\cap P\cal M P=\cal N^{\hat\tau}.$ If $ab=ba,$
then
$$[T_a,T_b]=(T_aT_b-T_{ab})+(T_{ba}-T_bT_a)\in\cal N^{\hat\tau}.$$
\end{proof}

\begin{dis*}
In the case that $a,b\in dom(\delta)$ commute we have by equation (1) and a
small calculation:
\begin{eqnarray}
[T_a,T_b]&=&P\tilde\pi(a)(P-1)\tilde\pi(b)P-P\tilde\pi(b)(P-1)\tilde\pi(a)P\\
=\cdots&=&\frac{1}{2}P(\tilde\pi(a)H\tilde\pi(b)-\tilde\pi(b)H\tilde\pi(a))P,
\end{eqnarray}
and both of these terms are trace-class. Applying the trace to equation (4)
we get:
\begin{equation}
\hat\tau([T_a,T_b])=
\frac{1}{2}\hat\tau(P(\tilde\pi(a)H\tilde\pi(b)-\tilde\pi(b)H\tilde\pi(a))P).
\end{equation}
On the other hand, applying the trace to equation (3) , using the cyclic 
property of the trace and a little calculation (see \cite{PhR}) we get:
\begin{equation}
\hat\tau([T_a,T_b])=
\frac{1}{2}\hat\tau((1-P)(\tilde\pi(a)H\tilde\pi(b)-
\tilde\pi(b)H\tilde\pi(a))(1-P)).
\end{equation}
Defining
$$T:=\tilde\pi(a)H\tilde\pi(b)-\tilde\pi(b)H\tilde\pi(a),$$
and averaging equations (4) and (6) we get:
\begin{equation}
\hat\tau([T_a,T_b])=\frac{1}{4}\hat\tau(PTP+(1-P)T(1-P)),
\end{equation}
and both of these terms are trace-class. Unfortunately, $T$ itself is not
usually trace-class. However, $T$ is in $\cal M^{\hat\tau}_2$ by the following
lemma.
\end{dis*}

\begin{lemma}\label{commutator3}(cf lemma 3.4 of \cite{PhR}) Suppose $a,b\in\;dom(\delta)$ and
$ab=ba.$ Then 
$$T=\tilde\pi(a)H\tilde\pi(b)-\tilde\pi(b)H\tilde\pi(a)$$
belongs to $\cal M^{\hat\tau}_2$; in fact it has the form $Ind(y)$ where $y$
is the function in $C_0(\real,A_{\frk Z})\cap L^2(\real,A_{\frk Z})$
given by $y(t)=(\pi it)^{-1}(a\alpha_t(b)-b\alpha_t(a)).$
\end{lemma}

\begin{proof}[\bf Proof]
It is straightforward to verify that we can also write:
$$T=[H,\tilde\pi(b)]\tilde\pi(a)-[H,\tilde\pi(a)]\tilde\pi(b).$$
Then by Lemma \ref{commutator} we see that $T=Ind(y)$ where 
$$y(t)=\frac{(\alpha_t(b)-b)\alpha_t(a)}{\pi it} - 
\frac{(\alpha_t(a)-a)\alpha_t(b)}{\pi it}
=\frac{a\alpha_t(b)-b\alpha_t(a)}{\pi it}.$$
Since $y(t)\to (\pi i)^{-1}(\delta(b)a-\delta(a)b)$ in the norm of $A$
as $t\to 0$, $y$ is a continuous $A$-valued function. As 
$\|y(t)\|\leq 2\|a\| \|b\|/\pi t$ for $t\neq 0$, we also see that 
$y\in L^2(\real,A_{\frk Z}).$
\end{proof}

\begin{rem*}
In the previous lemma $y(0)=
(\pi i)^{-1}(\delta(b)a-\delta(a)b)=-2(\pi i)^{-1}\delta(a)b.$
Combining this with equation (7) of the previous discussion {\bf would}
yield the desired formula: 
$$\hat\tau([T_a,T_b])= \frac{-1}{2\pi i}\hat\tau(\delta(a)b),$$
{\bf assuming} that the operator $T$ is trace-class. Since $T$ is generally not 
trace-class, we need an approximate identity argument.
\end{rem*}

\begin{lemma}\label{approxid}
If $S\in \cal M^{\hat\tau}$ and $\{f_n\}$ is a sequence
of functions in $C_c(\real)^+\subset C_c(\real,A_{\frk Z})$ each having 
integral $1$ and symmetric supports about $0$ shrinking to $0$ then
$$\hat\tau(S)=\rm{uw}\lim_{n\to\infty}\hat\tau(Ind(f_n)S).$$
\end{lemma}

\begin{proof}[\bf Proof]
As in the proof of Lemma \ref{Hilbtran}, we see that the operators,
$Ind(f_n)=V\pi(f_n)V^*$ 
are 
uniformly bounded on $\cal{H}_{\cal A}$
by $1$ and converge strongly to $1$ on $L^2(\real)\otimes\bar{A_{\frk Z}}.$
In particular, for all $x,y\in \cal A$ we have by Paschke's Cauchy-Schwarz
inequality (Propn. 2.3 of \cite{Pa}):
\begin{eqnarray*}
\hat\tau[Ind(x) Ind(y)]&=&\bra x^*,y\ket
=\bra y^*,x\ket
=\rm{norm}\lim_{n\to\infty}\bra y^*,\pi(f_n)x\ket\\
&=&\rm{norm}\lim_{n\to\infty}\bra(f_nx)^*,y\ket
=\rm{norm}\lim_{n\to\infty}\hat\tau[Ind(f_nx) Ind(y)]\\
&=&\rm{norm}\lim_{n\to\infty}\hat\tau[Ind(f_n)Ind(x)Ind(y)].
\end{eqnarray*}
Now, by item (3) of Lemma \ref{conv} we see that for all $\xi,\eta\in\cal A_b$:
$$\hat\tau[Ind(\xi) Ind(\eta)]=
\rm{uw}\lim_{n\to\infty}\hat\tau[Ind(f_n)Ind(\xi)Ind(\eta)].$$
Since every $S\in\cal M^{\hat\tau}$ has the form $S=Ind(\xi) Ind(\eta)$ for
some $\xi,\eta\in\cal A_b$, we are done.
\end{proof}

\begin{prop}\label{commutator4}
If $a,b\in dom(\delta)$ and $ab=ba$, then $[T_a,T_b]\in\cal N^{\hat\tau}$
and
$$\hat\tau [T_a,T_b] = \frac{-1}{2\pi i}\tau(\delta(a)b).$$
\end{prop}

\begin{proof}[\bf Proof]
Let $\{f_n\}$ be as in the previous lemma. Then, by equation (7) of the 
Discussion, the previous two lemmas, and the fact that
$Ind(f_n)P=PInd(f_n)$ we get:
\begin{eqnarray*}
\hat\tau([T_a,T_b])&=&\frac{1}{4}\hat\tau(PTP+(1-P)T(1-P))\\
&=&\rm{uw}\lim\frac{1}{4}\hat\tau(Ind(f_n)(PTP+(1-P)T(1-P)))\\
&=&\rm{uw}\lim\frac{1}{4}\hat\tau(Ind(f_n)PTP+Ind(f_n)(1-P)T(1-P))\\
&=&\rm{uw}\lim\frac{1}{4}\hat\tau(PInd(f_n)TP+(1-P)Ind(f_n)T(1-P))\\
&=&\rm{uw}\lim\frac{1}{4}\hat\tau(PInd(f_n)T+(1-P)Ind(f_n)T)\\
&=&\rm{uw}\lim\frac{1}{4}\hat\tau(Ind(f_n)T)\\
&=&\rm{uw}\lim\frac{1}{4}\hat\tau(Ind(f_n)Ind(y))\\
&=&\rm{uw}\lim\frac{1}{4\pi i}\int f_n(t)\tau\left(\frac{\alpha_t(b)-b}{t}a-
\frac{\alpha_t(a)-a}{t}b\right)dt.
\end{eqnarray*}
In fact, this last limit is easily seen to converge in norm, so that
\begin{eqnarray*}
\hat\tau([T_a,T_b])&=&\frac{1}{4\pi i}\tau(\delta(b)a-\delta(a)b)\\
&=&\frac{-1}{2\pi i}\tau(\delta(a)b).
\end{eqnarray*}
\end{proof}

\begin{proof}[\bf Proof of Theorem \ref{index}]
Recall that relative to the decomposition $1=P+(1-P)$ we have: 
$$\tilde\pi(a)=\left[\begin{array}{cc} T_a & B\\C & D\end{array}\right],$$
where
$$B=P\tilde\pi(a)(1-P)=P[P,\tilde\pi(a)]=\frac{1}{2}P[H,\tilde\pi(a)]
\in \cal M^{\hat\tau}_2,$$
and,
$$C=\frac{1}{2}[H,\tilde\pi(a)]P\in \cal M^{\hat\tau}_2.$$
By Corollary A4 of the Appendix and the previous proposition we have:
$$\hat\tau{\text -}Index(T_a)=\hat\tau([T_a,T_{a^{-1}}])=
\frac{-1}{2\pi i}\tau(\delta(a)a^{-1}).$$
This completes the proof of Theorem \ref{index}.
\end{proof}

\begin{cor}\label{IndMorph}
If $\varphi: A_1\to A_2$ defines a morphism from $(A_1,Z_1,\tau_1,\alpha^1)$ to $(A_2,Z_2,\tau_2,\alpha^2)$
and if $a\in A_1^{-1}\cap(dom(\delta_1))$ then $\varphi(a)\in A_2^{-1}\cap(dom(\delta_2))$  and
$$\hat\tau_1{\text -}Index(T_a)\in (Z_1)_{sa}\;\;{\rm while}\;\; \hat\tau_2{\text -}Index(T_{\varphi(a)})\in (Z_2)_{sa}\;\;{\rm and\; also}$$
 $$\varphi(\hat\tau_1{\text -}Index(T_a))=\hat\tau_2{\text-}Index(T_{\varphi(a)}).$$
\end{cor}
\begin{proof}
This follows immediately from Proposition \ref{wind2} and Theorem \ref{index}.
\end{proof}

\section{EXAMPLES}
{\bf 1. Kronecker (scalar trace) Example.} Recall: $A=C({\bf T}^2)$, the $C^{*}$-algebra of 
continuous functions on
the $2$-torus, with the usual scalar trace $\tau$ given by 
the Haar measure on ${\bf T}^2$, and
$\alpha :{\bf R}\to\hbox{Aut}(A)$ is the Kronecker flow on $A$ 
determined by the
real number, $\mu$. That is, for $s\in {\bf R}$, $f\in A$, and $
(z,w)\in {\bf T}^2$ we have:
$$(\alpha_s\,f)(z,w)=f\left(e^{-2\pi is}\,z,\,e^{-2\pi i\mu
s}\,w\right).$$
In this case, $Z=\frk Z = {\bf C}$ and so $A_{\frk Z}=A.$ Hence our 
$\frk Z$-Hilbert algebra $\mathcal A=C_c({\bf R},A)$ is just a Hilbert 
algebra in the ordinary sense and 
$\cal{H}_{\mathcal A}=L^2({\bf R},L^2({\bf T}^2)).$
Now, denoting $\cal{H}=L^2({\bf T}^2)$, we have that the $C^{*}$-crossed product 
$A\rtimes_{\alpha}{\bf R}$ is represented on $L^2({\bf R},\cal{H})$ by the induced
representation of Definition \ref{Indrep} as follows: for
$$s,t\in {\bf R},\quad\xi\in C_c({\bf R},A)\subseteq L^2({\bf R},\cal{H})\qquad
\hbox{and}\qquad f\in A$$
we define
\begin{eqnarray*}\left(\tilde\pi (f)\,\xi\right)(s)&=&\alpha_s^{-1}(f)\cdot\xi
(s)\qquad\hbox{and}\\
\left(\lambda_t\,\xi\right)(s)&=&\xi(s-t).
\end{eqnarray*}
Thus, $\tilde\pi\times \lambda$ is a faithful representation of 
$A\rtimes_{\alpha}{\bf R}$ on $L^2({\bf R},\cal H)$. It is
well-known that for $\mu$ irrational,
$\mathcal M=\left(\tilde\pi\times \lambda (A\rtimes_{\alpha}\,{\bf R})\right
)^{\prime\prime}$ is a $\hbox{II}_{\infty}$ factor, [CMX]. In general $\mathcal M$ is a semifinite von Neumann algebra and $\tilde\pi :A\to \mathcal M$. Now, if $\delta$ is the densely 
defined derivation on $A$ generating the representation
$\alpha :{\bf R}\to\hbox{Aut}(A)$ and we let $u\in U(A)$ be the 
function $u(z,w)=w$ then $u$ is a smooth element 
for $\delta$ and $\delta(u)=-(2\pi i\mu) u$. 
Thus by Theorem \ref{index}, the 
Toeplitz operator $T_u:=P\tilde\pi(u)P$ is Fredholm relative to the trace
$\hat\tau$ in the semifinite von Neumann algebra, $\mathcal N=P\mathcal MP$ and its index is given by:
$$\hat\tau{\text -}Index(T_u)=\frac{-1}{2\pi i}\tau(\delta(u)u^*)=\mu.$$

{\bf 2. General Kronecker Examples.} Recall $Z=C(X)$ is any commutative unital $C^{*}$-algebra 
with a faithful state $\omega$ and $\theta\in Z_{sa}$ is any self-adjoint 
element in 
$Z$. Recall $A=C({\bf T}^2,Z)=C(X)\otimes C({\bf T}^2)$, and $\tau :A\to Z$ is given by the
``slice-map'' $\tau=id_Z\otimes\varphi$ where $\varphi$ is the
trace on $C({\bf T}^2)$ given by Haar measure. That is, for $f\in A=C({\bf T}^2,Z)$ we have
$$\tau(f)=\int_{\bf{T}^2}f(z,w)d(z,w)\in Z,$$ and $\tau$ is a faithful,
tracial conditional expectation of $A$ onto $Z$. 
Recall that $\bar\omega:=\omega\circ\tau=\omega\otimes\varphi$ is a faithful (tracial) state $\bar\omega$ on $A.$  We use the element $\theta\in 
Z_{sa}$ to
define a $\tau$-invariant action $\{\alpha_t\}$ of $\real$ on $A$:
$$\alpha_t(f)(x,z,w)=f(x,e^{-2\pi it}z,e^{-2\pi i\theta(x)t}w),$$
for $f\in A$, $t\in\real$, $x\in X$, and $z,w\in \bf{T}.$

Let $(\pi,\mathcal H)$ be the GNS representation of $A$
induced by $\bar\omega$ then there is a continuous
unitary representation $\{U_t\}$ of $\real$ on $\mathcal H$ so that 
$(\pi,U)$ is covariant for $\alpha$ on $A$. Also, $\{U_t\}$ implements an
uw-continuous ``extension'' of $\alpha$ to $\bar\alpha$ acting on 
$\frk A:=\pi(A)^{\prime\prime}.$ Morover, 
letting $\frk Z:=\pi(Z)^{\prime\prime},$
there exists a unique faithful unital, uw-continuous 
$\frk Z$-trace $\bar\tau:\frk A \to \frk Z$ ``extending'' $\tau,$ 
and $\bar\alpha$ leaves $\bar\tau$ invariant. That is, in this
representation on $\mathcal H$, we have that {\bf Standing Assumptions}
are also satisfied. We simplify our notation and write $L^2(X)$, 
$L^2({\bf T}^2)$, $L^\infty(X)$, and  $L^\infty({\bf T}^2)$ for $L^2(X,\omega)$,
$L^2({\bf T}^2,\varphi)$, $L^\infty(X,\omega)$, and $L^\infty({\bf T}^2,\varphi)$,
respectively.

Then, in this representation, one easily verifies that:
\begin{eqnarray*}
\mathcal H &=& L^2(X)\otimes L^2({\bf T}^2),
\;as\;Hilbert\;spaces,\;and\\
\frk Z &=& L^{\infty}(X)\otimes 1,\; and\\
A_{\frk Z} &=& L^{\infty}(X)\otimes C({\bf T}^2)
\;as\;C^*-algebras,\;and\\
\frk A &=& L^{\infty}(X)\bar{\otimes} L^{\infty}({\bf T}^2)\;as\;
von\;Neumann\;algebras.
\end{eqnarray*}

Identifying $\frk Z=L^\infty(X)$, our $L^\infty(X)$-Hilbert algebra is
$\mathcal A = C_c(\real,L^{\infty}(X)\otimes C({\bf T}^2))$ with 
the $\bar\alpha$-twisted convolution multiplication and $L^\infty(X)$-valued
inner product for $f,g\in \mathcal A$ given by:
\begin{eqnarray*}\hat{\tau}(Ind(f)^*Ind(g))&=&\bra f,g\ket=\int_{\real}\bar\tau((f(t))^*g(t))dt\\
&=&\int_{\real}\left(\int_{\bf{T}^2}(f(t)[(z,w)])^*g(t)[(z,w)]d(z,w)\right)dt.
\end{eqnarray*}
Now, consider the following unitary $v$ in $A$: $v(x,z,w)=w.$ Then
$$\alpha_t(v)(x,z,w)=e^{-2\pi i\theta(x)t}w\;\;{\rm and\;\;so}\;\; \delta(v)(x,z,w)=-2\pi i\theta(x)w.$$
Hence, $(\delta(v)v^*)(x,z,w) = -2\pi i\cdot\theta(x).$ Since the trace $\tau$ on $A$ is just the slice map $id_Z\otimes\varphi$ we see that $\tau(\delta(v)v^*)=-2\pi i\cdot\theta.$ Hence, by Theorem \ref{index}, the Toeplitz operator
$T_v$ is Fredholm relative to the trace $\hat\tau$ on 
$\cal N=P(Ind(A\rtimes\real)^{\prime\prime})P$, and
$$\hat\tau{\text -}Index(T_v)=\frac{-1}{2\pi i}\tau(\delta(v)v^{*})=\theta\in C(X)=Z\hookrightarrow
Z\otimes C({\bf{T}}^2)=A.$$

{\bf 3. Fiberings of Toeplitz operators}.  Recall that for any fixed $x\in X$ (where $Z=C(X)$) the evaluation map at $x$ yields a homomorphism from $A=Z\otimes C({\bf{T}}^2)$ to $C({\bf{T}}^2)$ which defines a morphism from Example 2 to Example 1 which carries $\theta$ to $\mu:=\theta(x)$.
Moreover this morphism carries $v$ to $u=v(x)$. So that $Index(T_u)=\mu=\theta(x)=(Index(T_v))(x)$. That is, the Toeplitz operator $T_v$ fibers over $X$ as the Toeplitz operators $T_{\theta(x)}$ and moreover for each $x\in X$:
$$Index(T_{v(x)})=(Index(T_v))(x). $$
so the Index fibers accordingly.

Similarly, for any fixed $x\in X$ (where $Z=C(X)$) the evaluation map at $x$ yields a homomorphism from $A=Z\otimes A_\theta$ to $A_\theta$ which defines a morphism from \\$(Z\otimes A_\theta,Z, id\otimes \tau_\theta,\alpha^\eta)$ to $(A_\theta,{\bf C},\tau_\theta,\alpha^{\eta(x)}).$ This morphism carries $1\otimes V$ to $V$. Since $Index(T_{1\otimes V})=\eta$ and $Index(T_V)=\eta(x)$ we see that:
$$Index(T_{1\otimes V})(x) = Index(T_V)= Index(T_{1\otimes V}(x)). $$

{\bf 4. $C^*$-algebra of the Integer Heisenberg group.} Recall that $A=C^*(H)$ is the $C^*$-algebra of the Integer Heisenberg group viewed as the universal $C^*$-algebra generated by three unitaries $U, V, W$
satisfying:
$$WU=UW,\;\;\;WV=VW,\;\;\;and\;\;\;UV=WVU.$$
In this case $Z=C^*(W)$ is the centre of $A$ and also equals $C^*(C)$ the $C^*$-algebra generated by 
$C=\bra W\ket$ the centre of $H.$ The trace $\tau:C^*(H)\to C^*(C)$ on functions in $l^1(H)\subset C^*(H)$ 
is just given by restriction to $C$. Our Hilbert space $\cal{H}=l^2(H)$ acted on by the left regular representation of $C^*(H).$ The restriction of this action to $Z=C^*(C)$ on 
$$l^2(H)=\bigoplus_{(n,m)\in{\bf Z^2}} l^2(C\cdot(V^nU^m))$$
is unitarily equivalent to $1_{\bf Z^2}\otimes \pi_C(C)$ on 
$\bigoplus_{(n,m)\in{\bf Z^2}} l^2(C).$ In this labelling of the cosets, multiplication by $W$ acts the same on each coset: it increases the power of $W$ by one. Multiplication by $V$ acts as the identification of
$l^2(C\cdot(V^nU^m))$ with $l^2(C\cdot(V^{n+1}U^m))$: that is, it acts as a permutation of the copies of $l^2(C)$ while acting on the basis elements as the identity on $l^2(C)$. However, multiplication by $U$ not only maps $l^2(C\cdot(V^nU^m))$ to $l^2(C\cdot(V^nU^{m+1}))$, but it also acts on the basis elements of
$l^2(C)$ by sending $W^k$ to $W^{k+1}$. In this representation we recall that the map $\tau :C^*(H)\to C^*(C)$
is given by $\tau(x)=1_{{\bf Z^2}}\otimes ExE,$ where $E$ is the projection of $l^2(H)$ onto $l^2(C).$
Thus we have an action $\alpha:{\bf R}\to Aut(A),$ that fixes $Z=C^*(W)$ and leaves the $Z$-valued trace 
$\tau$ invariant. A short calculation using Theorem \ref{index} then gives us the nontrivial index:
$$\hat\tau{\text -}Index(T_{V^nU^mW^p})=(n\theta+m)\in Z=C^*(W).$$

\appendix{APPENDIX: FREDHOLM THEORY RELATIVE to a $\frk Z$-VALUED TRACE 
on a von NEUMANN ALGEBRA}

We let $\cal N$ denote a semifinite von Neumann algebra and let $\frk Z$
denote a unital von Neumann subalgebra of the centre of $\cal N$. We suppose
that we have a faithful, normal, semifinite $\frk Z$-trace $\phi$ defined on
$\cal N_+$ as in Definition 6.1. We will show that using $\phi$ as a dimension
function we can adapt M. Breuer's arguments in \cite{Br1}, \cite{Br2} to obtain
a Fredholm theory involving a $\frk Z$-valued index with the usual algebraic 
and topological stability properties, and in which the role of the compact 
operators is replaced by the norm-closed ideal, $\cal K_{\cal N}^{\phi}$
generated by the projections of $\phi$-finite trace.\\
\indent A projection $E$ in $\cal N$ will be called $\phi {\text -}finite$ if
$\phi(E)\in \frk Z_+$. Since $\phi$ is faithful, it is clear that any 
$\phi$-finite projection is also finite in the Murray-von Neumann sense.
An operator $T\in\cal N$ is called $\phi {\text -}Fredholm$ if the 
projection $N_T$
on $ker(T)$ is $\phi$-finite and there is a $\phi$-finite projection $E$
with $range(1-E)\subseteq range(T).$ Since $\phi$-finite projections are
finite, every $\phi$-Fredholm operator is Fredholm in Breuer's sense. If
$T$ is $\phi$-Fredholm, the $\phi$-$index$ of $T$ is by definition
$$\phi{\text -}Index(T):=\phi(N_T)-\phi(N_{T^*}):$$
we shall see below that $T^*$ is also $\phi$-Fredholm so that 
$\phi$-$Index(T)$ is a well-defined self-adjoint element of $\frk Z$.

We observe as we did in \cite{PhR} that the ideal $\cal K_{\cal N}^{\phi}$
can also be described as the {\it closure} of any of:\\
\noindent (1) the span of the $\phi$-finite projections in $\cal N$,\\
\noindent (2) the span of the $\phi$-finite elements in $\cal N$,\\
\noindent (3) the algebra of elements $T\in\cal N$ whose range projection
$R_T$ is $\phi$-finite.\\
This ideal is clearly contained in Breuer's ideal $\cal K$ generated by all
the finite projections in $\cal N.$

Now the further remarks and proofs concerning how Breuer's arguments carry over 
to this situation follow verbatim from Appendix B of \cite{PhR}. So, we obtain
the analogues of Breuer's theorems exactly as we did in \cite{PhR}.

\begin{thm*}[\bf A1]
Let $\phi$ be a faithful, normal, semifinite $\frk Z$-trace on the von
Neumann algebra $\cal N$, and let $\cal K^{\phi}_{\cal N}$ be the norm-closed
ideal in $\cal N$ generated by the $\phi$-finite projections.

(1) (The Fredholm alternative) If $T\in\cal K^{\phi}_{\cal N}$, then $(1-T)$
is $\phi$-Fredholm and $$\phi{\text -}Index(1-T) = 0.$$

(2) (Atkinson's Theorem) An operator $T\in \cal N$ is $\phi$-Fredholm if and
only if $T+\cal K^{\phi}_{\cal N}$ is invertible in 
$\cal N/\cal K^{\phi}_{\cal N}$.

(3) If $S$ and $T$ are $\phi$-Fredholm, then so are $S^*$ and $ST$, and
$$\phi{\text -}Index(S^*)=-(\phi{\text -}Index(S)),\hspace{.5in} 
\phi{\text -}Index(ST)=\phi{\text -}Index(S)+\phi{\text -}Index(T).$$
\end{thm*}

The following corollary is proved exactly as Corollary B2 of \cite{PhR}.

\begin{cor*}[\bf A2] The set $\cal F_{\phi}(\cal N)$ of 
$\phi$-Fredholm
operators is open in the norm topology of $\cal N$, and the index map
$T\mapsto \phi$-$Index(T)$ is locally constant on $\cal F_{\phi}(\cal N)$.
\end{cor*}

The following trace formula for the index goes back to Calder\'{o}n for
pseudodifferential operators. The general type $I$ case is due to
H\"{o}rmander \cite{H} but Connes also has an elegant proof \cite{Co}.
One of the authors generalised H\"{o}rmander's proof to the case of a 
factor of type $II_{\infty}$ in \cite{Ph}, Theorem A7. It is this latter
proof that goes through essentially verbatim to our present setting, so we refer
the reader to Appendix A of \cite{Ph} for the proof.

\begin{thm*}[\bf A3]
Let $\phi$ be a faithful, normal, semifinite $\frk Z$-trace on the von
Neumann algebra $\cal N$, and let $S,T\in\cal N$ so that
$R_1=1-ST$ and $R_2=1-TS$ are both $n$-summable for some integer $n>0$.
Then, $T$ is a $\phi$-Fredholm operator and 
$$\phi{\text -}Index(T)=\phi(R_1^n)-\phi(R_2^n).$$
\end{thm*}

\begin{cor*}[\bf A4]
Let $A$ be a unital $C^*$-algebra and let $Z\subseteq Z(A)$ be a unital
$C^*$-subalgebra of the centre of $A.$ Let $\tau:A \to Z$ be a faithful,
unital $Z$-trace which is invariant under a continuous
action $\alpha$ of $\real$. Then for any $a\in A^{-1}\cap dom(\delta)$, 
the Toeplitz operator
$T_a$ is Fredholm relative to the trace $\hat\tau$ on 
$\cal N=P(Ind(A\rtimes\real)^{\prime\prime})P$, and
$$\hat\tau{\text -}Index(T_a)=\hat\tau([T_a,T_{a^{-1}}]).$$
\end{cor*}

\begin{proof}[\bf Proof]
We let $T=T_a$ and $S=T_{a^{-1}}$ and $\phi=\hat\tau$ in the statement 
of the previous theorem.
Then, $R_1=1-T_{a^{-1}}T_a=T_{a^{-1}a}-T_{a^{-1}}T_a\in \cal N^{\hat\tau}$
by Corollary 9.4 and similarly, $R_2\in \cal N^{\hat\tau}.$
Then, by the previous theorem, $T_a$ is $\hat\tau$-Fredholm and 
$$\hat\tau{\text -}Index(T_a)=\hat\tau(R_1)-\hat\tau(R_2)=
\hat\tau([T_a,T_{a^{-1}}]).$$
\end{proof}


\begin{thebibliography}{Arv}

\bibitem[AP]{AP} J. Anderson and W. Paschke, \emph{The Rotation Algebra}, Houston J. Math. {\bf 15}  (1989), 1-26.

\bibitem[Arv]{Arv} W. Arveson, \emph{Subalgebras of $C^*$-algebras}, Acta Math.,
{\bf 123} (1969), 141-224.

\bibitem[Br1]{Br1} M. Breuer, \emph{Fredholm Theories in von Neumann Algebras,
I}, Math. Ann., {\bf 178} (1968), 243-254.

\bibitem[Br2]{Br2} M. Breuer, \emph{Fredholm Theories in von Neumann Algebras,
II}, Math. Ann., {\bf 180} (1969), 313-325.

\bibitem[Co]{Co} A. Connes, \emph{Noncommutative Differential Geometry},
Publ. Math. Inst. Hautes Etudes Sci., {\bf 62} (1985), 41-144.

\bibitem[CMX]{CMX} R. Curto, P. S. Muhly, and J. Xia, \emph{Toeplitz operators 
on flows}, J. Functional Analysis, {\bf 93} (1990), 391-450.

\bibitem[Dix]{Dix} J. Dixmier, Les alg\`ebres d'op\'erateurs dans l'espace
Hilbertien (Alg\`ebres de von Neumann), Gauthier-Villars, Paris, 1969.

\bibitem[DM]{DM} H. Dym and H.P. McKean, Fourier Series and Integrals,
Academic Press, New York, London, 1972.

\bibitem[H]{H} L. H\"{o}rmander, \emph{The Weyl Calculus of Pseudodifferential
Operators}, Comm. Pure Appl. Math., {\bf 32} (1979), 359-443.

\bibitem[Ji]{Ji} R. Ji,\emph {Toeplitz Operators on Noncommutative Tori and Their
Real-valued Index}, Proc. Symp. Pure Math. (Amer. Math. Soc.), vol. 51, Part 2
(1990), pages~153-158.

\bibitem[K]{K} I. Kaplansky, \emph{Modules over operator algebras},
Amer. J. Math., {\bf 75} (1953), 839-858.

\bibitem[L]{L} E. C. Lance, Hilbert $C^*$-modules, London Math. Soc. Lecture 
Notes Series 210, Cambridge University Press, Cambridge, 1995.

\bibitem[Le]{Le} M. Lesch, \emph{On the Index of the Infinitesimal Generator of a Flow},
J. Operator Theory, {\bf 26} (1991), 73-92.

\bibitem[PR]{PR} J.A. Packer and Iain Raeburn, \emph{On the Structure of Twisted Group $C^*$-algebras}, Trans. Amer. Math. Soc. {\bf 334} (1992), 685-718.

\bibitem[Pa]{Pa} W. Paschke, \emph{Inner Product Modules Over $B^*$-algebras},
Trans. Amer. Math. Soc., {\bf 182} (1973), 443-468.

\bibitem[Ped]{Ped} G.K. Pedersen, $C^*$-Algebras and their Automorphism
Groups, Academic Press, London, 1979.

\bibitem[Ph]{Ph} John Phillips, \emph{Spectral Flow in Type I and II Factors--
A New Approach}, Fields Inst. Comm., {\bf 17} (1997), 137-153.

\bibitem[PhR]{PhR} John Phillips and Iain Raeburn, \emph{An Index Theorem for
Toeplitz Operators with Noncommutative Symbol Space}, J. Functional Analysis,
{\bf 120} (1994), 239-263.

\bibitem[R]{R} M. Rieffel, \emph{Morita Equivalence for $C^*$-algebras
and $W^*$-algebras}, J. Pure and Applied Algebra, {\bf 5} (1974), 51-96. 

\bibitem[T]{T} J. Tomiyama, \emph{On the projection of norm one in
$W^*$-algebras}, Proc. Japan Acad., {\bf 33} (1957), 608-612.

\bibitem[U]{U} H. Umegaki, \emph{Conditional expectation in an operator
algebra I}, Toh\^{o}ku Math. J., {\bf 6} (1954), 358-362.
\end{thebibliography}
\end{document}